\documentclass[oneside]{article}

\usepackage[english]{babel}
\usepackage[utf8x]{inputenc}
\usepackage[T1]{fontenc}

\usepackage{amsmath, amstext, amsthm, mathrsfs,dsfont,amssymb}

\usepackage{tikz}
\usetikzlibrary{patterns}

\usepackage{graphicx}
\usepackage{caption}
\usepackage{subcaption}

\usepackage[all]{xy}

\usepackage{hyperref}
\hypersetup{colorlinks = true, urlcolor = blue, citecolor=blue, bookmarksopen = true}

\numberwithin{equation}{section}

\definecolor{dartmouthgreen}{rgb}{0.05, 0.5, 0.06}

\theoremstyle{plain}
\newtheorem{theorem}{Theorem}
\newtheorem{hypo}{Hypothesis}
\newtheorem{prop}{Proposition}[section]
\newtheorem{lemma}{Lemma}[section]
\newtheorem*{lemma*}{Lemma}

\newtheorem*{corollary*}{Corollary}
\newtheorem{definition}{Definition}

\theoremstyle{remark}

\newtheorem{remark}{Remark}
\renewenvironment{proof}{\vspace{5mm}\noindent\textbf{Proof }}{\hspace*{\fill}$\Box$\medskip\vspace{5mm}}

\def \ind {\mathds{1}}
\def \ppp {\ldots}
\def \lc {\left\lbrace}
\def \rc {\right\rbrace}

\def \C {\mathbb{C}}
\def \D {\mathbb{D}}

\def \M {\mathbb{M}}
\def \N {\mathbb{N}}

\def \R {\mathbb{R}}
\def \S {\mathbb{S}}

\def \Z {\mathbb{Z}}

\def \Bc {\mathcal{B}}
\def \Cc {\mathcal{C}}

\def \Fc {\mathcal{F}}
\def \Gc {\mathcal{G}}
\def \Hc {\mathcal{H}}

\def \Lc {\mathcal{L}}
\def \Mc {\mathcal{M}}
\def \Oc {\mathcal{O}}

\def \Rc {\mathcal{R}}

\def \Tc {\mathcal{T}}
\def \Uc {\mathcal{U}}

\def \Lcc {\mathscr{L}}
\def \Pcc {\mathscr{P}}
\def \Rcc {\mathscr{R}}
\def \Tcc {\mathscr{T}}

\def \Pg {\textbf{P}}
\def \Rg {\textbf{R}}

\begin{document}
	\begin{center}
		{\large \bfseries Tamed stability of finite difference schemes for the transport equation on the half-line}
	\end{center}
	
	\begin{center}
		Lucas \textsc{Coeuret}\footnote{Institut de Mathématiques de Toulouse ; UMR5219 ; Université de Toulouse ; CNRS ; UPS, 118 route de Narbonne, F-31062 Toulouse Cedex 9, France. Research of the author was supported by the Agence Nationale de la Recherche project Indyana (ANR-21-CE40-0008-01), as well as by the Labex Centre International de Mathématiques et Informatique de Toulouse under grant agreement ANR-11-LABX-0040 and Artificial and Natural Intelligence Toulouse Institute  under grant agreement  ANR-19-P3IA-0004. E-mail: lucas.coeuret@math-univ.toulouse.fr}
	\end{center}
		
	\vspace{5mm}
	
	\begin{center}
		\textbf{Abstract}
	\end{center}
	
	In this paper, we prove that, under precise spectral assumptions, some finite difference approximations of scalar leftgoing transport equations on the positive half-line with numerical boundary conditions are $\ell^1$-stable but $\ell^q$-unstable for any $q>1$. The proof relies on the accurate description of the Green's function for a particular family of finite rank perturbations of Toeplitz operators whose essential spectrum belongs to the closed unit disk and with a simple eigenvalue of modulus $1$ embedded into the essential spectrum.
	
	\vspace{4mm}
	
	
	
	\vspace{4mm}
	
	Throughout this article, we define the following sets:
	$$\Uc:=\lc z\in \C, |z|>1\rc,\quad \D:=\lc z\in \C, |z|<1\rc, \quad \S^1:=\lc z\in \C, |z|=1\rc,$$
	$$\overline{\Uc}:=\S^1\cup \Uc,\quad  \overline{\D}:=\S^1\cup \D.$$
	
	For $z\in \C$ and $r>0$, we let $B_r(z)$ denote the open ball in $\C$ centered at $z$ with radius $r$.
	
	For $E$ a Banach space, we denote $\Lc(E)$ the space of bounded operators acting on $E$ and $\left\|\cdot\right\|_{\Lc(E)}$ the operator norm. For $T$ in $\Lc(E)$, the notations $\sigma(T)$ and $\rho(T)$ stand respectively for the spectrum and the resolvent set of the operator $T$.
	
	We let $\Mc_{n,k}(\C)$ denote the space of complex valued $n\times k$ matrices and we use the notation $\Mc_n(\C)$ when $n=k$.
	
	We use the notation $\lesssim$ to express an inequality up to a multiplicative constant. Eventually, we let $C$ (resp. $c$) denote some large (resp. small) positive constants that may vary throughout the text (sometimes within the same line).
	
	\section{Introduction}
	
	\subsection{Context}
	
	The purpose of this article is to study the so-called semigroup stability for discretizations of hyperbolic initial boundary value problems. More precisely, we focus our attention on explicit finite difference schemes that are consistent with the scalar leftgoing ($v<0$) transport equation on the positive half-line with a Cauchy initial datum
	\begin{align}\label{def:PDE}
		\begin{split}
			\forall t\geq0,\forall x\geq 0,\quad&\partial_tu+v\partial_x u =0,\\
			\forall x\geq 0,\quad &u(0,x) = u_0(x)\in\R.
		\end{split}
	\end{align}
	No boundary condition is required here at $x=0$ for \eqref{def:PDE} since the transport operator is outgoing with respect to the boundary. However, the numerical schemes we consider will require the introduction of nonphysical numerical boundary conditions which can generate instabilities. One of the cornerstone in numerical analysis is the so-called Lax equivalence Theorem \cite{Lax} which claims that a consistent scheme is convergent if and only if it is stable. Thus, finding a reliable way to ensure the stability or instability of a numerical scheme is central. 
	
	When considering discretizations of initial boundary value problems, there are several possible definitions for stability that have been introduced. The one interesting us in the present paper is the semigroup stability in the $\ell^q$-topology (see Definition \ref{definition:Stable} of $\ell^q$-stability) which can be presented as the following power boundedness property
	\begin{equation}\label{intro:powerbound}
		\sup_{n\in \N}\left\|\Tc^n\right\|_{\Lc(\Hc_q)}<+\infty
	\end{equation}
	where $\Tc$ is the discrete evolution operator that allows to compute the solution of the numerical scheme from one time step to the next and the vector space $\Hc_q$ is a modification of the vector space $\ell^q(\N)$ which takes into account the numerical boundary conditions of the numerical scheme. A direct conclusion is that the $\ell^q$-stability prevents the existence of eigenvalues of the operator $\Tc$ in the set $\Uc$ (unstable eigenvalues). This corresponds to the so-called Godunov-Ryabenkii condition introduced in \cite{GodRyab}. Let us point out that the existence and position of eigenvalues for the operator $\Tc$ highly depends on the choice of numerical boundary condition that is done, as will be explained in the article.
	
	One of the other most notable definition of stability is the notion of \textit{strong stability} (also known as GKS-stability) introduced in the fundamental contribution \cite{GKS}. It can be considered to be one of the most robust definitions of stability in the context of finite difference schemes for initial boundary value problems as it is stable with respect to perturbations. We refer the interested reader to \cite{GKO} for a complete overview of GKS-theory as it will not be the main focus of this paper. However, we need to point out that the strong stability of a finite difference scheme is fully characterized by the fulfillment of the so-called uniform Kreiss-Lopatinskii condition which in our case corresponds to the operator $\Tc$ not having eigenvalues or generalized eigenvalues in the set $\overline{\Uc}$. Even though the two notions of stability we introduced are quite different, it can be shown using energy estimates that the strong stability, i.e. the verification of the uniform Kreiss-Lopatinskii condition, implies the $\ell^2$-stability (see \cite{Wu,CouGlo,Cou13} which respectively tackle the cases of the scalar one-dimensional problem and one time step scheme, of the multidimensional system and one time step scheme and of the scalar multidimensional problem and multistep scheme) and even the $\ell^q$-stability for all $q\in[1,+\infty]$ (using the semigroup estimates in \cite{CF2}).
	
	However, it remains uncertain to conclude on the $\ell^q$-stability of a numerical scheme whenever the Godunov-Ryabenkii condition is satisfied but not the uniform Kreiss-Lopatinskii condition. Up to our knowledge, this question was first formalized and tackled in \cite{Trefethen,KreissWu}. In the recent paper \cite{CF2}, which highly influenced the study carried in the present article, the authors proved that when the operator $\Tc$ admits simple eigenvalues on the unit circle that do no belong to the essential spectrum of $\Tc$, the numerical scheme remains $\ell^q$-stable for all $q\in[1,+\infty]$. The goal of the present paper is to carry the same kind of analysis when the operator $\Tc$ admits a simple eigenvalue of modulus $1$ that lies within the essential spectrum of $\Tc$, which is up to our knowledge a spectral configuration that was not handled before. Let us point out that studying this spectral assumption has two further incentives: 
	
	- First, the analysis carried under this type of spectral assumption will have to be carefully dealt with since the spatial Green's function, a tool introduced in Section \ref{sec:GS}, will only be meromorphically extended near some interest point and not holomorphically as in \cite{CF2}. This will motivate the introduction of careful computations that could be reused in other studies with similar spectral configuration. Similar computations have already been presented in the continuous setting when studying viscous shock profiles  (see \cite{MasciaZumbrun}) but this is an occasion to extend them to the fully discrete setting.
	
	- Second, this type of spectral configuration also occurs in the context of the study of stability of discrete shock profiles (see \cite{Serre}). The author hopes that the analysis carried in this paper could be used to extend results on this subject, for instance the result of \cite{Godillon} (see Remark \ref{rem:Cons} for more details).
	
	In direct opposition with the main result of \cite{CF2}, the main result of the present paper (Theorem \ref{th:Stab}) states that if the operator $\Tc$ has a simple eigenvalue of modulus $1$ that is not isolated from the essential spectrum, then the numerical scheme remains $\ell^1$-stable, but, under some explicit algebraic condition, the numerical scheme is also either $\ell^q$-stable for all $q\in]1,+\infty]$ or $\ell^q$-unstable for all $q\in]1,+\infty]$. In the later case, we even prove a sharp growth rate \eqref{rateofexploTqn} of the norm $\left\|\Tc^n\right\|_{\Lc(\Hc_q)}$ depending on $n$.
	
	The proof of Theorem \ref{th:Stab} relies on a precise description of the asymptotic behavior of the semigroup associated with the numerical scheme (see Theorem \ref{th:Green}) using an approach referred to as "spatial dynamics". It follows a series of papers initiated by \cite{ZH} and aims at using functional calculus (see \cite[Chapter VII]{Conway}) to express the temporal Green's function (the fundamental solution of the numerical scheme, defined below by \eqref{def:GTToep}) using the resolvent of the operator $\Tc$ via the so-called spatial Green's function (defined in Section \ref{sec:GS}). The general structure of the article is kept quite similar with the one of \cite{CF2}. On the one hand, it allows the interested reader to observe the similarities between both papers that could lead to a more general result combining both the result of the present paper and that of \cite{CF2}. On the other hand, the author thinks that the fundamental differences on the choice of spectral setup will be clearer this way as every step of the proof can be compared with its equivalent in \cite{CF2}.

	\subsection{Setup}
	
	We seek to approach the real valued solution $u$ of the Cauchy problem \eqref{def:PDE}. We introduce a time step $\Delta t>0$ as well as a space step $\Delta x>0$ and from now on we consider that the Courant number $\lambda := \frac{\Delta t}{\Delta x}$ is always kept fixed. To approach the value of the solution $u$ of the Cauchy problem \eqref{def:PDE} in the cell $[n\Delta t,(n+1)\Delta t[\times[(j-1)\Delta x,j\Delta x[$, we define an explicit one-step in time finite difference scheme applied to \eqref{def:PDE}. In the interior cells $[(j-1)\Delta x,j\Delta x[$ for $j\in\N\backslash\lc0\rc$, the finite difference scheme is defined by
	\begin{subequations}\label{def:numScheme}
		\begin{equation}
			\forall n\in \N, \forall j\in \N\backslash\lc0\rc, \quad u^{n+1}_j = \sum_{k=-r}^p a_k u^n_{j+k}.\label{def:numSchemeMain}
		\end{equation}
		where $r,p$ are non negative integers, the coefficients $a_k$ are given real numbers which can depend on the velocity $v$ and on the Courant number $\lambda$. There remains to define how the values $u^n_{1-r},\ppp,u^n_0$ in the so-called "ghost cells" are dealt with. As is usual, the numerical boundary conditions to compute the values of $u^n_{1-r},\ppp,u^n_0$ are given by a linear combination of the first values close to the boundary:
		\begin{equation}
			\forall n\in \N,\forall j\in \lc1-r,\ppp,0\rc, \quad u^n_j = \sum_{k=1}^{p_b} b_{k,j}u^n_k,\label{def:numSchemeCB}
		\end{equation}
	\end{subequations}
	where $p_b$ is a non negative integer, the coefficients $b_{k,j}$ are also given real numbers which can depend on the velocity $v$ and on the Courant number $\lambda$.  The initial values $u^0_j$ for $j\in \N\backslash\lc0\rc$ are computed using the Cauchy initial datum $u_0$ of the PDE \eqref{def:PDE}, for instance as
	$$\forall j\in\N\backslash\lc0\rc, \quad u^0_j:= \frac{1}{\Delta x}\int_{(j-1)\Delta x}^{j\Delta x} u_0(x)dx.$$
	We will assume that the integers $r$ and $p$ are fixed and that $a_{-r},a_p\neq0$. We claim that the case $r=0$ is special as there would be no numerical boundary condition to implement in the numerical scheme \eqref{def:numScheme} since the discretization of the transport equation would be of the "upwind" type. From now on, we will assume that $r\geq1$ while keeping the normalization condition $a_{-r}\neq0$. We will also assume that the integer $p_b$ in \eqref{def:numSchemeCB} satisfies $p_b\leq p$. This condition on $p_b$, which was also made in \cite{CF2}, will ease the computations in the paper as it allows us to translate the boundary condition of the numerical scheme \eqref{def:numSchemeCB} as
	$$\forall n\in \N,\quad \Bc \begin{pmatrix}
		u^n_p\\ \vdots \\ u^n_{1-r}
	\end{pmatrix}=0$$
	where 
	\begin{equation}\label{def:B}
		\Bc:= \left(\begin{array}{ccccccccccc}
			0&\hdots & 0 & -b_{p_b,0} & \hdots & -b_{1,0} & 1 &0 & \hdots&\hdots &0 \\
			 \vdots &\vdots & \vdots &\vdots & \vdots & \vdots & 0 & 1 &\ddots & & \vdots \\
			 \vdots &\vdots & \vdots &\vdots & \vdots & \vdots &\vdots & \ddots & \ddots & \ddots & \vdots \\
			 \vdots &\vdots & \vdots &\vdots & \vdots & \vdots &\vdots & & \ddots & 1& 0\\
			0 &\hdots & 0 &-b_{p_b,1-r} & \hdots & -b_{1,1-r} & 0 &  \hdots&\hdots & 0& 1
		\end{array}\right)\in \Mc_{r,p+r}(\R).
	\end{equation}
	This matrix form of the numerical boundary conditions \eqref{def:numSchemeCB} will often appear in this article and the matrix $\Bc$ will play a major role in our stability analysis. 
	
	We introduce for $q\in[1,+\infty]$ the Banach space
	$$\Hc_q :=\lc (w_j)_{j\geq 1-r}\in \ell^q\left(\lc j\in\Z,j\geq 1-r\rc,\R\right),\quad \forall j\in\lc1-r,\ppp,0\rc, w_j=\sum_{k=1}^{p_b}b_{k,j}w_k\rc$$
	with the norm
	$$\forall w\in\Hc_q,\quad \left\|w\right\|_{\Hc_q}:=\left\|(w_j)_{j\in\N\backslash\lc0\rc}\right\|_{\ell^q(\N\backslash\lc0\rc)}.$$
	We define the bounded operator $\Tc\in \Lc\left(\Hc_q\right)$ defined by
	\begin{equation}\label{def:Toep}
		\forall w\in\Hc_q,\forall j\in \N\backslash\lc0\rc,\quad (\Tc w)_j:=\sum_{k=-r}^pa_kw_{j+k}.
	\end{equation}
	The values $(\Tc w)_j$ for $j\in\lc1-r,\ppp,0\rc$ are determined by the condition $\Tc w\in\Hc_q$. Using the same reasoning, we also allow ourselves occasionally to say that some sequence $(u_j)_{j\geq 1}$ belongs to $\Hc_q$ without making the values for $j\in\lc1-r,\ppp,0\rc$ precise in order to alleviate the redaction.
	
	The definition of the operator $\Tc$ does not depend on $q$ but the Banach space $\Hc_q$ on which it acts does. We observe that for an initial condition $u^0\in \Hc_q$, the numerical scheme \eqref{def:numScheme} can be rewritten as the following discrete evolution problem
	\begin{equation}\label{def:numSchemeRew}
		\forall n\in \N, \quad u^{n+1}=\Tc u^n.
	\end{equation}
	We thus introduce the standard terminology:
	
	\begin{definition}[$\ell^q$-stability]\label{definition:Stable}
		The numerical scheme \eqref{def:numScheme} is said to be $\ell^q$-stable if there exists some positive constant $C>0$ such that for all $u^0\in \Hc_q$, the solution $(u^n)_{n\in\N}$ of the scheme \eqref{def:numScheme} computed using the initial condition $u^0$ satisfies
		$$\sup_{n\in\N}\left\|u^n\right\|_{\Hc_q}\leq C\left\|u^0\right\|_{\Hc_q}.$$
		This is equivalent to proving that the family of operators $(\Tc^n)_{n\in\N}$ is bounded in $\Lc\left(\Hc_q\right)$.
	\end{definition}

	The purpose of this article is to demonstrate that under a specific type of spectral condition which corresponds to the operator $\Tc$ having a simple eigenvalue of modulus $1$ that is located in its essential spectrum, the numerical scheme \eqref{def:numScheme} is $\ell^1$-stable but $\ell^q$-unstable for every $q>1$. This is in sharp contrast with the result of \cite{CF2} where it is proved that the existence of simple eigenvalues of the operator $\Tc$ on the unit circle outside the essential spectrum maintains the $\ell^q$-stability of the numerical scheme \eqref{def:numScheme}.

	\subsection{Hypotheses and main result}
	
	We will introduce a few objects that will allow us to present the main hypotheses we will make in this article.
	
	\subsubsection{Consistency, dissipativity and diffusivity conditions}
	
	Since the stencil of the numerical scheme \eqref{def:numScheme} is finite, if we consider an initial condition $u^0$ with a support that is located far from the boundary, then the numerical boundary condition \eqref{def:numSchemeCB} will not have any effect on the computation of $u^n$ for small times $n$. We then deduce that the solutions of the numerical scheme \eqref{def:numScheme} are closely linked to the solutions $(u^n_j)_{n\in\N, j\in\Z}$ of the following system on the whole one dimensional lattice $\Z$:
	\begin{equation}\label{def:numSchemeZ}
		\forall n\in \N, \forall j\in \Z, \quad u^{n+1}_j = \sum_{k=-r}^p a_k u^n_{j+k}.
	\end{equation}
	The system \eqref{def:numSchemeZ} corresponds to a numerical scheme for the transport equation on the whole line $\R$
	\begin{align*}
		\begin{split}
			\forall t\geq0,\forall x\in\R,\quad&\partial_tu+v\partial_x u =0,\\
			\forall x\in\R,\quad &u(0,x) = u_0(x).
		\end{split}
	\end{align*}
	The stability of schemes of the form \eqref{def:numSchemeZ} has been studied thoroughly in \cite{Thomee,D-S,CF,Coeuret}. We introduce the Laurent operator $\Lcc\in \Lc(\ell^q(\Z))$ defined by
	\begin{equation}\label{def:Laur}
		\forall w\in\ell^q(\Z),\forall j\in \Z,\quad (\Lcc w)_j:=\sum_{k=-r}^pa_kw_{j+k}
	\end{equation}
	which allows us to rewrite the numerical scheme \eqref{def:numSchemeZ} as a discrete evolution problem
	$$\forall n \in\N, \quad u^{n+1}=\Lcc u^n$$
	just as the operator $\Tc$ allowed us to rewrite the numerical scheme \eqref{def:numScheme} as \eqref{def:numSchemeRew}.
	
	We introduce the symbol $F$ associated with the scheme:
	\begin{equation}\label{def:F}
		\forall \kappa \in \C\backslash\lc0\rc, \quad F(\kappa):= \sum_{j=-r}^p a_j \kappa^j.
	\end{equation}
	We now make the following assumption on the numerical scheme \eqref{def:numScheme} that we consider to discretize the transport equation.
	\begin{hypo}\label{H:scheme}
		We assume that 
		$$F(1)=1, \quad \alpha:=-F^\prime(1)=\lambda v.  \quad \text{(Consistency condition)}$$
		This implies that $\alpha<0$. Moreover, we suppose that
		$$\forall \kappa\in\S^1\backslash\lc1\rc, \quad |F(\kappa)|<1  \quad \text{(Dissipativity condition)}$$
		and that there exist an integer $\mu\in\N\backslash\lc0\rc$ and a complex number $\beta\in \C$ with $\Re(\beta)>0$ such that
		\begin{equation}\label{eq:devAsympF}
			F(e^{it})\underset{t\rightarrow 0}= \exp(-i\alpha t -\beta t^{2\mu}+o(t^{2\mu})). \quad \text{(Diffusivity condition)}
		\end{equation}
	\end{hypo}
	
	In \cite{Thomee}, Hypothesis \ref{H:scheme} and especially the asymptotic expansion \eqref{eq:devAsympF} are crucial for the stability analysis of the numerical scheme \eqref{def:numSchemeZ} in the $\ell^q$-topology on the whole line $\Z$. For instance, the dissipativity condition of Hypothesis \ref{H:scheme} implies the Von Neumann condition and thus the $\ell^2$-stability of the numerical scheme \eqref{def:numSchemeZ} since Fourier analysis implies
	$$ \forall n\in\N\backslash\lc0\rc,\quad \left\|\Lcc^n\right\|_{\Lc(\ell^2(\Z))}=\left\|F^n\right\|_{L^\infty(\S^1)}=1.$$
	To be more precise, the addition of the diffusivity condition given by \eqref{eq:devAsympF} implies that the numerical scheme \eqref{def:numSchemeZ} on the whole lattice $\Z$ is actually $\ell^q$-stable for all $q\in[1,+\infty]$. Thus, Hypothesis \ref{H:scheme} being verified provides a starting point to study the stability of the numerical scheme \eqref{def:numScheme} on the half-line in the $\ell^q$-topology.
	
	\begin{figure}
		\begin{center}
			\begin{tikzpicture}[scale=2]
				\fill[color=gray!20] (-1.5,-1.5) -- (-1.5,1.5) -- (1.5,1.5) -- (1.5,-1.5) -- cycle;
				\draw[color=black!60] (-1,1) node {$\Oc$};
				\fill[color=white] plot [samples = 100, domain=0:2*pi] ({(1-sin(\x r /2)/3)*cos(\x r)},{sin(\x r)/3}) -- cycle ;
				
				\draw (0,0) circle (1);
				\draw[dashed] (1,0) circle (0.3);
				\draw (1.2,-0.5) node {$B_{\widetilde{\varepsilon}_0}(1)$};
				\draw (75:1.2) node {$\S^1$};
				\draw[thick,red] plot [samples = 100, domain=0:2*pi] ({(1-sin(\x r /2)/3)*cos(\x r)},{sin(\x r)/3});
				\draw[red] (0,0.5) node {$F(\S^1)$};
				\draw[blue] (1,0) node {$\bullet$} node[right] {$1$};
				\draw (0,0) node {$\times$} node[below left] {$0$};
			\end{tikzpicture}
		\end{center}
		
		\caption{An example of curve $F(\S^1)$. Hypothesis \ref{H:scheme} implies that the curve $F(\S^1)$ (in red) is inside the closed disk $\overline{\D}$ and touches the boundary $\S^1$ only at $1$. In grey, we have $\Oc$ the unbounded connected component of $\C\backslash F(\S^1)$. In dashed, we find the ball $B_{\widetilde{\varepsilon}_0}(1)$ where we have a more precise spectral decomposition \eqref{eq:decompStableUnstable_near_1} associated with the matrix $\M(z)$ defined by \eqref{def:M}.}
		\label{fig:Spec}
	\end{figure}
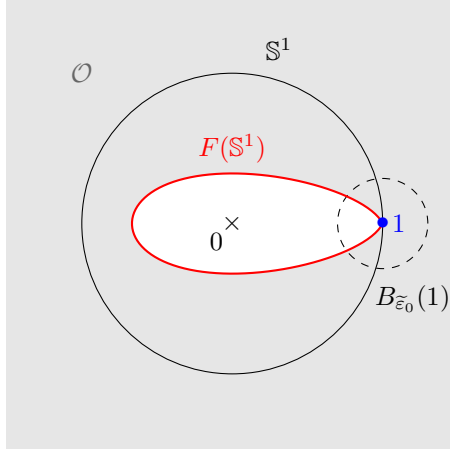
	
	In the rest of the paper, the set $\Oc$ represented on Figure \ref{fig:Spec} corresponds to the exterior of the curve $F(\S^1)$, i.e. the unbounded connected component of $\C\backslash F(\S^1)$. The following lemma is a consequence of Hypothesis \ref{H:scheme}.
	
	\begin{lemma}\label{lem:CFL}
		If Hypothesis \ref{H:scheme} is verified, then $\alpha \in]-p,0[$.
	\end{lemma}
	
	The proof of Lemma \ref{lem:CFL} is entirely similar to the proof of \cite[Lemma 6]{CF} so we will omit it. The result of Lemma \ref{lem:CFL} is comparable to the well-known Courant-Friedrichs-Lewy condition (see \cite{CFL}) and implies that $p\geq 1$. Consequently, the numerical scheme \eqref{def:numScheme} which satisfies Hypothesis \ref{H:scheme} must take information on the right of $j$ to compute $u_j^{n+1}$.
	
	\subsubsection{Lopatinskii determinant, spectral hypothesis and main result}
	
	We introduce for $z\in\C$ the companion matrix 
	\begin{equation}\label{def:M}
		\M(z) := \begin{pmatrix}
			\frac{z\delta_{p-1,0}-a_{p-1}}{a_p} & \frac{z\delta_{p-2,0}-a_{p-2}}{a_p} & \hdots &\hdots & \frac{z\delta_{-r,0}-a_{-r}}{a_p} \\
			1 & 0 &\hdots &  \hdots & 0 \\
			0 & 1& \ddots &    & \vdots\\
			\vdots & \ddots & \ddots & \ddots & \vdots \\
			0 & \hdots & 0 & 1 & 0 
		\end{pmatrix}\in \Mc_{p+r}(\C).
	\end{equation}
	Since $r,p\geq1$ and $a_{-r}a_p\neq 0$, we observe that the matrix $\M(z)$ is well-defined and invertible for all $z\in \C$ and it depends holomorphically on $z$. The matrix $\M(z)$ appears when we study the eigenvalue problem for $\Hc_q$:
	$$(zId_{\Hc_q}-\Tc)u=0.$$
	It will also serve us later on to describe the so-called spatial Green's function in Section \ref{sec:GS}. The following lemma is due to Kreiss (see \cite{Kreiss}) and describes precisely the spectrum of the matrix $\M(z)$ as $z$ belongs to $\Oc\cup\lc1\rc$.
	
	\begin{lemma}[Spectral Splitting]\label{lem:SpecSpl}
		\begin{itemize}
			\item For $z\in \C$, $\kappa\in\C$ is an eigenvalue of $\M(z)$ if and only if $\kappa\neq 0$ and 
			$$F(\kappa)=z.$$
			\item For $z\in \Oc$, the matrix $\M(z)$ has
			\begin{itemize}
				\item no eigenvalue on $\S^1$,
				\item $r$ eigenvalues in $\D\backslash\lc0\rc$ (that we call stable eigenvalues),
				\item  $p$ eigenvalues in $\Uc$ (that we call unstable eigenvalues).
			\end{itemize} 
			\item The matrix $\M(1)$ has $1$ as a simple eigenvalue, $r$ eigenvalues in $\D\backslash\lc0\rc$ and $p-1$ eigenvalues in $\Uc$.
		\end{itemize} 
	\end{lemma}
	
	A complete proof of Lemma \ref{lem:SpecSpl} can be found in \cite[Lemma 3]{CF2}. 
	
	For $z\in\Oc$, we define $E^s(z)$ (resp. $E^u(z)$) the stable (resp. unstable) subspace of $\M(z)$ which corresponds to the subspace spanned by the generalized eigenvectors of $\M(z)$ associated with eigenvalues in $\D$ (resp. $\Uc$). We therefore know that the subspace $E^s(z)$ (resp. $E^u(z)$) has dimension $r$ (resp. $p$) thanks to Lemma \ref{lem:SpecSpl} and we have the decomposition
	\begin{equation}\label{eq:decompStableUnstable_on_Oc}
		\C^{p+r} = E^s(z) \oplus E^u(z).
	\end{equation} 
	The associated projectors are denoted $\pi^s(z)$ and $\pi^u(z)$. Those linear maps commute with $\M(z)$ and depend holomorphically on $z\in \Oc$ (see \cite{Kato}). 
	
	We now need to clarify the situation near $z=1$. Using Lemma \ref{lem:SpecSpl}, we know that $1$ is a simple eigenvalue of $\M(1)$ and that the matrix $\M(1)$ has $r$ eigenvalues in $\D\backslash\lc0\rc$ and $p-1$ eigenvalues in $\Uc$. Therefore, there exist a radius $\widetilde{\varepsilon}_0>0$ and a holomorphic function $\kappa:B_{\widetilde{\varepsilon}_0}(1)\rightarrow \C$ such that $\kappa(1)=1$ and for all $z\in B_{\widetilde{\varepsilon}_0}(1)$, $\kappa(z)$ is a simple eigenvalue of $\M(z)$, $\M(z)$ has $r$ eigenvalues distinct from $\kappa(z)$ in $\D\backslash\lc0\rc$ and $p-1$ eigenvalues distinct from $\kappa(z)$ in $\Uc$. We then have that for $z\in B_{\widetilde{\varepsilon}_0}(1)$, the vector
	$$R_c(z):= \begin{pmatrix}\kappa(z)^{p+r-1} \\ \vdots \\ 1\end{pmatrix}\in\C^{p+r}$$
	is a nonzero eigenvector of $\M(z)$ associated with $\kappa(z)$. For $z\in B_{\widetilde{\varepsilon}_0}(1)$, we define $E^c(z):=\mathrm{Span}(R_c(z))$ and $E^{ss}(z)$ (resp. $E^{su}(z)$) the strictly stable (resp. strictly unstable) subspace of $\M(z)$ which corresponds to the subspace spanned by the generalized eigenvectors of $\M(z)$ associated with eigenvalues distinct from $\kappa(z)$ in $\D$ (resp. $\Uc$). We therefore know that $E^{ss}(z)$ (resp. $E^{su}(z)$) has dimension $r$ (resp. $p-1$) and we have the decomposition
	\begin{equation}\label{eq:decompStableUnstable_near_1}
		\C^{p+r} = E^{ss}(z)\oplus E^c(z) \oplus E^{su}(z).
	\end{equation} 
	The associated projectors are denoted $\pi^{ss}(z)$, $\pi^c(z)$ and $\pi^{su}(z)$. Again, those linear maps commute with $\M(z)$ and depend holomorphically on $z\in \Oc$ (see \cite{Kato}).
	
	For $z\in B_{\widetilde{\varepsilon}_0}(1)\cap\Oc$, Lemma \ref{lem:SpecSpl} implies that $\kappa(z)\in\Uc$. In other words, the "central" eigenvalue $\kappa(z)$ that is close to $1$ comes from $\Uc$ as $z\in\Oc$ approaches $1$. Therefore, for $z\in B_{\widetilde{\varepsilon}_0}(1)\cap\Oc$, we can link the two decompositions \eqref{eq:decompStableUnstable_on_Oc} and \eqref{eq:decompStableUnstable_near_1} of the spectrum of $\M(z)$ and observe that
	\begin{equation}\label{eq:Compa_des_decompos}
		E^{s}(z)=E^{ss}(z) \quad\text{ and }\quad E^u(z)=E^c(z)\oplus E^{su}(z),
	\end{equation}
	so
	\begin{equation*}
		\pi^{s}(z)=\pi^{ss}(z) \quad\text{ and }\quad \pi^u(z)=\pi^c(z)+\pi^{su}(z).
	\end{equation*}
	This allows us to extend holomorphically the vector spaces $E^s$ and $E^u$ and the projectors $\pi^s$ and $\pi^u$ in a neighborhood of $B_{\widetilde{\varepsilon}_0}(1)$. Even if we have to take $\widetilde{\varepsilon}_0>0$ smaller, we can introduce a family of holomorphic functions 
	$$e_1,\ppp,e_r:B_{\widetilde{\varepsilon}_0}(1)\rightarrow \C^{p+r}$$
	such that for all $z\in B_{\widetilde{\varepsilon}_0}(1)$, the family $(e_1(z),\ppp,e_r(z))$ is a basis of $E^s(z)=E^{ss}(z)$ (see \cite[Section II.4.2.]{Kato}). We then define the so-called Lopatinskii determinant $\Delta$ near $1$ by the following formula:
	\begin{equation}\label{def:Lopat}
		\forall z\in B_{\widetilde{\varepsilon}_0}(1), \quad \Delta(z):=\det(\Bc e_1(z),\ppp,\Bc e_r(z)),
	\end{equation}
	with $\Bc$ the matrix defined as \eqref{def:B}. In this way, the function $\Delta$ depends holomorphically on $z$ near $1$. The Lopatinskii determinant $\Delta$ does depend on the choice of the basis $(e_1,\ppp,e_r)$ but we will only be interested in its zeroes, and these are independent of the choice of the basis. 
	
	We need to point out that for all $\underline{z}\in\Oc$, we can also define a holomorphic basis of $E^s(z)$ and thus a Lopatinskii determinant $\Delta(z)$ for $z$ in a neighborhood of $\underline{z}$. The Lopatinskii determinant plays in our situation a similar role as the characteristic polynomial for a matrix as it allows to detect the eigenvalues of the operator $\Tc$. The stability of the numerical scheme \eqref{def:numScheme} depends on the vanishing points of the Lopatinskii determinant $\Delta$. Let us outline some terminology and results (see \cite{GKO}):
	
	\begin{itemize}
		\item[$\blacktriangleright$] The so-called Godunov-Ryabenkii condition introduced in \cite{GodRyab} states that the Lopatinskii determinant does not vanish on $\Uc$, i.e. the operator $\Tc$ acting on $\Hc_q$ does not have any eigenvalue outside of the closed unit disk. It is a necessary stability condition for the numerical scheme \eqref{def:numScheme}. 
		\item[$\blacktriangleright$] If the Lopatinskii determinant does not vanish on the whole set $\overline{\Uc}$, this means that the so-called uniform Kreiss-Lopatinskii condition is satisfied (see \cite{GKS,CoulombelTrieste}). In that case, the main result in \cite{Wu} shows that the operator $\Tc$ on $\Hc_2$ is power bounded.
		\item[$\blacktriangleright$] In \cite{CF2}, the authors study the stability of the explicit numerical schemes for the scalar rightgoing ($v>0$) transport equation on the positive half-line \eqref{def:PDE}. The authors make the assumption that the Godunov-Ryabenkii condition is satisfied and that the Lopatinskii determinant has a finite number of simple zeroes on $\S^1$ that are not in $F(\S^1)$ (i.e. that are different from $1$). Thus, the uniform Kreiss-Lopatinskii condition is not satisfied, and yet the authors prove semigroup estimates that lead to the $\ell^q$-stability of the numerical scheme \eqref{def:numScheme} for all $q\in[1,+\infty]$.		
	\end{itemize}

	In this paper, we make the following assumption.
	
	\begin{hypo}\label{H:spec}
		\begin{itemize}
			\item We suppose that for all $z\in \overline{\Uc}\backslash\lc1\rc$, we have
			$$E^s(z)\oplus \ker \Bc=\C^{p+r}.$$
			In particular, this implies that the Lopatinskii determinant does not vanish anywhere on the set $\overline{\Uc}\backslash\lc1\rc$.
			\item We assume that $1$ is a simple zero of the Lopatinskii determinant, i.e.
			$$\Delta(1)=0,\quad \Delta^\prime(1)\neq0.$$
			We will also consider $\widetilde{\varepsilon}_0$ small enough so that the function $\Delta$ only vanishes in $1$, which implies that
			$$\forall  z\in B_{\widetilde{\varepsilon}_0}(1)\backslash\lc1\rc, \quad E^s(z)\oplus \ker \Bc=\C^{p+r}.$$
		\end{itemize}
	\end{hypo}
	
	\begin{remark}\label{remH:spec}
		We would like to make two observations on Hypothesis \ref{H:scheme}. First, noticing that for $z\in \Oc\cup B_{\widetilde{\varepsilon}_0}(1)$, we have $\dim E^s(z)=r$ and $\dim \ker \Bc=p$, it is interesting to observe for future purposes (see Lemma \ref{lem:SpecTc}) that
		$$E^s(z)\oplus \ker \Bc=\C^{p+r} \Leftrightarrow E^s(z)\cap \ker \Bc=\lc0\rc.$$
		This also means that $\Bc_{|E^s(z)}$ is an isomorphism from $E^s(z)$ to $\C^r$. 
		
		Second, we observe that proving for some concrete choice of numerical scheme that Hypothesis \ref{H:spec} is verified can be challenging (see Section \ref{sec:Num} for simple examples). We would like to point out that in the recent papers \cite{BBS,BBS2}, though the study is done for numerical schemes applied to the rightgoing ($v>0$) transport equation \eqref{def:PDE} which does not coincide with the study of the present paper, the authors present a reliable way to study the verification of the Uniform Kreiss-Lopatinskii condition by counting the number of zeroes of a modified version of the Lopatinskii determinant $\Delta$.
	\end{remark}
	
	Let us settle on the position of Hypothesis \ref{H:spec} compared to the previously stated results. If Hypothesis \ref{H:spec} is verified, then the Godunov-Ryabenkii condition is verified but not the uniform Kreiss-Lopatinskii condition, i.e. the Lopatinskii determinant $\Delta$ does not vanish on the set $\Uc$ but it vanishes on the unit circle and more precisely at $1$, in the essential spectrum of $\Tc$. A similar situation is tackled in \cite{CF2} but with two differences:
	\begin{itemize}
		\item In \cite{CF2}, the transport equation \eqref{def:PDE} which is approached using the numerical scheme \eqref{def:numScheme} is rightgoing (i.e. $v>0$) which is in direct opposition with the case handled in this paper. The main effect of this change is that the "central" eigenvalue $\kappa(z)$ of the matrix $\M(z)$ that is close to $1$ as $z\in\Oc$ approaches $1$ no longer comes from $\Uc$, but rather from the open unit disk $\D$. This in turn changes the equality \eqref{eq:Compa_des_decompos} on the stable and unstable subspaces of the matrix $\M(z)$.
		\item The second and main difference however is on the position of the zeroes of the Lopatinskii determinant $\Delta$ which has a direct impact on the analysis of the so-called spatial Green's function, a useful tool that will be defined in Section \ref{sec:GS}. In \cite{CF2}, the function $\Delta$ can vanish at a finite number of points on the unit circle but not at the point $1$. This allows the authors to holomorphically extend the spatial Green's function in a neighborhood of the interest point $1$, through the essential spectrum of $\Tc$. However, in the present paper, we assume that $1$ is a simple zero of the Lopatinskii determinant $\Delta$ which restricts us to only meromorphically extend the spatial Green's function in a neighborhood of $1$ (see Lemma \ref{lem:GS_près}), which will toughen the computations. It is also interesting to point out that the type of spectral configuration we consider here in Hypothesis \ref{H:spec} with a simple eigenvalue at $1$ which lies in the essential spectral $\Tc$ also occurs in the study of the linear stability of discrete shock profiles for conservation law approximations (see \cite{Serre,Godillon}). We detail a little more on the later point in Remark \ref{rem:Cons} after introducing the temporal Green's function.
	\end{itemize}
	
	The goal is now to study the stability of the scheme \eqref{def:numScheme} under Hypotheses \ref{H:scheme} and \ref{H:spec}. The main result of this paper is the following theorem.
	
	\begin{theorem}\label{th:Stab}
		We assume that Hypotheses \ref{H:scheme} and \ref{H:spec} are verified. 
		\begin{itemize}
			\item If $\Bc\begin{pmatrix}
				1 & \hdots & 1
			\end{pmatrix}^T\in\Bc E^{s}(1)$, then the numerical scheme \eqref{def:numScheme} is $\ell^q$-stable for all $q\in[1,+\infty]$ (see Definition \ref{definition:Stable} of $\ell^q$-stability).
			\item If $\Bc\begin{pmatrix}
				1 & \hdots & 1
			\end{pmatrix}^T\notin\Bc E^{s}(1)$, then the numerical scheme \eqref{def:numScheme} is $\ell^1$-stable but $\ell^q$-unstable for all $q\in]1,+\infty]$. Furthermore, for all $q\in ]1,+\infty]$, there exists a positive constant $C$ such that
			\begin{equation}\label{rateofexploTqn}
				\forall n\in\N\backslash\lc0\rc, \quad \left\|\Tc^n\right\|_{\Lc(\Hc_q)}\geq Cn^{1-\frac{1}{q}}.
			\end{equation}
		\end{itemize}
	\end{theorem}
	
	This is in direct opposition with the result of \cite{CF2} which proved that the existence of simple zeroes of the Lopatinskii determinant $\Delta$ on the unit circle that are different from $1$ does not prevent the $\ell^q$-stability of the numerical scheme \eqref{def:numScheme} for all $q\in[1,+\infty]$. Let us observe that Hypothesis \ref{H:spec} implies that $\Bc E^s(1)$ is a hyperplane of $\C^r$. Thus, the condition $\Bc (1\hdots 1)^T\in \Bc E^s(1)$ is "rarely" verified. In Section \ref{sec:Num}, we will present a concrete example for each possibility.
	
	\subsubsection{Temporal Green's function}
	
	The proof of Theorem \ref{th:Stab} relies on a precise description of the temporal Green's function associated with $\Tc$ given in Theorem \ref{th:Green} below. We will start by defining the temporal Green's function associated with $\Tc$.
	
	For all $j_0\in \N\backslash\lc0\rc$, we introduce the Dirac mass $\delta_{j_0}\in \bigcap_{q\in[1,+\infty]}\Hc_q$ such that
	\begin{equation}\label{def:DiracT}
		\forall j\in\N\backslash\lc0\rc,\quad  (\delta_{j_0})_j:=\lc\begin{array}{cc}1 & \text{ if }j=j_0, \\ 0 & \text{ else.}\end{array}\right.
	\end{equation}
	We then define the temporal Green's function $\Gc(n,j_0,j)$ of the operator $\Tc$ as 
	\begin{equation}\label{def:GTToep}
		\forall n\in \N,\quad \Gc(n,j_0,\cdot):=\Tc^n\delta _{j_0}\in \bigcap_{q\in[1,+\infty]}\Hc_q .
	\end{equation}
	For any initial condition $u^0\in\Hc_q$, the solution $(u^n)_{n\in\N}$ of \eqref{def:numScheme} can be written as
	\begin{equation}\label{eq:u^n//Gcc}
		\forall n\in\N, \forall j\geq 1-r, \quad u^n_j=(\Tc^n u^0)_j = \sum_{j_0\geq 1} u^0_{j_0}\Gc(n,j_0,j).
	\end{equation}
	Thus, the temporal Green's function $\Gc(n,j_0,j)$ allows us to analyze the behavior of solutions of the numerical scheme \eqref{def:numScheme} over time. Let us observe that the sequences $\delta_{j_0}$ and $\Gc(n,j_0,\cdot)$ are finitely supported. More precisely, we recursively prove that:
	\begin{equation*}
		\forall n\in\N, \forall j_0, j\in\N\backslash\lc0\rc, \quad j-j_0\notin\lc-np,\ppp,nr\rc\Rightarrow \Gc(n,j_0,j)=0.
	\end{equation*}
	
	Just as we introduced the Dirac mass and the temporal Green's function $\Gc(n,j_0,j)$ of the operator $\Tc$, we introduce the Dirac mass $\widetilde{\delta}\in \bigcap_{q\in[1,+\infty]} \ell^q(\Z)$ and the temporal Green's function $\widetilde{\Gc}(n,j)$ of the convolution operator $\Lcc$ defined by:
	\begin{equation}\label{def:DiracL}
		\forall j\in \Z, \quad \widetilde{\delta}_j:=\lc\begin{array}{cc}1 & \text{ if }j=0, \\ 0 & \text{ else,}\end{array}\right.
	\end{equation}
	and
	\begin{equation}\label{def:GTLaur}
		\forall n\in \N,\quad \widetilde{\Gc}(n,\cdot):=\Lcc^n\widetilde{\delta}\in  \bigcap_{q\in[1,+\infty]}\ell^q(\Z) .
	\end{equation}
	
	We also introduce the functions $H_{2\mu}^\beta,E_{2\mu}^\beta: \R\rightarrow \C$, where $\mu\in\N\backslash\lc0\rc$ and $\beta \in \C$ has positive real part, which are defined as
	\begin{align}\label{def:H2mu_et_E2mu}
		\begin{split}
			\forall x \in \R,\quad &H_{2\mu}^\beta(x) := \frac{1}{2\pi} \int_\R e^{ixu}e^{-\beta u^{2\mu}}du,\\
			\forall x \in \R,\quad &E_{2\mu}^\beta(x) := \int_x^{+\infty} H^\beta_{2\mu}(y)dy.
		\end{split}
	\end{align}
	We will recall below why $H_{2\mu}^\beta\in L^1(\R,\R)$. We call the functions $H_{2\mu}^\beta$ generalized Gaussians and the functions $E_{2\mu}^\beta$ generalized Gaussian error functions since for $\mu=1$, we have 
	$$\forall x \in \R,\quad H_{2}^\beta(x)=\frac{1}{\sqrt{4\pi\beta}}e^{-\frac{x^2}{4\beta}}.$$
	Noticing that the function $H_{2\mu}^\beta$ is the inverse Fourier transform of $u\mapsto e^{-\beta u^{2\mu}}$, we observe that
	\begin{equation}\label{eq:E_en_-infty}
		\lim_{x\rightarrow -\infty}E_{2\mu}^\beta(x)=\int_{-\infty}^{+\infty} H^\beta_{2\mu}(y)dy= 1.
	\end{equation}
	
	The temporal Green's function $\widetilde{\Gc}(n,j)$ of the operator $\Lcc$ has been studied thoroughly in \cite{Thomee,D-S,R-S,CF,Coeuret}. For instance, under Hypothesis \ref{H:scheme}, it is known that the family $\left(\widetilde{\Gc}(n,\cdot)\right)_{n\in\N}$ is bounded in $\ell^1(\Z)$. Furthermore, in \cite{R-S,Coeuret}, it is proved that the leading order of the asymptotic behavior of $\widetilde{\Gc}(n,j)$ when $n$ becomes large is the generalized Gaussian wave which travels at speed $\alpha$. For instance, the main result in \cite{R-S} gives:
	\begin{equation}\label{eq:GTLaur//Gauss}
		\widetilde{\Gc}(n,j)=\frac{1}{n^\frac{1}{2\mu}}H_{2\mu}^\beta\left(\frac{j-n\alpha}{n^\frac{1}{2\mu}}\right)+o\left(\frac{1}{n^\frac{1}{2\mu}}\right)
	\end{equation}
	where the remainder is uniform with respect to $j\in\Z$.
	
	In this paper, we aim to prove the following theorem which describes the long time behavior of the temporal Green's function $\Gc(n,j_0,j)$:
	
	\begin{theorem}\label{th:Green}
		Under Hypotheses \ref{H:scheme} and \ref{H:spec}, there exist two sequences $(\Rc^c(j))_{j\in\N\backslash\lc0\rc}$ and $(\Rc^u(j_0,j))_{j_0,j\in\N\backslash\lc0\rc}$ and two constants $C,c>0$ such that if we define for all $n,j_0,j\in\N\backslash\lc0\rc$,
		\begin{equation}\label{def:Err}
			\mathrm{Err}(n,j_0,j):= \Gc(n,j_0,j)- \widetilde{\Gc}(n,j-j_0) - \ind_{np\geq j_0}\Rc^u(j_0,j)-E_{2\mu}^\beta\left(\frac{j_0+n\alpha}{n^\frac{1}{2\mu}}\right)\Rc^c(j),
		\end{equation}
		then, we have that:
		\begin{equation}\label{in:Err}
			\forall n,j_0,j\in\N\backslash\lc0\rc,\quad  |\mathrm{Err}(n,j_0,j)|\leq \frac{Ce^{-cj}}{n^\frac{1}{2\mu}}\exp\left(-c\left(\frac{|n\alpha+j_0|}{n^\frac{1}{2\mu}}\right)^\frac{2\mu}{2\mu-1}\right).
		\end{equation}
		Furthermore, there exist two positive constants $C,c$ such that
		\begin{equation}
			\forall j_0,j\in\N\backslash\lc0\rc,\quad \left|\Rc^u(j_0,j)\right|\leq Ce^{-c(j+j_0)}\quad \text{and}\quad \left|\Rc^c(j)\right|\leq Ce^{-cj}.\label{in:Rc}
		\end{equation}
		Finally, the sequence $\Rc^c$ satisfies that
		\begin{equation}\label{condRc}
			\Rc^c = 0 \Leftrightarrow \Bc(1 \hdots 1)^T\in\Bc E^s(1).
		\end{equation}
	\end{theorem}
	
	The sequences $\Rc^u(j_0,j)$ and $\Rc^c(j)$ correspond to boundary layers which are linked respectively to the vector spaces $E^{su}(1)$ and $E^c(1)$ defined in \eqref{eq:decompStableUnstable_near_1}. The coefficients $\ind_{np\geq j_0}$ and $E_{2\mu}^\beta\left(\frac{j_0+n\alpha}{n^\frac{1}{2\mu}}\right)$ which are in front of the values $\Rc^u(j_0,j)$ and $\Rc^c(j)$ in the definition \eqref{def:Err} of $\mathrm{Err}(n,j_0,j)$ can be described as "activation" coefficients. They are close to $0$ for small times $n$ and get closer to $1$ as $n$ becomes larger.
	
	Let us portray the description of the temporal Green's function $\Gc(n,j_0,j)$ that Theorem \ref{th:Green} conveys. For an initial condition $u^0=\delta_{j_0}$, the solution $\Gc(n,j_0,j)$ of the numerical scheme \eqref{def:numScheme} does not see the boundary condition for sufficiently small times $n$ since the stencil of the numerical scheme is finite. Therefore, it coincides with $(\widetilde{\Gc}(n,j-j_0))_{j\geq1}$ the solution of the numerical scheme \eqref{def:numSchemeZ} with a similar initial condition. Thus, \eqref{eq:GTLaur//Gauss} tells us that the temporal Green's function $\Gc(n,j_0,j)$ is close to a generalized Gaussian wave for small times $n$. The boundary layer $\Rc^c$ and $\Rc^u$ are not activated yet. However, when the time $n$ gets close to $-\frac{j_0}{\alpha}$, the generalized Gaussian wave associated with $\widetilde{\Gc}(n,j-j_0)$ reaches the boundary and the boundary layers $\Rc^u(j_0,j)$ and $\Rc^c(j)$ get activated. As $n$ becomes large compared to $-\frac{j_0}{\alpha}$, most of the generalized Gaussian wave will have passed through the boundary and the boundary layers $\Rc^u(j_0,j)$ and $\Rc^c(j)$ are fully activated.
	
	We can already intuitively deduce Theorem \ref{th:Stab} from Theorem \ref{th:Green} when the boundary layer $\Rc^c$ is different from zero. The defining element that implies the $\ell^1$-stability and the $\ell^q$-instability for $q\in]1,+\infty]$ of the numerical scheme \eqref{def:numScheme} is the independence with respect to $j_0$ of the boundary layer $\Rc^c$. If we consider an initial condition $u^0$, we expect to see a boundary layer
	$$\left(\sum_{j_0\geq1}u^0_{j_0}\right) \Rc^c$$
	appear for large times $n$ because of the equality \eqref{eq:u^n//Gcc} and Theorem \ref{th:Green}. We will clarify the proof in Section \ref{sec:Cor}.
	
	\begin{remark}\label{rem:Cons}
		We make two remarks here:
		
		$\bullet$ It should certainly be possible to prove a generalization of Theorem \ref{th:Stab} and Theorem \ref{th:Green} with a relaxed Hypothesis \ref{H:spec} that allows additional simple zeroes of modulus $1$ for the Lopatinskii determinant $\Delta$ either or not embedded into the essential spectrum. This would be achieved by combining the techniques from this paper and from \cite{CF2}.
		
		$\bullet$ As explained earlier, there are connections between the study of approximations of hyperbolic PDEs with boundary conditions and the study of discrete shock profiles for conservation law approximations. The spectral configuration presented in Hypothesis \ref{H:spec} with a simple eigenvalue at $1$ which lies in the essential spectrum $\Tc$ also occurs in the study of the linear stability of discrete shock profiles for conservation law approximations for Lax shocks (see \cite{Godillon}). Using a similar analysis as in Theorem \ref{th:Green} and, more precisely, calculations similar as those done in Section \ref{sec:GT} may improve the description of the temporal Green's function of stationary discrete shock profiles done in \cite{Godillon}. This could potentially result in an argument for linear (and possibly non-linear) stability for the stationary discrete shock profiles.
	\end{remark}
	
	\subsection{Plan of the paper}
	
	We now present the outline of the paper.
	
	Firstly, in Section \ref{sec:Num}, we numerically verify Theorems \ref{th:Stab} and \ref{th:Green} on the example of the Lax-Friedrichs scheme with a boundary condition so that 
	$$\Bc (1\hdots 1)^T \notin \Bc E^s(1)$$
	which implies that the numerical scheme will be $\ell^1$-stable but $\ell^q$-unstable for all $q\in]1,+\infty]$. We compute the temporal Green's function $\Gc(n,j_0,j)$ and observe the formation of the boundary layer $\Rc^c(j)$. We then assess the accuracy of the estimates \eqref{in:Err} on the error term $\mathrm{Err}(n,j_0,j)$ and of the rate of growth \eqref{rateofexploTqn} for the family $(\Tc^n)_{n\in\N}$ acting on $\Hc_q$. We then consider the example of the O3 scheme with a boundary condition so that 
	$$\Bc (1\hdots 1)^T \in \Bc E^s(1)$$
	which implies that the scheme is $\ell^q$-stable for all $q\in[1,+\infty]$. We will numerically verify this statement. 
	
	In Section \ref{sec:Cor}, we present the proof of Theorem \ref{th:Stab} whilst assuming that Theorem \ref{th:Green} has been proved.  
	
	The main part of this article will be dedicated to the proof of Theorem \ref{th:Green}, which will rely on an approach referred to as spatial dynamics, also used in \cite{ZH,Godillon,CF,CF2,Coeuret}. In Section \ref{sec:GS}, we study the spectrum of the operators $\Tc$ and $\Lcc$ and define the spatial Green's functions for those operators. We then demonstrate precise estimates for the difference between the two spatial Green's functions and extend them meromorphically in a neighborhood of $1$. In Lemma \ref{lem:GS_près}, we also define the boundary layers $\Rc^u$ and $\Rc^c$ that occur in Theorem \ref{th:Green} and prove the assertion \eqref{condRc}.
	
	The proof of Theorem \ref{th:Green} will be presented in Section \ref{sec:GT}. The approach involves expressing the difference of the temporal Green's functions $\Gc(n,j_0,j)-\widetilde{\Gc}(n,j-j_0)$ through the spatial Green's functions. Using the results obtained in Section \ref{sec:GS}, we will then prove bounds on $\mathrm{Err}(n,j_0,j)$. We expect three different behaviors depending on the ratio $j_0/n$: the case where $j_0$ is large compared to $n$ (i.e. $j_0>np$), the case where $j_0$ is small compared to $n$ (i.e. $j_0<-\frac{n\alpha}{2}$) and the case where $j_0$ is close to $-n\alpha$ (i.e. $j_0\in\left[-\frac{n\alpha}{2},np\right]$). The later case will be the bulk of the proof.
	
	\section{Numerical result}\label{sec:Num}
	
	\subsection{Example of unstable boundary condition for the modified Lax-Friedrichs scheme}
	
	\begin{figure}
		\centering
		\begin{subfigure}{0.58\textwidth}
			\includegraphics[width=\textwidth]{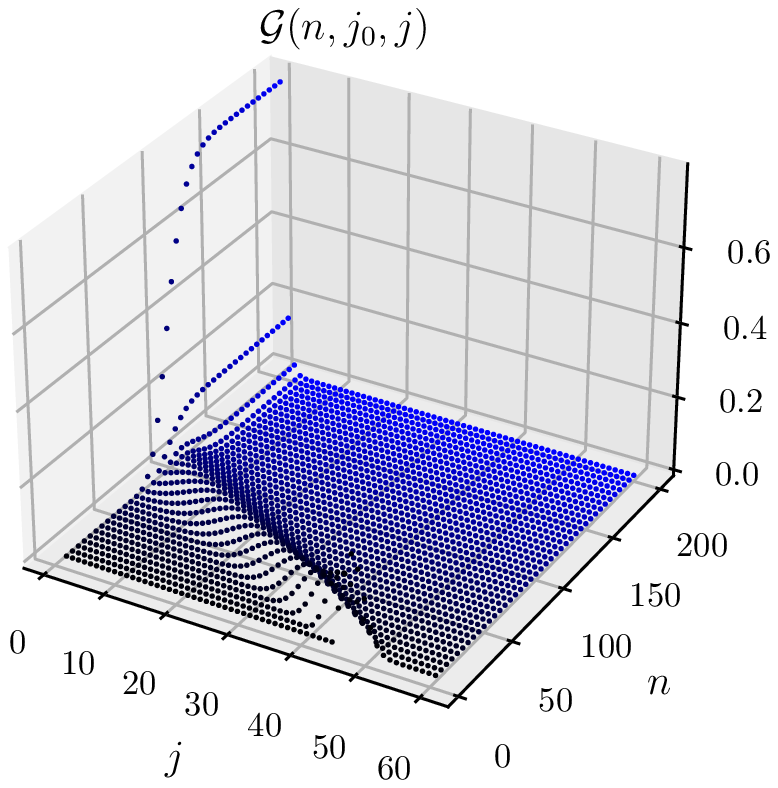}
		\end{subfigure}
		\hfill
		\centering
		\begin{subfigure}{0.37\textwidth}
			\includegraphics[width=\textwidth]{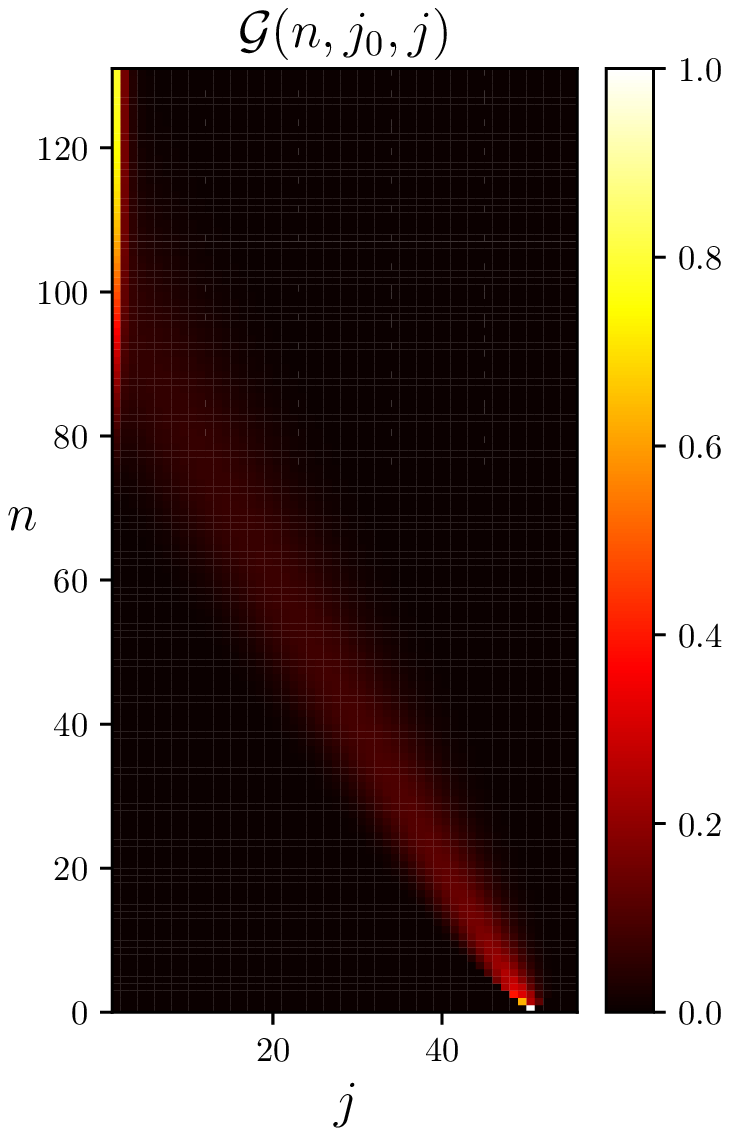}
		\end{subfigure}
		\caption{We consider the modified Lax-Friedrichs scheme \eqref{numLFR} with the parameters $\alpha=-\frac{1}{2}$, $D=\frac{3}{4}$ and $b=5$. Both figures represent the temporal Green's function $\Gc(n,j_0,j)$ for $j_0=50$ which is the solution of the numerical scheme \eqref{numLFR} for the initial condition $u^0=\delta_{j_0}$.}
		\label{fig:GreenLFR}
	\end{figure}
	
	We consider the modified Lax-Friedrichs scheme for the transport equation \eqref{def:PDE}
	\begin{align}\label{numLFR}
		\begin{split}
			\forall n\in \N, \forall j\in \N\backslash\lc0\rc, \quad &u^{n+1}_j = a_{-1}u^n_{j-1} + a_0u^n_j+ a_1u^n_{j+1},\\
			\forall n\in \N, \quad &u^n_0 = b u^n_1,
		\end{split}
	\end{align}
	where $p=r=1$, $\alpha:=\lambda v$, $D>0$, the coefficient $b\in\R$ determines the boundary condition and
	$$a_{-1}=\frac{D+\alpha}{2},\quad a_0=1-D,\quad a_1=\frac{D-\alpha}{2}.$$
	We assume that $D\neq-\alpha$ so that all three coefficients $a_{-1}$, $a_0$ and $a_1$ are nonzero. The symbol $F$ defined by \eqref{def:F} verifies
	$$\forall t\in\R,\quad F(e^{it}) = 1-D+D\cos(t)-i\alpha\sin(t).$$
	If we consider that $\alpha^2<D<1$, then there holds: 
	$$\forall t\in[-\pi,\pi]\backslash\lc0\rc,\quad |F(e^{it})|<1.$$
	Furthermore, we have that
	$$F(e^{it})\underset{t\rightarrow 0}= \exp\left(-i\alpha t -\beta t^2+o(t^2)\right)$$
	with $\beta:= \frac{D-\alpha^2}{2}>0$ and $\mu:=1$ in our notation of Hypothesis \ref{H:scheme}. Therefore, Hypothesis \ref{H:scheme} is verified. 
	
	We will now make a choice for the coefficient $b$ so that Hypothesis \ref{H:spec} is also verified. The matrix $\Bc$ and its kernel for the numerical scheme \eqref{numLFR} are equal to
	$$\Bc = \begin{pmatrix}
		-b & 1
	\end{pmatrix}\quad \text{and} \quad  \ker\Bc = \mathrm{Span} \begin{pmatrix}
		1 \\ b
	\end{pmatrix}.$$
	
	We now need to determine $E^s(z)$ for $z\in \Oc$ and $z=1$. Lemma \ref{lem:SpecSpl} implies that for $z\in \Oc$, the matrix $\M(z)$ has an eigenvalue $\kappa_s(z)\in \D$ and an eigenvalue $\kappa_u(z)\in \Uc$ and that for $z=1$, the matrix $\M(1)$ has an eigenvalue $\kappa_s(1)\in \D$ and $\kappa_u(1):=1$ is a simple eigenvalue of $\M(1)$. We thus have that 
	$$\forall z\in \Oc\cup\lc1\rc,\quad E^s(z)=\mathrm{Span}\begin{pmatrix}
		\kappa_s(z)\\ 1
	\end{pmatrix}.$$
	We also observe that the determinant of the matrix $\M(z)$ is constantly equal to $\frac{a_{-1}}{a_1}$. Thus, for $z\in\Oc\cup\lc1\rc$, $\kappa_s(z)=\kappa_s(1)$ if and only if $1$ is an eigenvalue of $\M(z)$, which is only verified when $z=1$. Furthermore, the Lopatinskii determinant $\Delta$ in a neighborhood of $1$ is equal to $1-b\kappa_s(z)$. To satisfy Hypothesis \ref{H:spec}, we obviously choose $b=\frac{1}{\kappa_s(1)}=\frac{a_1}{a_{-1}}$ so that $1$ is a simple zero of $\Delta$ and 
	$$\forall z\in \Oc, \quad \ker \Bc \cap E^s(z) =\lc0\rc.$$
	Thus, Hypothesis \ref{H:spec} is satisfied. Furthermore, we observe that 
	\begin{equation}\label{subsec:LFR:eg_B}
		\Bc \begin{pmatrix} 1 \\ 1 \end{pmatrix} =1-b \neq 0\quad\text{ and }\quad  \Bc E^s(1) = \lc0\rc.
	\end{equation}
	It ensues from Theorem \ref{th:Stab} that the numerical scheme \eqref{numLFR} is $\ell^1$-stable but $\ell^q$-unstable for all $q\in ]1, +\infty]$.
	
	\begin{figure}
		\centering
		\includegraphics[width=0.55\textwidth]{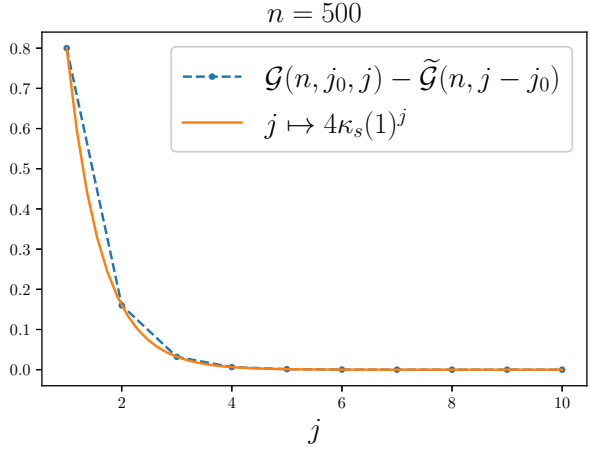}
		\caption{We consider the modified Lax-Friedrichs scheme \eqref{numLFR} with the parameters $\alpha=-\frac{1}{2}$, $D=\frac{3}{4}$ and $b=5$. We represent the difference of the temporal Green's functions $\Gc(n,j_0,j)-\widetilde{\Gc}(n,j-j_0)$ for $j_0=50$ and $n=500$. Since $n$ is large compared to $-\frac{j}{\alpha}$, there only remains the boundary layer $\Rc^c$ which we can identify.}
		\label{fig:RcLFR}
	\end{figure}

	We now consider the numerical scheme \eqref{numLFR} with the parameters $\alpha=-\frac{1}{2}$, $D=\frac{3}{4}$ and $b=5$. Hypotheses \ref{H:scheme} and \ref{H:spec} are thus fulfilled. We want to observe the decomposition of the temporal Green's function $\Gc(n,j_0,j)$ presented in Theorem \ref{th:Green}. In Figure \ref{fig:GreenLFR}, we consider $j_0=50$ and we apply the numerical scheme \eqref{numLFR} for the initial datum $u^0=\delta_{j_0}$ in order to compute the temporal Green's function $\Gc(n,j_0,j)$ defined by \eqref{def:GTToep}. Let us observe that, since the strictly unstable subspace $E^{su}(1)$ is equal to $\lc0\rc$, following the proof of Theorem \ref{th:Green}, we can prove that $\Rc^u=0$ (see proof of Lemma \ref{lem:GS_près}). Furthermore, because of \eqref{subsec:LFR:eg_B}, Theorem \ref{th:Green} states that $\Rc^c\neq0$. On Figure \ref{fig:GreenLFR}, we observe that:
	
	$\bullet$ For $n$ small compared to $-\frac{j_0}{\alpha}$, we observe that the temporal Green's function $\Gc(n,j_0,j)$ resembles a Gaussian wave which travels at a speed $\alpha$, as expected for the temporal Green's function $\widetilde{\Gc}(n,j-j_0)$ computed using \eqref{def:GTLaur}. 
	
	$\bullet$ When $n$ is close to $-\frac{j_0}{\alpha}$, the Gaussian wave reaches the boundary and activates the boundary layer $\Rc^c$.
	
	$\bullet$When $n$ is large compared to $-\frac{j_0}{\alpha}$, the Gaussian wave has passed the boundary. Furthermore, $E_{2\mu}^\beta\left(\frac{j_0+n\alpha}{n^\frac{1}{2\mu}}\right)$ is close to $1$ so the boundary layer $\Rc^c$ is fully activated.
	
	On Figure \ref{fig:RcLFR}, we consider a large time iteration $n=500$. Theorem \ref{th:Green} states that the difference of the temporal Green's functions $\Gc(n,j_0,j)-\widetilde{\Gc}(n,j-j_0)$ should be close to the boundary layer $\Rc^c$. This allows us to identify the expression of the boundary layer $\Rc^c$ in our case.
	
	Now that we have determined the boundary layer $\Rc^c$ which does not depend on the value of $j_0$, we can compute the error term $\mathrm{Err}(n,j_0,j)$ defined by \eqref{def:Err}. Let us observe that for the numerical scheme \eqref{numLFR}, since Hypotheses \ref{H:scheme} and \ref{H:spec} are verified and $\mu=1$, Theorem \ref{th:Green} states that there exist two positive constants $C,c>0$ such that the term $\mathrm{Err}(n,j_0,j)$ satisfies the following estimate
	\begin{equation}\label{subsec:LFR:in_Err}
		\forall n,j_0,j\in\N\backslash\lc0\rc,\quad \sqrt{n} |\mathrm{Err}(n,j_0,j)|\leq Ce^{-cj} \exp\left(-c\frac{|n\alpha+j_0|^2}{n}\right).
	\end{equation}
	In Figure \ref{fig:ErrorLFR}, we fix $j=1$ and plot $\sqrt{n}\:\mathrm{Err}(n,j_0,1)$ against $n$ and $j_0$. The results support the sharpness of the estimates of Theorem \ref{th:Green}, showing Gaussian behavior as predicted by \eqref{subsec:LFR:in_Err}.

	\begin{figure}
		\centering
		\begin{subfigure}{0.58\textwidth}
			\includegraphics[width=\textwidth]{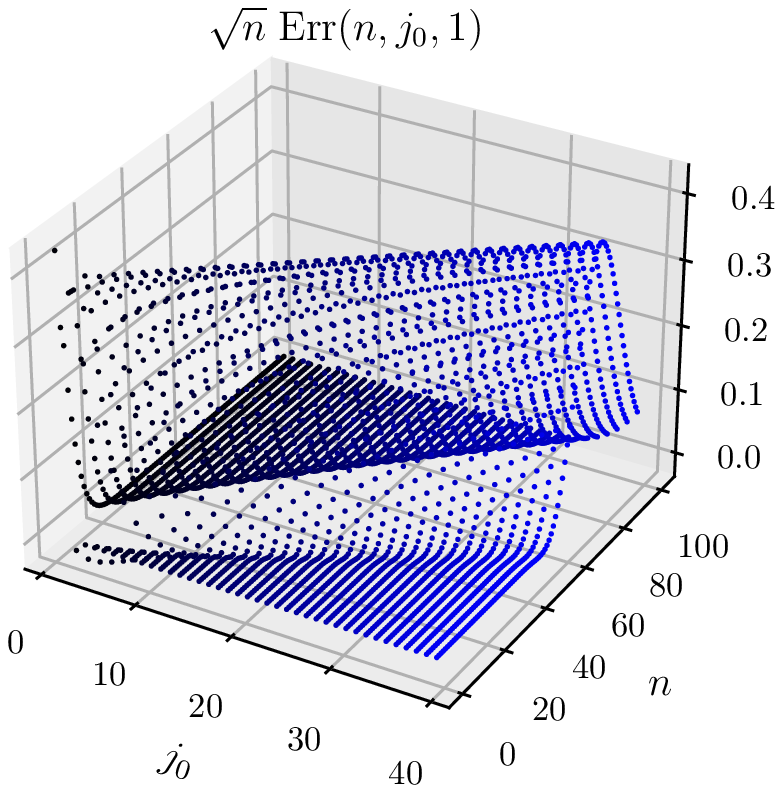}
		\end{subfigure}
		\hfill
		\centering
		\begin{subfigure}{0.37\textwidth}
			\includegraphics[width=\textwidth]{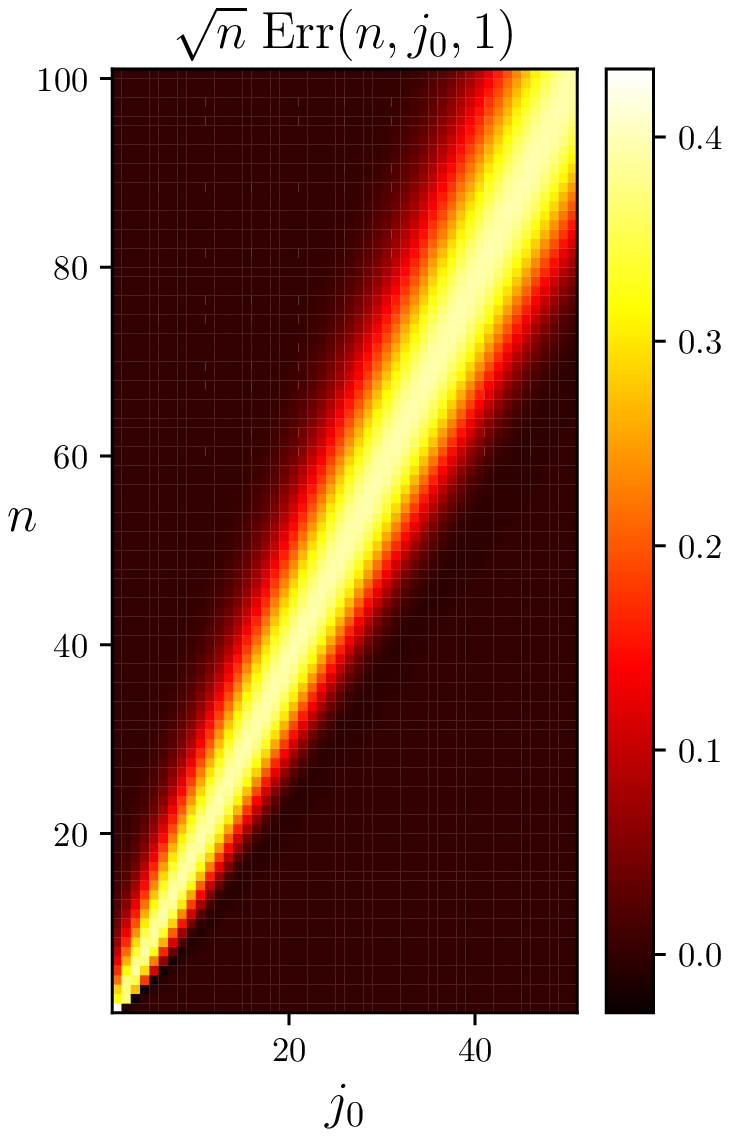}
		\end{subfigure}
		\caption{For the modified Lax-Friedrichs scheme \eqref{numLFR}, we plot the error term $\sqrt{n}\mathrm{Err}(n,j_0,j)$ defined by \eqref{def:Err}. As predicted by Theorem \ref{th:Green}, we observe that $\sqrt{n}\mathrm{Err}(n,j_0,j)$ satisfies Gaussian estimates of the form \eqref{subsec:LFR:in_Err}.}
		\label{fig:ErrorLFR}
	\end{figure}

	Finally, because of \eqref{subsec:LFR:eg_B}, Theorem \ref{th:Stab} states that for $q\in]1,+\infty]$, the numerical scheme \eqref{numLFR} is $\ell^q$-unstable and that there exists a constant $C>0$ such that the family of operators $(\Tc^n)_{n\geq 0}$ in $\Lc(\Hc_q)$ satisfies the inequality \eqref{rateofexploTqn} that we recall here
	\begin{equation}\label{subsecLFR:rateofexploTqn}
		\forall n\in\N, \quad \left\|\Tc^n\right\|_{\Lc(\Hc_q)}\geq C n^{1-\frac{1}{q}}.
	\end{equation}
	In order to numerically verify the sharpness of the inequality \eqref{subsecLFR:rateofexploTqn}, we will consider for $J\in \N\backslash\lc0\rc$ the initial condition $u_J\in \Hc_q$ defined by
	\begin{equation}\label{subsec:LFR:def_uJ}
		u_J=\sum_{j_0=1}^J\delta_{j_0}
	\end{equation}
	and compute the solution $(\Tc^n u_J)_{n\in\N}$ of the numerical scheme \eqref{numLFR}. 
	
	\begin{itemize}
		\item On the left-side of Figure \ref{fig:stabLFR}, we choose $q=+\infty$ and we represent for several choices of $J$ the ratio of $\left\|\Tc^nu_J\right\|_{\Hc_\infty}$ and $\left\|u_J\right\|_{\Hc_\infty}$ which is lower than $\left\|\Tc^n\right\|_{\Lc(\Hc_\infty)}$. We observe a linear increase of $\left\| \Tc^n\right\|_{\Lc(\Hc_{\infty})}$ depending on $n$ that supports the inequality \eqref{subsecLFR:rateofexploTqn}.
		
		\item On the right-side of Figure \ref{fig:stabLFR}, we choose $q=2$ and we represent in the logarithmic scale, for several choices of $J$, the ratio of $\left\|\Tc^nu_J\right\|_{\Hc_2}$ and $\left\|u_J\right\|_{\Hc_2}$ which is lower than $\left\|\Tc^n\right\|_{\Lc(\Hc_2)}$. We observe a growth of $\left\| \Tc^n\right\|_{\Lc(\Hc_{2})}$ at a rate of $\sqrt{n}$ that also supports the inequality \eqref{subsecLFR:rateofexploTqn}.
	\end{itemize}
	
	\begin{figure}
		\centering
		\begin{subfigure}{0.495\textwidth}
			\includegraphics[width=\textwidth]{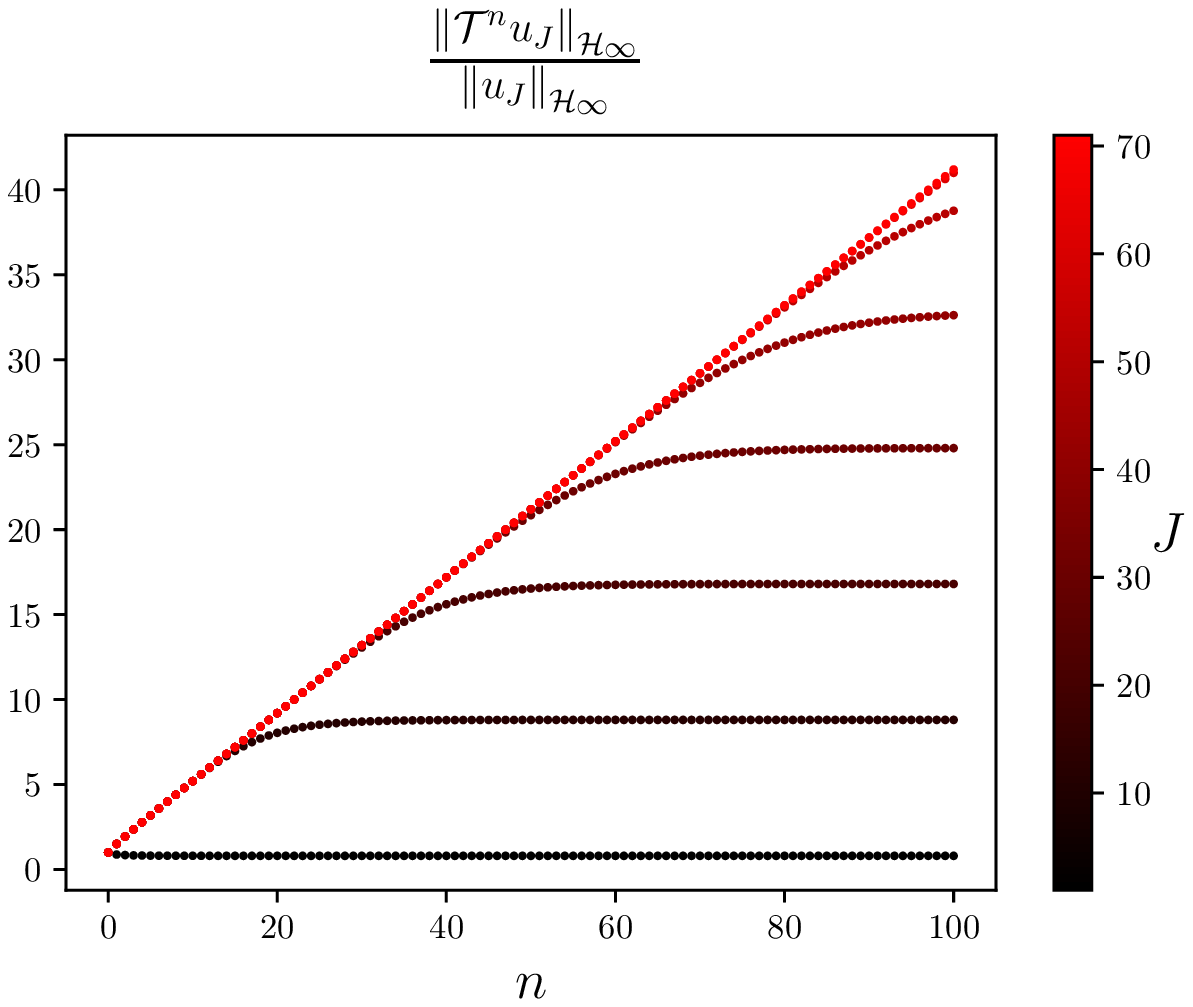}
		\end{subfigure}
		\hfill
		\centering
		\begin{subfigure}{0.495\textwidth}
			\includegraphics[width=\textwidth]{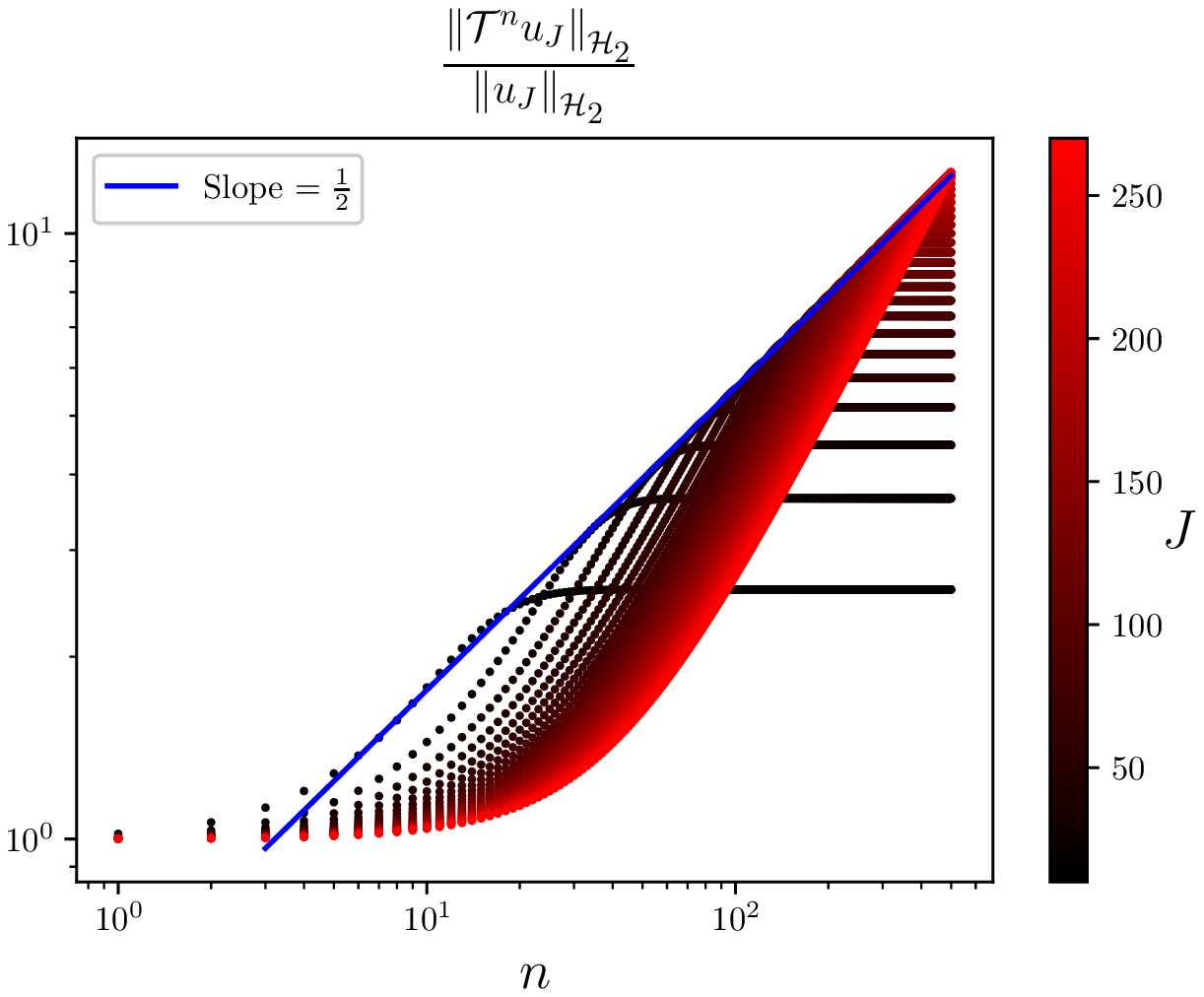}
		\end{subfigure}
		\caption{We consider the modified Lax-Friedrichs scheme \eqref{numLFR}. \newline\underline{On the left side:} For several choices of integers $J$, we compute the ratio between $\left\|\Tc^nu_J\right\|_{\Hc_\infty}$ and $\left\|u_J\right\|_{\Hc_\infty}$ depending on $n$, where the sequence $u_J$ is defined by \eqref{subsec:LFR:def_uJ}. For $n$ fixed, this gives a lower bound for $\left\|\Tc^n\right\|_{\Lc(\Hc_\infty)}$. The figure supports the fact that $\left\|\Tc^n\right\|_{\Lc(\Hc_\infty)}$ grows at a linear rate with regards to $n$. \newline\underline{On the right side:} For several choices of integers $J$, we compute the ratio between $\left\|\Tc^nu_J\right\|_{\Hc_2}$ and $\left\|u_J\right\|_{\Hc_2}$ depending on $n$, where the sequence $u_J$ is defined by \eqref{subsec:LFR:def_uJ}. For $n$ fixed, this gives a lower bound for $\left\|\Tc^n\right\|_{\Lc(\Hc_2)}$. The representation in the logarithmic scale supports the fact that $\left\|\Tc^n\right\|_{\Lc(\Hc_2)}$ grows at a rate $\sqrt{n}$.}
		\label{fig:stabLFR}
	\end{figure}
	
	\subsection{Example of stable boundary condition for the O3 scheme}

	We consider the O3 scheme for the transport equation \eqref{def:PDE}
	\begin{align}\label{numO3}
		\begin{split}
			\forall n\in \N, \forall j\in \N\backslash\lc0\rc, \quad &u^{n+1}_j = a_{-1}u^n_{j-1} + a_0u^n_j+ a_1u^n_{j+1}+a_2 u^n_{j+2},\\
			\forall n\in \N, \quad &u^n_0 = b_1 u^n_1+ b_2u^n_2,
		\end{split}
	\end{align}
	where $r=1$ and $p=2$, $\alpha:=\lambda v$, the coefficients $b_1,b_2\in\R$ determine the boundary condition and
	\begin{align*}
		a_{-1}:= \frac{\alpha(1+\alpha)(2+\alpha)}{6}, & \quad a_0 :=  \frac{(1-\alpha^2)(2+\alpha)}{2},\\
		a_{1}:=  -\frac{\alpha(1-\alpha)(2+\alpha)}{2}, & \quad a_2 :=  \frac{\alpha(1-\alpha^2)}{6}.
	\end{align*}
	We refer to \cite{Despres} for a detailed analysis of this scheme for the rightgoing transport equation on the whole line $\R$. For $\alpha \in]-1,0[$, the symbol $F$ defined by \eqref{def:F} satisfies $F(1) =1$ and 
	$$\forall \kappa \in\S^1\backslash\lc1\rc,\quad |F(\kappa)|<1.$$  
	Furthermore, there exists a constant $\beta>0$ such that
	$$F(e^{it})\underset{t\rightarrow0}= \exp\left(-i\alpha t - \beta t^4+o(t^4)\right).$$
	We have that $\mu=2$ in our notation of Hypothesis \ref{H:scheme}. Thus, Hypothesis \ref{H:scheme} is satisfied.
	
	We now make a choice for the coefficients $b_1$ and $b_2$ that define the numerical boundary condition for the numerical scheme \eqref{numO3}. Lemma \ref{lem:SpecSpl} implies that for $z\in\Oc$, the matrix $\M(z)$ has an eigenvalue $\kappa_s(z)\in\D$ and two (not necessarily distinct) eigenvalues $\kappa_u^1(z),\kappa_u^2(z)\in\Uc$ and that for $z=1$ the matrix $\M(1)$ has an eigenvalue $\kappa_s(1)\in\D$, one eigenvalue $\kappa_u^1(1)\in\Uc$ and $\kappa_u^2(1)=1$ is a simple eigenvalue of $\M(1)$. We thus have that
	$$\forall z\in\Oc\cup\lc1\rc,\quad E^s(z)=\mathrm{Span}\begin{pmatrix}
		\kappa_s(z)^2\\ \kappa_s(z)\\ 1
	\end{pmatrix}.$$
	We observe that the trace and the determinant of $\M(z)$ are respectively equal to $-\frac{a_1}{a_2}$ and $-\frac{a_{-1}}{a_2}$. This implies that for $z\in\Oc\cup\lc1\rc$, $\kappa_s(z)=\kappa_s(1)$ if and only if $1$ is an eigenvalue of $\M(z)$, which is only verified when $z=1$. We choose
	$$b_1:=\frac{1+\kappa_s(1)}{\kappa_s(1)},\quad b_2:=-\frac{1}{\kappa_s(1)}.$$
	so that the Lopatinskii determinant $\Delta$ in a neighborhood of $1$ is defined by:
	$$\Delta(z):=1-b_1\kappa_s(z)-b_2\kappa_s(z)^2 = -b_2(\kappa_s(z)-\kappa_s(1))(\kappa_s(z)-1).$$
	Thus, $1$ is a simple zero of the Lopatinskii determinant and
	$$\forall z\in \Oc, \quad \ker \Bc \cap E^s(z) =\lc0\rc.$$
	Thus, Hypothesis \ref{H:spec} is satisfied. Furthermore, we observe that 
	\begin{equation*}
		\Bc \begin{pmatrix} 1 \\ 1\\1 \end{pmatrix} =1-b_1-b_2 =0 \in\Bc E^s(1).
	\end{equation*}
	It ensues from Theorem \ref{th:Stab} that the numerical scheme \eqref{numLFR} is $\ell^q$-stable for all $q\in [1, +\infty]$. On Figure \ref{fig:stabO3}, we consider for several choices of $J\in \N\backslash\lc0\rc$ the initial condition $u_J\in \Hc_q$ defined by \eqref{subsec:LFR:def_uJ} and compute the solution $(\Tc^n u_J)_{n\in\N}$ of the numerical scheme \eqref{numO3}. We represent the ratio of $\left\|\Tc^nu_J\right\|_{\Hc_q}$ and $\left\|u_J\right\|_{\Hc_q}$ which is lower than $\left\|\Tc^n\right\|_{\Lc(\Hc_q)}$. On the left-side of Figure \ref{fig:stabO3} and on the right-side of Figure \ref{fig:stabO3}, we respectively choose $q=+\infty$ and $q=2$.  The figures support the fact that $\left\|\Tc^n\right\|_{\Lc(\Hc_q)}$ can be uniformly bounded in $n$.
	
	\begin{figure}
		\centering
		\begin{subfigure}{0.495\textwidth}
			\includegraphics[width=\textwidth]{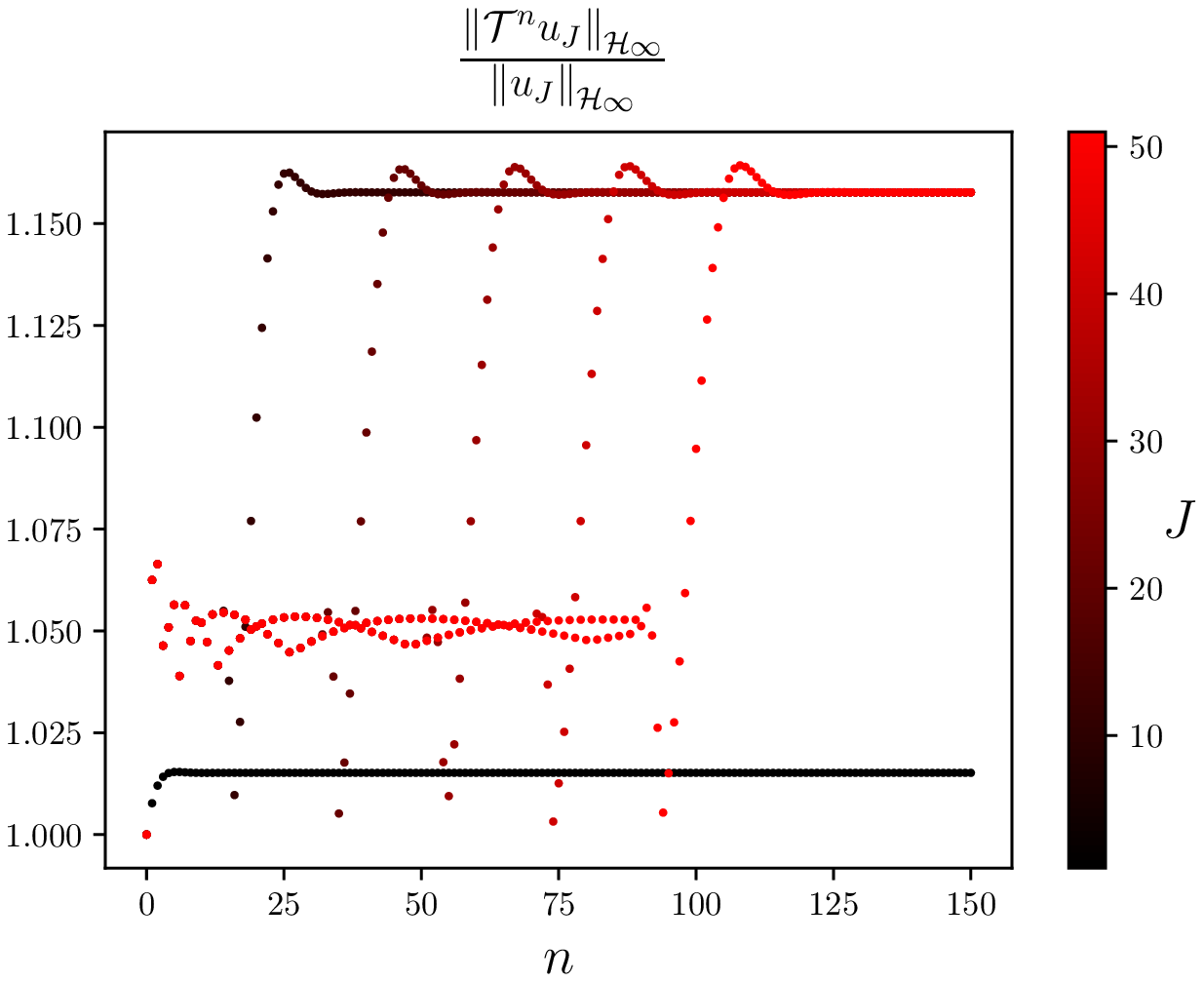}
		\end{subfigure}
		\hfill
		\centering
		\begin{subfigure}{0.495\textwidth}
			\includegraphics[width=\textwidth]{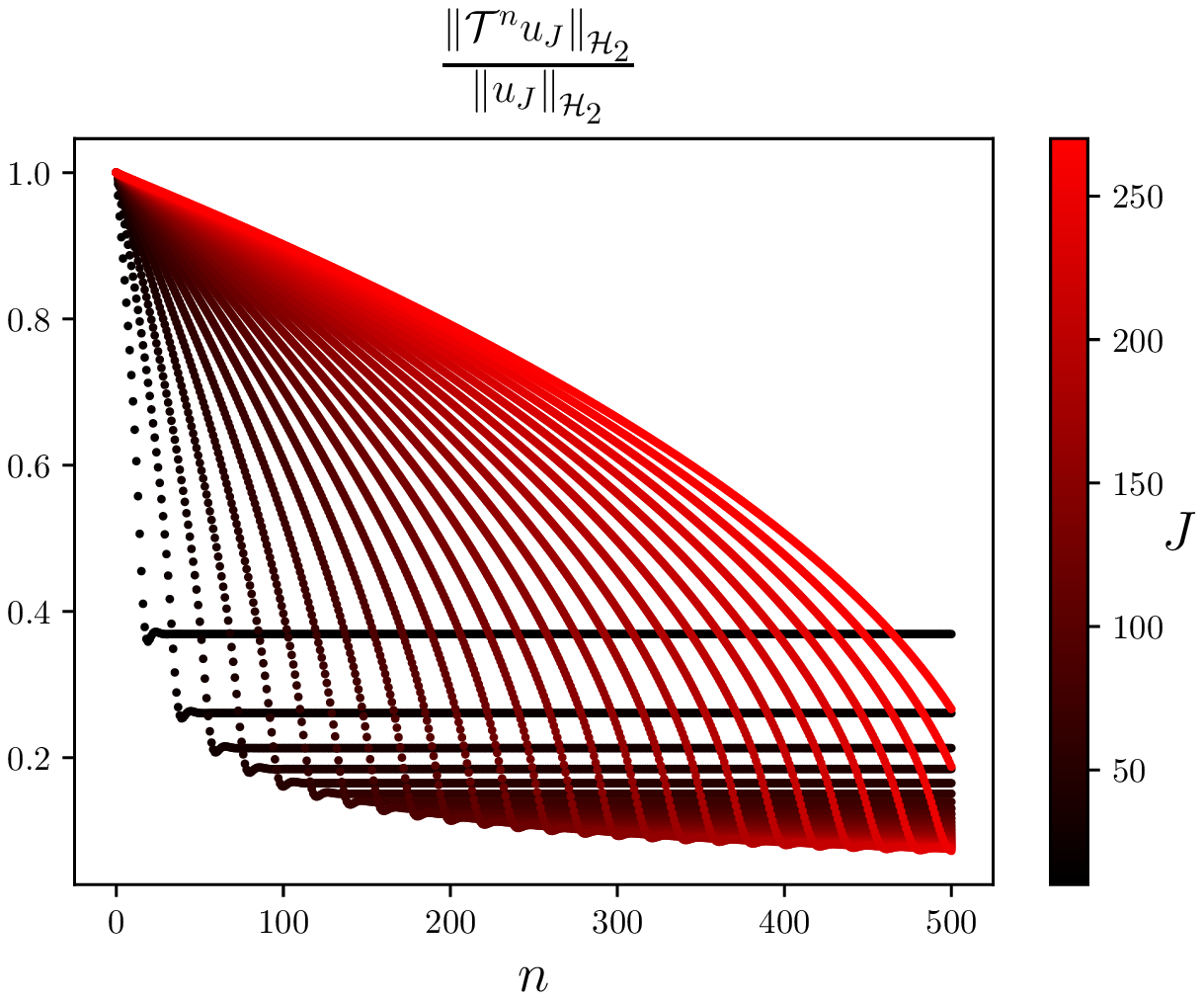}
		\end{subfigure}
		\caption{For the O3 scheme \eqref{numO3}, for several choices of integers $J$, we compute the ratio between $\left\|\Tc^nu_J\right\|_{\Hc_q}$ and $\left\|u_J\right\|_{\Hc_q}$ depending on $n$, where the sequence $u_J$ is defined by \eqref{subsec:LFR:def_uJ}. For $n$ fixed, this gives a lower bound for $\left\|\Tc^n\right\|_{\Lc(\Hc_q)}$. On the left and right sides, we choose respectively $q=+\infty$ and $q=2$.}
		\label{fig:stabO3}
	\end{figure}

	\section{Proof of Theorem \ref{th:Stab} using Theorem \ref{th:Green}}\label{sec:Cor}
	
	In this section, we will prove Theorem \ref{th:Stab} whilst assuming that Theorem \ref{th:Green} has been proved. Theorem \ref{th:Green} allows us to decompose the temporal Green's function $\Gc(n,j_0,j)$ as follows
	\begin{equation}\label{eq:decompoGcc_sec:Cor}
		\forall n,j_0,j\in\N\backslash\lc0\rc,\quad \Gc(n,j_0,j) = \mathrm{Err}(n,j_0,j) + \widetilde{\Gc}(n,j-j_0) + \ind_{np\geq j_0}\Rc^u(j_0,j) + E_{2\mu}^\beta\left(\frac{j_0+n\alpha}{n^\frac{1}{2\mu}}\right)\Rc^c(j).
	\end{equation}
	Furthermore, for $q\in[1,+\infty]$, $u^0\in\Hc_q$ and $n\in \N$, we have that
	\begin{equation}\label{eq:u^n//Gcc_sec:Cor}
		\forall j\in\N\backslash\lc0\rc,\quad (\Tc^nu^0)_j = \sum_{j_0\geq1} u^0_{j_0} \Gc(n,j_0,j).
	\end{equation}
	We decompose the operator $\Tc^n$ in two parts by introducing for all $q\in[1,+\infty]$ and $n\in\N\backslash\lc0\rc$ the two operators $K_{q,n},L_{q,n}\in \Lc(\Hc_q)$ defined as 
	\begin{align*}
		\forall u^0\in \Hc_q, \forall j\in\N\backslash\lc0\rc,& & (K_{q,n}u^0)_j& := \sum_{j_0\geq 1} u^0_{j_0}\left(\mathrm{Err}(n,j_0,j)+\widetilde{\Gc}(n,j-j_0)+ \ind_{np\geq j_0}\Rc^u(j_0,j)\right), \\ 
		\forall u^0\in \Hc_q, \forall j\in\N\backslash\lc0\rc,& & (L_{q,n}u^0)_j& := \sum_{j_0\geq 1} u^0_{j_0}E_{2\mu}^\beta\left(\frac{j_0+n\alpha}{n^\frac{1}{2\mu}}\right)\Rc^c(j),
	\end{align*}
	Using \eqref{eq:decompoGcc_sec:Cor} and \eqref{eq:u^n//Gcc_sec:Cor}, we have
	\begin{equation}\label{eq:decompoTcc_sec:Cor}
		\forall q\in[1,+\infty],\forall n\in\N\backslash\lc0\rc, \quad \Tc^n = K_{q,n}+L_{q,n}.
	\end{equation}	
	
	First, we will prove that for all $q\in[1,+\infty]$ the family of operators $(K_{q,n})_{n\in\N\backslash\lc0\rc}$ is bounded in $\Lc(\Hc_q)$. Using \eqref{eq:decompoTcc_sec:Cor}, this will obviously imply the $\ell^q$-stability of the numerical scheme \eqref{def:numScheme} for all $q\in[1,+\infty]$ when the boundary layer $\Rc^c$ is equal to $0$.
	
	Then, we will prove that the family of $(L_{1,n})_{n\in\N\backslash\lc0\rc}$ is also bounded in $\Lc(\Hc_1)$. Using \eqref{eq:decompoTcc_sec:Cor}, we will have then proved the $\ell^1$-stability of the numerical scheme \eqref{def:numScheme} even when $\Rc^c$ is not equal to $0$
	
	Finally, when the boundary layer $\Rc^c$ is non zero, we will prove for all $q\in ]1, +\infty]$ that there exists a positive constant $C$ such that 
	\begin{equation}\label{in:rateofexploK_sec:Cor}
		\forall n\in\N\backslash\lc0\rc, \quad \left\|L_{q,n}\right\|_{\Lc(\Hc_q)}\geq Cn^{1-\frac{1}{q}}.
	\end{equation}
	Using \eqref{eq:decompoTcc_sec:Cor}, we will thus have proved the existence of a positive constant $C$ such that \eqref{rateofexploTqn} is verified and the $\ell^q$-instability of the numerical scheme \eqref{def:numScheme}.
	
	\vspace{0.1cm}
	\textbf{\underline{Step 1:}} Boundedness of the family $(K_{q,n})_{n\in\N\backslash\lc0\rc}$ in $\Lc(\Hc_q)$ and $\ell^q$-stability when $\Rc^c=0$
	
	\vspace{0.1cm}
	$\bullet$ We consider $q\in[1,+\infty]$ and $u^0\in\Hc_q$. We also introduce $\tilde{q}\in[1,+\infty]$ the Hölder conjugate of $q$. Using the estimates \eqref{in:Err} on $\mathrm{Err}(n,j_0,j)$ and Hölder's inequality, there exist two positive constants $C,c$ such that for all $n,j\in\N\backslash\lc0\rc$
	$$\sum_{j_0\geq1} |\mathrm{Err}(n,j_0,j)||u^0_{j_0}|\leq Ce^{-cj}\left\|u^0\right\|_{\Hc_q}\left\|\left( \frac{1}{n^\frac{1}{2\mu}}\exp\left(-c\left(\frac{|n\alpha+j_0|}{n^\frac{1}{2\mu}}\right)^\frac{2\mu}{2\mu-1}\right)\right)_{j_0\geq 1}\right\|_{\ell^{\tilde{q}}}.$$
	Furthermore, there exists a constant $C>0$ such that
	$$\forall n\in \N\backslash\lc0\rc, \quad \left\|\left( \frac{1}{n^\frac{1}{2\mu}}\exp\left(-c\left(\frac{|n\alpha+j_0|}{n^\frac{1}{2\mu}}\right)^\frac{2\mu}{2\mu-1}\right)\right)_{j_0\geq 1}\right\|_{\ell^{\tilde{q}}}\leq \frac{C}{n^\frac{1}{2\mu q}}.$$
	Therefore, there exists a constant $C>0$ such that for all $u^0\in\Hc_q$ and $n\in \N\backslash\lc0\rc$, the sequence 
	$$\left(\sum_{j_0\geq1}\mathrm{Err}(n,j_0,j)u^0_{j_0}\right)_{j\in\N\backslash\lc0\rc}$$
	belongs to $\Hc_q$ and 
	\begin{equation}\label{in:Err_sec:Cor}
		\forall n\in\N\backslash\lc0\rc,\quad \left\|\left(\sum_{j_0\geq1}\mathrm{Err}(n,j_0,j)u^0_{j_0}\right)_{j\in\N\backslash\lc0\rc}\right\|_{\Hc_q}\leq \frac{C}{n^\frac{1}{2\mu q}}\left\|u^0\right\|_{\Hc_q}.
	\end{equation}
	
	$\bullet$ Using the main result of \cite{Thomee}, we prove that Hypothesis \ref{H:scheme} is one of two conditions so that the family $(\widetilde{\Gc}(n,\cdot))_{n\in\N}$ is bounded in $\ell^1(\Z)$. For $q\in[1,+\infty]$, using Young's convolution inequality $\ell^1(\Z)\ast \ell^q(\Z)\rightarrow \ell^q(\Z)$, we can prove the existence of a positive constant $C>0$ such that for all $u^0\in\Hc_q$ and $n\in \N$, the sequence $\left(\sum_{j_0\geq1}\widetilde{\Gc}(n,j-j_0)u^0_{j_0}\right)_{j\in\N\backslash\lc0\rc}$ belongs to $\Hc_q$ and 
	\begin{equation}\label{in:tildeGcc_sec:Cor}
		\forall n\in\N\backslash\lc0\rc,\quad\left\|\left(\sum_{j_0\geq1}\widetilde{\Gc}(n,j-j_0)u^0_{j_0}\right)_{j\in\N\backslash\lc0\rc}\right\|_{\Hc_q}\leq C\left\|u^0\right\|_{\Hc_q}.
	\end{equation}
	
	$\bullet$ We consider $q\in[1,+\infty]$ and $u^0\in\Hc_q$. We also introduce $\tilde{q}\in[1,+\infty]$ the Hölder conjugate of $q$. Using the bounds \eqref{in:Rc} on $\Rc^u$ and Hölder's inequality, we prove that there exist two positive constants $C,c$ independent from $q$, $u^0$, $n$, $j_0$ and $j$ such that 
	$$\left|\sum_{j_0\geq1}  \ind_{np\geq j_0}\Rc^u(j_0,j)u^0_{j_0}\right|\leq  Ce^{-cj} \left\|u^0\right\|_{\Hc_q} \left\|\left(e^{-cj_0}\right)_{j_0\geq 1}\right\|_{\ell^{\tilde{q}}}.$$
	Therefore, there exists a constant $C>0$ such that for all $u^0\in\Hc_q$ and $n\in \N\backslash\lc0\rc$, the sequence $$\left(\sum_{j_0\geq1} \ind_{np\geq j_0}\Rc^u(j_0,j)u^0_{j_0}\right)_{j\in\N\backslash\lc0\rc}$$ belongs to $\Hc_q$ and 
	\begin{equation}\label{in:Rcu_sec:Cor}
		\forall n\in\N\backslash\lc0\rc,\quad\left\|\left(\sum_{j_0\geq1} \ind_{np\geq j_0}\Rc^u(j_0,j)u^0_{j_0}\right)_{j\in\N\backslash\lc0\rc}\right\|_{\Hc_q}\leq C\left\|u^0\right\|_{\Hc_q}.
	\end{equation}
	
	For $q\in[1,+\infty]$, combining \eqref{in:Err_sec:Cor}-\eqref{in:Rcu_sec:Cor}, we have proved that the family of operators $(K_{q,n})_{n\in\N\backslash\lc0\rc}$ is bounded in $\Lc(\Hc_q)$. When $\Bc\begin{pmatrix} 1 & \hdots &1 \end{pmatrix} ^T\in\Bc E^s(1)$, Theorem \ref{th:Green} implies that the boundary layer $\Rc^c$ is equal to $0$ and thus that the operator $L_{q,n}$ is equal to $0$. We conclude using \eqref{eq:decompoTcc_sec:Cor} that the numerical scheme \eqref{def:numScheme} is $\ell^q$-stable.
	
	\vspace{0.1cm}
	\textbf{\underline{Step 2:}} Boundedness of the family $(L_{1,n})_{n\in\N\backslash\lc0\rc}$ in $\Lc(\Hc_1)$ and proof of the $\ell^1$-stability of the numerical scheme \eqref{def:numScheme}
	
	\vspace{0.1cm}
	We consider that $q=1$. Since the function $E_{2\mu}^\beta$ is bounded, the family $\left(E_{2\mu}^\beta\left(\frac{j_0+n\alpha}{n^\frac{1}{2\mu}}\right)\right)_{n,j_0\in\N\backslash\lc0\rc}$ is bounded. Using the estimates \eqref{in:Rc} on $\Rc^c$, there exist two positive constants $C,c$ such that for all $u^0\in\Hc_1$, $n,j_0,j\in\N\backslash\lc0\rc$, 
	$$\sum_{j_0\geq1} \left|E_{2\mu}^\beta\left(\frac{j_0+n\alpha}{n^\frac{1}{2\mu}}\right)\Rc^c(j)u^0_{j_0}\right|\leq Ce^{-cj} \sum_{j_0\geq1} |u^0_{j_0}|\leq Ce^{-cj} \left\|u^0\right\|_{\Hc_1}.$$
	Therefore, there exists a constant $C>0$ such that for all $u^0\in\Hc_1$ and $n\in \N$, the sequence $$L_{1,n}u^0=\left(\sum_{j_0\geq1}E_{2\mu}^\beta\left(\frac{j_0+n\alpha}{n^\frac{1}{2\mu}}\right)\Rc^c(j)u^0_{j_0}\right)_{j\in\N\backslash\lc0\rc}$$ belongs to $\Hc_1$ and 
	$$\left\|L_{1,n}u^0\right\|_{\Hc_1}\leq C\left\|u^0\right\|_{\Hc_1}.$$
	This implies that the family of operators $(L_{1,n})_{n\in\N\backslash\lc0\rc}$ is bounded in $\Lc(\Hc_1)$. Using \eqref{eq:decompoTcc_sec:Cor}, we then immediately conclude that the family of operators $(\Tc^n)_{n\in\N\backslash\lc0\rc}$ is bounded in $\Lc(\Hc_1)$ and thus that the numerical scheme \eqref{def:numScheme} is $\ell^1$-stable.
	
	\vspace{0.1cm}
	\textbf{\underline{Step 3:}} Proof of \eqref{in:rateofexploK_sec:Cor} and of the $\ell^q$-instability of the numerical scheme \eqref{def:numScheme} for $q\in]1,+\infty]$ when $\Rc^c\neq0$
	
	\vspace{0.1cm}

	We fix $q\in]1,+\infty]$. Since the function $E_{2\mu}^\beta$ is continuous, \eqref{eq:E_en_-infty} implies that there exists a constant $M\in \R$ such that
	\begin{equation}\label{in:E2mu_sec:Cor}
		\forall x\leq M, \quad E_{2\mu}^\beta(x)\geq \frac{1}{2}.
	\end{equation}
	Let us consider an integer $n\in\N\backslash\lc0\rc$ such that 
	$$-\frac{n\alpha}{3}>1\quad \text{ and } \quad \frac{\alpha}{3}n^\frac{2\mu-1}{2\mu}\leq M,$$
	which is possible since $\alpha<0$. We consider an integer $J\in\Z\cap\left[-\frac{n\alpha}{3}, -2\frac{n\alpha}{3}\right]$ and define
	$$u_J:= \sum_{j_0=1}^J\delta_{j_0}\in\Hc_q.$$
	We observe that
	$$L_{q,n}u_J=\left(\sum_{j_0=1}^JE_{2\mu}^\beta\left(\frac{j_0+n\alpha}{n^\frac{1}{2\mu}}\right)\right)\Rc^c.$$
	Yet, for $j_0\in\lc1, \ppp,J\rc$, we have
	$$\frac{n\alpha+j_0}{n^\frac{1}{2\mu}}\leq \frac{\alpha}{3}n^\frac{2\mu-1}{2\mu}\leq M.$$
	Thus, \eqref{in:E2mu_sec:Cor} allows us to conclude that
	\begin{equation*}
		\left\|L_{q,n}u_J\right\|_{\Hc_q}\geq\frac{J}{2}\left\|\Rc^c\right\|_{\Hc_q}.
	\end{equation*}
	Noticing that $\left\|u_J\right\|_{\Hc_q}=J^\frac{1}{q}$, we conclude that
	$$\left\|L_{q,n}\right\|_{\Lc(\Hc_q)}\geq \frac{\left\|\Rc^c\right\|_{\Hc_q}}{2}J^{1-\frac{1}{q}}\geq \frac{\left\|\Rc^c\right\|_{\Hc_q}}{2}\left(-\frac{n\alpha}{3}\right)^{1-\frac{1}{q}}.$$
	Thus, when the boundary layer $\Rc^c$ is a nonzero sequence, the operators $L_{q,n}$ are nonzero operators and there exists a positive constant $C$ such that \eqref{in:rateofexploK_sec:Cor} is verified. Using \eqref{eq:decompoTcc_sec:Cor}, this concludes the proof of the existence of a positive constant $C$ such that \eqref{rateofexploTqn} is verified and the $\ell^q$-instability of the numerical scheme \eqref{def:numScheme}.

	\section{Spatial Green's function}\label{sec:GS}
	
	From now on, we assume that Hypotheses \ref{H:scheme} and \ref{H:spec} are verified and our goal is to prove Theorem \ref{th:Green}. Since the values of the temporal Green's function $\Gc(n,j,j_0)$ are independent of $q$, we consider that we are in the case $q=2$ and we omit the subscript $q$ when we introduce the Banach space $\Hc$. This section is dedicated to the definition and analysis of the spatial Green's functions defined below by \eqref{def:GS_Toep} and \eqref{def:GS_Laur} respectively for the operators $\Tc$ and $\Lcc$.
	
	\subsection{Resolvent set of the operators $\Tc$ and $\Lcc$}
	
	First, we study the spectrum of the operators $\Lcc$ and $\Tc$. This information is fundamental to determine the domain of definition of the spatial Green's functions for $\Tc$ and $\Lcc$. This section will be fairly similar to \cite[Section 2.1]{CF2}.
	
	First, Wiener's theorem \cite{Newman} allows us to conclude that the spectrum of $\Lcc$ is given by:
	$$\sigma(\Lcc)=F(\S^1).$$
	In particular, this implies that the connected open set $\Oc$ is included in the resolvent $\rho(\Lcc)$. We recall that $\Oc$ corresponds to the unbounded connected component of $\C \backslash F(\S^1)$ (see Figure \ref{fig:Spec}).
	
	We now shift our attention to the study of the spectrum of the operator $\Tc$.  The operator $\Tc$ is a finite rank perturbation of the Toeplitz operator defined by
	$$\begin{array}{cccc}\Tcc:& \ell^2(\N\backslash\lc0\rc)&\rightarrow  &\ell^2(\N\backslash\lc0\rc)\\ & u & \mapsto & \left(\displaystyle\sum_{\underset{j+l\geq 1}{l=-r}}^p a_l u_{j+l}\right)_{j\in\N\backslash\lc0\rc}\end{array},$$
	since the Toeplitz operator $\Tcc$ corresponds to the numerical scheme \eqref{def:numScheme} where we impose the numerical boundary conditions $u^n_{1-r}=\ppp=u^n_0=0$. The spectrum of the operator $\Tcc$ has been studied thoroughly, see for instance \cite{Duren} and \cite{TrefethenEmbree}. Mainly, the resolvent set of the operator $\Tc$ corresponds to the points on the complex plane that do not belong to the curve $F(\S^1)$ and such that the winding number of the curve $F(\S^1)$ for these points is $0$. This is the case for all points in the set $\Oc$ for instance. Furthermore, the essential spectrum of $\Tcc$ is the curve $F(\S^1)$. 
	
	We observe that since the operator $\Tc$ is a compact perturbation of the operator $\Tcc$, they share the same essential spectrum (see \cite[Chapter XI, Proposition 4.2]{Conway}) which is equal to the curve $F(\S^1)$. In particular, $1$ belongs to the essential spectrum of the operator $\Tc$. However, a careful examination of the eigenvalues of $\Tc$ is still necessary. Using the decompositions \eqref{eq:decompStableUnstable_on_Oc} and \eqref{eq:decompStableUnstable_near_1} of the vector space $\C^{p+r}$, we prove the following lemma:
	
	\begin{lemma}\label{lem:SpecTc}
		We have that
		\begin{align}
			\forall z\in \Oc\cup B_{\widetilde{\varepsilon}_0}(1), \quad &\dim \ker \Bc \cap E^s(z) =\dim \ker(zId-\Tc),\label{eq:Assertion2_lem:SpecTc}\\
			\forall z\in \Oc, \quad &\ker \Bc \cap E^s(z) = \lc0\rc \Rightarrow z\in \rho(\Tc).\label{eq:Assertion1_lem:SpecTc}
		\end{align}
	\end{lemma}
	
	Hypothesis \ref{H:spec} allows us to conclude that $1$ is a simple eigenvalue of $\Tc$ which lies inside of its essential spectrum. Furthermore, we have
	\begin{equation}\label{resolvTc}
		\overline{\Uc}\backslash\lc1\rc\subset\rho(\Tc)\quad \text{ and } \quad B_{\widetilde{\varepsilon}_0}(1)\cap\Oc \subset\rho(\Tc).
	\end{equation}
	
	\begin{proof}
		The proofs of Assertions \eqref{eq:Assertion2_lem:SpecTc} and \eqref{eq:Assertion1_lem:SpecTc} use the same method as the proof of \cite[Lemma 2.1]{CF2} that we will shortly summarize here. We let the interested reader investigate and adapt the details using \cite{CF2}.
		
		The starting point of the proof is to look for a solution $w\in\Hc$ of the recurrence relation with $z\in \C$ and $f\in \Hc$: 
		\begin{equation}\label{eq:eigProb}
			(zId_\Hc- \Tc)w=f.
		\end{equation}
		We introduce the vectors 
		$$\forall j\geq 1, \quad W_j:=\begin{pmatrix}
			w_{j+p-1} \\ \vdots \\ w_{j-r}
		\end{pmatrix}\in \C^{p+r}, \quad e=\begin{pmatrix}
			1 \\ 0\\\vdots \\ 0
		\end{pmatrix}\in \C^{p+r},$$
		so that we can rewrite \eqref{eq:eigProb} as the dynamical system\footnote{We use here the fact that $p_b\leq p$. The case $p_b>p$ could be dealt similarly but would use heavier notations.}
		\begin{equation}\label{eq:systVectEig}
			\lc\begin{array}{cc}
				\forall j\geq 1, & W_{j+1} = \M(z) W_j - \frac{f_j}{a_p}e,\\
				&\Bc W_1=0,
			\end{array}\right.
		\end{equation}
		where the matrices $\M(z)$ and $\Bc$ are respectively defined by \eqref{def:M} and \eqref{def:B}. 
		
		$\blacktriangleright$ To prove Assertion \eqref{eq:Assertion2_lem:SpecTc}, we prove that the linear map
		$$\begin{array}{ccc} \ker (zId-\Tc) &\rightarrow & \ker \Bc \cap E^s(z)\\ w & \mapsto & W_1= \begin{pmatrix}
				w_{p} \\ \vdots \\ w_{1-r}
			\end{pmatrix}\end{array}$$
		is well-defined and invertible. Its inverse will be the linear map 
		$$\begin{array}{ccc}  \ker \Bc \cap E^s(z) &\rightarrow & \ker (zId-\Tc)\\ W_1 &\mapsto & ((W_j)_p)_{j\geq 1}\end{array}$$
		where $(W_j)_{j\geq1}= (\M(z)^{j-1}W_1)_{j\geq1}$ is the solution of \eqref{eq:systVectEig} for $f=0$ and $(X)_p$ denotes the $p$th coefficient of $X\in\C^{p+r}$.
		
		$\blacktriangleright$ The proof of Assertion \eqref{eq:Assertion1_lem:SpecTc} is more technical and relies on the decomposition \eqref{eq:decompStableUnstable_on_Oc} of $\C^{p+r}$ using the stable and unstable subspaces $E^s(z)$ and $E^u(z)$ of the matrix $\M(z)$. The idea is to use the hyperbolic dichotomy of $\M(z)$ to carefully integrate the system \eqref{eq:systVectEig}. This yields the expressions:
		\begin{align}
			&\forall j\geq 1,& \pi^u(z)W_j & =\sum_{k=0}^{+\infty}\frac{f_{j+k}}{a_p}\M(z)^{-1-k}\pi^u(z)e,\label{eq:U}\\
			& & \Bc\pi^s(z)W_1 & =-\Bc\pi^u(z)W_1,\label{eq:B}\\
			&\forall j\geq 1, & \pi^s(z)W_j & =\M(s)^{j-1}\pi^s(z)W_1-\sum_{k=1}^{j-1}\frac{f_{k}}{a_p}\M(z)^{j-1-k}\pi^s(z)e.\label{eq:S}
		\end{align}
		The exponential decay of the sequences $(\M(z)^j\pi^s(z))_{j\in\N}$ and $(\M(z)^{-j}\pi^u(z))_{j\in\N}$ implies that the series 
		$$\sum_{k=0}^{+\infty}\frac{f_{j+k}}{a_p}\M(z)^{-1-k}\pi^u(z)e \quad \text{ and } \quad \sum_{k=1}^{j-1}\frac{f_{k}}{a_p}\M(z)^{j-1-k}\pi^s(z)e$$
		converge. We observe that $\Bc_{|E^s(z)}$ is an isomorphism from $E^s(z)$ onto $\C^r$ because of Remark \ref{remH:spec}. Since $\pi^u(z)W_1$ is defined by \eqref{eq:U}, we can deduce $\pi^s(z)W_1$ from \eqref{eq:B} and then construct two sequences $(\pi^u(z)W_j)_{j\geq1}$ and $(\pi^s(z)W_j)_{j\geq1}$ which are solutions of \eqref{eq:U}-\eqref{eq:S}. This allows us to construct a solution $w\in \Hc$ of \eqref{eq:eigProb} that depends continuously on $f$ and that we can prove to be unique. 
	\end{proof}
	
	Now that we have a clearer idea of the localization of the spectrum of the operators $\Tc$ and $\Lcc$, we can define the spatial Green's functions of those operators.
	
	\subsection{Definition and estimates of the spatial Green's function}
	
	For $j_0\in\N\backslash\lc0\rc$ and $z\in\rho(\Tc)$, we define the spatial Green's function $G(z,j_0,\cdot)\in\Hc$ associated with the operator $\Tc$ as
	\begin{equation}\label{def:GS_Toep}
		G(z,j_0,\cdot):= (zId_\Hc-\Tc)^{-1}\delta_{j_0}
	\end{equation}
	where $\delta_{j_0}$ is defined by \eqref{def:DiracT}. We also define for $j_0\in\N\backslash\lc0\rc$ and $z\in\rho(\Lcc)$ the spatial Green's function $\widetilde{G}(z,\cdot)\in\ell^2(\Z)$ associated with the operator $\Lcc$ as
	\begin{equation}\label{def:GS_Laur}
		\quad \widetilde{G}(z,\cdot):= (zId_{\ell^2(\Z)}-\Lcc)^{-1}\widetilde{\delta}.
	\end{equation}
	where $\widetilde{\delta}$ is defined by \eqref{def:DiracL}. We notice that both functions $G(\cdot,j_0,j)$ and $\widetilde{G}(\cdot,j)$ are defined and holomorphic on $\Oc\cap\rho(\Tc)$ and in particular on $\overline{\Uc}\backslash\lc1\rc$ and $B_{\widetilde{\varepsilon}_0}(1)\cap\Oc$ because of \eqref{resolvTc}. 
	
	The temporal Green's function $\Gc(n,j_0,j)$ and $\widetilde{\Gc}(n,j)$ can be expressed using the spatial Green's function $G(\cdot,j_0,j)$ and $\widetilde{G}(\cdot,j)$ we just introduced (see \eqref{egGTGSToep} and \eqref{egGTGSLaur} below). To obtain the estimates on $\mathrm{Err}(n,j_0,j)$ expected in Theorem \ref{th:Green}, we will need to study the difference of those two spatial Green's function defined by
	\begin{equation}\label{def:R}
		\forall z\in \Oc\cap\rho(\Tc),\forall j_0,j\in\N\backslash\lc0\rc, \quad R(z,j_0,j):=G(z,j_0,j)- \widetilde{G}(z,j-j_0).
	\end{equation}
	In this section, we prove local uniform estimates of $R(z,j_0,j)$ and meromorphic extensions for $R(\cdot,j_0,j)$ in a neighborhood of $1$ with a pole of order $1$ at $z=1$. The first step will be to find an alternative expression of $R(z,j_0,j)$.

	In the same manner as in the proof of Lemma \ref{lem:SpecTc}, we introduce the following vectors for $z\in \Oc\cap \rho(\Tc)$
	$$\forall j,j_0\in\N\backslash\lc0\rc, \quad W(z,j_0,j):=\begin{pmatrix}G(z,j_0,j+p-1)\\ \vdots\\G(z,j_0,j-r) \end{pmatrix}\in\C^{p+r}$$
	and
	$$\forall j\in\Z, \quad \widetilde{W}(z,j):=\begin{pmatrix}\widetilde{G}(z,j+p-1)\\ \vdots\\\widetilde{G}(z,j-r) \end{pmatrix}\in\C^{p+r}.$$
	Both function $W(z,j_0,j)$ and $\widetilde{W}(z,j)$ depend holomorphically on $z$. Using the definition \eqref{def:Toep} and \eqref{def:Laur} of the operators $\Tc$ and $\Lcc$, we prove that they are solutions of the following systems for $z\in \Oc\cap\rho(\Tc)$ and $j_0\in\N\backslash\lc0\rc$:
	\begin{equation*}
		\lc\begin{array}{cc}
			\forall j\geq 1,\quad & W(z,j_0,j+1) = \M(z) W(z,j_0,j) - \displaystyle\frac{\ind_{j=j_0}}{a_p}e,\\
			&\Bc W(z,j_0,1) =0. 
		\end{array}\right.
	\end{equation*}
	and
	\begin{equation*}
		\forall j\in\Z, \quad \widetilde{W}(z,j+1) = \M(z) \widetilde{W}(z,j) - \frac{\ind_{j=0}}{a_p}e.
	\end{equation*}
	Using the projectors $\pi^s(z)$ and $\pi^u(z)$ introduced via the decomposition \eqref{eq:decompStableUnstable_on_Oc} in the same manner as to obtain the equalities \eqref{eq:U}-\eqref{eq:S}, we obtain that for $z\in\Oc\cap\rho(\Tc)$ and $j,j_0\in \N\backslash\lc0\rc$
	\begin{align}
		\pi^{u}(z)W(z,j_0,j) & =\frac{\ind_{j\in]-\infty,j_0]}}{a_p}\M(z)^{j-(j_0+1)}\pi^{u}(z)e,\label{eq:U_GS_toep}\\
		\Bc\pi^{s}(z)W(z,j_0,1) & =-\Bc\pi^{u}(z)W(z,j_0,1),\label{eq:B_GS_toep}\\
		\pi^{s}(z)W(z,j_0,j) & =\M(z)^{j-1}\pi^{s}(z)W(z,j_0,1)-\frac{\ind_{j\in[j_0+1,+\infty[}}{a_p}\M(z)^{j-(j_0+1)}\pi^{s}(z)e.\label{eq:S_GS_toep}
	\end{align}
	and for $z\in\Oc\cap\rho(\Tc)$ and $j\in\Z$
	\begin{align}
		\pi^{u}(z)\widetilde{W}(z,j) & =\frac{\ind_{j\in]-\infty,0]}}{a_p}\M(z)^{j-1}\pi^{u}(z)e,\label{eq:U_GS_Laur}\\
		\pi^{s}(z)\widetilde{W}(z,j) & =-\frac{\ind_{j\in[1,+\infty[}}{a_p}\M(z)^{j-1}\pi^{s}(z)e.\label{eq:S_GS_Laur}
	\end{align}
	Therefore, summing equalities \eqref{eq:U_GS_toep} and \eqref{eq:S_GS_toep} and using equations \eqref{eq:U_GS_Laur} and \eqref{eq:S_GS_Laur}, we have:
	\begin{equation}\label{eq:R}
		\forall z\in \Oc \cap \rho(\Tc), \forall j_0,j\in\N\backslash\lc0\rc,\quad R(z,j_0,j)=(\M(z)^{j-1}\pi^{s}(z)W(z,j_0,1))_p
	\end{equation}
	where $(X)_p$ denotes the $p$th coefficient of $X\in\C^{p+r}$ and we recall that the remainder $R(z,j_0,j)$ is defined by \eqref{def:R}. The expressions \eqref{def:R} and \eqref{eq:R} are the starting point for studying $R(z,j_0,j)$. 
	
	The following lemma is an equivalent form of \cite[Lemma 5]{CF2} and gives local uniform exponential bounds on $R(z,j_0,j)$ far from the spectrum of $\Tc$. 
	\begin{lemma}\label{lem:GS_loin}
		For $z_0\in\Oc\cap\rho(\Tc)$, there exist a radius $\delta>0$ and two positive constants $C,c$ such that $B_\delta(z_0)$ is contained within $\Oc\cap\rho(\Tc)$ and 
		$$\forall z\in B_\delta(z_0),\forall j,j_0\geq1,\quad |R(z,j_0,j)|\leq C\exp(-c(j+j_0)).$$
	\end{lemma}
	
	We observe that the set $\overline{\Uc}\backslash\lc1\rc$ is included in $\Oc\cap\rho(\Tc)$ because of \eqref{resolvTc} and Hypothesis \ref{H:scheme}. Thus, Lemma \ref{lem:GS_loin} can be applied on all points of the set $\overline{\Uc}\backslash\lc1\rc$.
	
	\begin{proof}
		Since $\Oc\cap\rho(\Tc)$ is an open set, we can consider $\delta>0$ such that 
		$$ B_\delta(z_0)\subset\Oc\cap\rho(\Tc).$$
		Up to considering a smaller radius $\delta>0$, the families of matrices $(\M(z)^j\pi^{s}(z))_{j\in\N}$ and $(\M(z)^{-j}\pi^{u}(z))_{j\in\N}$ decay uniformly exponentially, i.e. there exist two positive constants $C,c$ such that
		\begin{equation}\label{in:DichHyperb_lem:GS_loin}
			\forall z\in B_{\delta}(z_0),\forall j\in\N, \quad \begin{array}{cc}\left\|\M(z)^j\pi^{s}(z)\right\| & \leq Ce^{-cj},\\ \left\|\M(z)^{-j}\pi^{u}(z)\right\| & \leq Ce^{-cj}.\end{array}
		\end{equation}		
		Thus, \eqref{in:DichHyperb_lem:GS_loin} implies 
		\begin{equation}\label{in:SInterm_lem:GS_loin}
			\forall z\in B_{\delta}(z_0),\forall j_0,j\in\N\backslash\lc0\rc, \quad \left|\M(z)^{j-1}\pi^{s}(z)W(z,j_0,1)\right|\leq Ce^{-c(j-1)}\left|\pi^{s}(z)W(z,j_0,1)\right|.
		\end{equation}
		Furthermore, using \eqref{eq:U_GS_toep} and \eqref{in:DichHyperb_lem:GS_loin}, we prove the existence of a constant $C>0$ such that
		\begin{equation}\label{in:U_lem:GS_loin}
			\forall z\in B_{\delta}(z_0),\forall j_0\in\N\backslash\lc0\rc, \quad \left|\pi^{u}(z)W(z,j_0,1)\right|\leq Ce^{-cj_0}.
		\end{equation}
		Since the linear map $\Bc_{|E^s(z)}$ is an isomorphism for all $z\in B_{\delta}(z_0)$, up to considering a smaller $\delta$, there exists a constant $C>0$ such that
		\begin{equation}\label{in:B_lem:GS_loin}
			\forall z\in B_{\delta}(z_0),\forall x\in E^s(z), \quad \left|\Bc x\right|\geq C|x|.
		\end{equation}
		Using \eqref{eq:B_GS_toep}, \eqref{in:U_lem:GS_loin} and \eqref{in:B_lem:GS_loin}, we prove that there exists a positive constant $C$ such that
		\begin{equation}\label{in:S_lem:GS_loin}
			\forall z\in B_{\delta}(z_0),\forall j_0\in\N\backslash\lc0\rc, \quad \left|\pi^{s}(z)W(z,j_0,1)\right|\leq Ce^{-cj_0}.
		\end{equation}
		Combining \eqref{in:SInterm_lem:GS_loin} and \eqref{in:S_lem:GS_loin} allows us to conclude that there exist two constants $C,c>0$ such that
		$$\forall z\in B_{\delta}(z_0),\forall j_0,j\in\N\backslash\lc0\rc, \quad \left|\M(z)^{j-1}\pi^{s}(z)W(z,j_0,1)\right|\leq Ce^{-c(j+j_0)}.$$
		Using \eqref{eq:R}, we conclude the proof.
	\end{proof}
	
	We will now study the function $R(z,j_0,j)$ defined by \eqref{def:R} in a neighborhood of $1$ and introduce an equivalent of \cite[Lemma 6]{CF2} with some important differences. Let us first recall that the curve $F(\S^1)$ corresponds to the essential spectrum of the operator $\Tc$ and that $1$ is a simple eigenvalue of $\Tc$. The function $R(z,j_0,j)$ is thus only defined outside the curve $F(\S^1)$ in a neighborhood on $1$. The following lemma proves that $R(z,j_0,j)$ can actually be meromorphically extended near $1$ with a pole of order $1$ at $1$. We point out in advance that the radius $\widetilde{\varepsilon}_0$ present in the statement of the following lemma comes from the decomposition \eqref{eq:decompStableUnstable_near_1} of the vector space $\C^{p+r}$ deduced from the study of the spectrum of the matrix $\M(z)$ in a neighborhood $B_{\widetilde{\varepsilon}_0}(1)$ of $1$.
	
	\begin{lemma}\label{lem:GS_près} 
		There exists a radius $\widetilde{\varepsilon}_1\in]0,\widetilde{\varepsilon}_0[$ such that for all $j_0,j\in\N\backslash\lc0\rc$, the function 
		$$ z\in\Oc\cap\rho(\Tc)\mapsto R(z,j_0,j) $$
		can be meromorphically extended on $B_{\widetilde{\varepsilon}_1}(1)$ with a pole of order $1$ at $1$. Furthermore, there exist some holomorphic complex valued functions $P^{u}(\cdot,j_0,j)$ and $P^{c}(\cdot,j_0,j)$ defined on $B_{\widetilde{\varepsilon}_1}(1)$ for all $j_0,j\in\N\backslash\lc0\rc$ such that :
		\begin{itemize}
			\item The following equality is satisfied
			\begin{equation}\label{egDecW}
				\forall j_0,j\in\N\backslash\lc0\rc,\forall z\in B_{\widetilde{\varepsilon}_1}(1)\backslash\lc1\rc,\quad R(z,j_0,j)= \frac{P^{c}(z,j_0,j)}{z-1}+\frac{P^{u}(z,j_0,j)}{z-1}.
			\end{equation}
			
			\item  There exist some positive constants $C,c$ independent from $z, j_0$ and $j$ such that:
			\begin{equation}\label{inPruprès}
				\forall z\in B_{\widetilde{\varepsilon}_1}(1), \forall j_0,j\in\N\backslash\lc0\rc,\quad |P^{u}(z,j_0,j)|\leq Ce^{-cj-cj_0}.
			\end{equation}
			We will notate 
			\begin{equation}\label{defRu}
				\forall j_0,j\in\N\backslash\lc0\rc,\quad\Rc^u(j_0,j):=P^{u}(1,j_0,j).
			\end{equation}
			
			\item The function $P^{c}$ satisfies the following estimates where $C,c$ are some positive constants independent from $z, j_0$ and $j$ and $\Rc^c:=(\Rc^{c}(j))_{j\in\N\backslash\lc0\rc}$ is a complex valued family:
			\begin{equation}\label{inPrcprès}
				\forall z\in B_{\widetilde{\varepsilon}_1}(1), \forall j_0,j\in\N\backslash\lc0\rc,\quad \begin{array}{c}
					|P^{c}(z,j_0,j)|\leq Ce^{-cj}|\kappa(z)|^{-j_0},\\
					|\Rc^{c}(j)|\leq Ce^{-cj},\\
					|P^{c}(z,j_0,j)-\Rc^{c}(j)\kappa(z)^{-j_0}|\leq C|z-1|e^{-cj}|\kappa(z)|^{-j_0}.
				\end{array}
			\end{equation}
			
			\item The sequence $\Rc^c$ satisfies:
			\begin{equation}\label{lem:GS_près:condRc}
				\Rc^c=0 \Leftrightarrow \Bc\begin{pmatrix}
					1& \hdots  &1
				\end{pmatrix}^T\in \Bc E^s(1).
			\end{equation}
		\end{itemize}
	\end{lemma}
	
	The two sequences $(\Rc^c(j))_{j\in\N\backslash\lc0\rc}$ and $(\Rc^u(j_0,j))_{j_0,j\in\N\backslash\lc0\rc}$ will correspond to the same sequences introduced in the statement of Theorem \ref{th:Green}.
	
	Let us observe that Lemma \ref{lem:GS_près} and \cite[Lemma 6]{CF2} have the same goal but do not state the same result. We recall that in \cite{CF2}, the authors suppose that the Lopatinskii determinant  $\Delta$ does not vanish at $1$ and thus that $1$ is not an eigenvalue of the operator $\Tc$. This allows in \cite[Lemma 6]{CF2} for an holomorphic extension of the spatial Green's function $G(z,j_0,j)$ on a whole neighborhood of $1$.
	
	\begin{proof}
		Using the projectors $\pi^{ss}(z)$, $\pi^c(z)$ and $\pi^{su}(z)$ defined via the decomposition \eqref{eq:decompStableUnstable_near_1}, the equalities \eqref{eq:U_GS_toep}-\eqref{eq:S_GS_toep} imply for $z\in  B_{\widetilde{\varepsilon}_0}(1)\cap\Oc$ and $j_0,j\in\N\backslash\lc0\rc$
		\begin{align}
			\pi^{su}(z)W(z,j_0,j) & =\frac{\ind_{j\in]-\infty,j_0]}}{a_p}\M(z)^{j-(j_0+1)}\pi^{su}(z)e,\label{egUW}\\
			\pi^c(z)W(z,j_0,j) & =\frac{\ind_{j\in]-\infty,j_0]}}{a_p}\kappa(z)^{j-(j_0+1)}\pi^c(z)e,\label{egCW}\\
			\Bc\pi^{ss}(z)W(z,j_0,1) & =-\Bc\left(\pi^{su}(z)W(z,j_0,1)+\pi^c(z)W(z,j_0,1)\right),\label{egBW}\\
			\pi^{ss}(z)W(z,j_0,j) & =\M(z)^{j-1}\pi^{ss}(z)W(z,j_0,1)-\frac{\ind_{j\in[j_0+1,+\infty[}}{a_p}\M(z)^{j-(j_0+1)}\pi^{ss}(z)e.\label{egSW}
		\end{align}
		We want to extend meromorphically the function
		$$z\in B_{\widetilde{\varepsilon}_0}(1)\cap\Oc\mapsto \M(z)^{j-1}\pi^{ss}(z)W(z,j_0,1)$$ 
		on $B_{\widetilde{\varepsilon}_0}(1)$ with a pole at $1$ of order $1$. To do so, we will use equality \eqref{egBW} and Hypothesis \ref{H:spec} to extend meromorphically the function $\pi^{ss}W(\cdot,j_0,1)$ on $B_{\widetilde{\varepsilon}_0}(1)$ with a pole at $1$ of order $1$. For $z\in B_{\widetilde{\varepsilon}_0}(1)\backslash\lc1\rc$, since Hypothesis \ref{H:spec} implies that $\Delta(z)\neq0$ where $\Delta$ is defined by \eqref{def:Lopat}, the matrix 
		$$\begin{pmatrix}
			\Bc e_1(z) &\ppp& \Bc e_r(z)
		\end{pmatrix}$$
		is invertible and 
		$$\begin{pmatrix}
			\Bc e_1(z) &\ppp& \Bc e_r(z)
		\end{pmatrix}^{-1} = \frac{\mathrm{com}\begin{pmatrix} \Bc e_1(z) &\ppp& \Bc e_r(z)\end{pmatrix}^T}{\Delta(z)}.$$
		Hypothesis \ref{H:spec} states that $1$ is a simple zero of $\Delta$. Thus, we can consider the holomorphic function $D:B_{\widetilde{\varepsilon}_0}(1)\rightarrow \Mc_r(\C)$ defined by
		\begin{equation}\label{lem:GS_près:def:D}
			\forall z\in B_{\widetilde{\varepsilon}_0}(1),\quad D(z)=\lc\begin{array}{cc} 
				\frac{z-1}{\Delta(z)}\mathrm{com}\begin{pmatrix} \Bc e_1(z) &\ppp& \Bc e_r(z)\end{pmatrix}^T & \quad \text{ if }z\neq 1,\\ 
				\frac{1}{\Delta^\prime(1)}\mathrm{com}\begin{pmatrix} \Bc e_1(1) &\ppp& \Bc e_r(1)\end{pmatrix}^T  &\quad \text{ if }z= 1 .\end{array}\right.
		\end{equation}
		To study the equality \eqref{egBW}, for $z\in B_{\widetilde{\varepsilon}_0}(1)$ and $j_0\in\N\backslash\lc0\rc$, we introduce
		\begin{align}
			\begin{split}
				V^u(z,j_0)&:=-(e_1(z) \ppp e_r(z))D(z)\Bc \left(\displaystyle\frac{1}{a_p}\M(z)^{-j_0}\pi^{su}(z)e\right),\\
				V^c(z,j_0)&:=-(e_1(z) \ppp e_r(z))D(z)\Bc \left(\displaystyle\frac{1}{a_p}\kappa(z)^{-j_0}\pi^{c}(z)e\right).
			\end{split}\label{defV}
		\end{align}
		We observe that, for all $j_0\in\N\backslash\lc0\rc$, the functions $V^u(\cdot,j_0)$ and $V^c(\cdot,j_0)$ are holomorphic on $B_{\widetilde{\varepsilon}_0}(1)$ and for all $z\in B_{\widetilde{\varepsilon}_0}(1)$, $V^u(z,j_0)$ and $V^c(z,j_0)$ belong to $E^{ss}(z)$ since they are linear combinations of $e_1(z),\ppp,e_r(z)$. The equality \eqref{egBW} and the definition \eqref{defV} also imply
		$$\forall j_0\in\N\backslash\lc0\rc, \forall z\in B_{\widetilde{\varepsilon}_0}(1)\cap\Oc, \quad \Bc\pi^{ss}(z)W(z,j_0,1) = \Bc\frac{V^c(z,j_0)+V^u(z,j_0)}{z-1}.$$
		Furthermore, since $\Delta(z)\neq0$ for $z\in B_{\widetilde{\varepsilon}_0}(1)\cap\Oc$, Hypothesis \ref{H:spec} implies that 
		$$\ker\Bc\cap E^{ss}(z)=\lc0\rc.$$
		We can then conclude that
		\begin{equation}\label{egSWV}
			\forall j_0\in\N\backslash\lc0\rc, \forall z\in B_{\widetilde{\varepsilon}_0}(1)\cap\Oc, \quad \pi^{ss}(z)W(z,j_0,1) = \frac{V^c(z,j_0)+V^u(z,j_0)}{z-1}.
		\end{equation}
		We have thus found a meromorphic extension of the function $\pi^{ss}W(\cdot,j_0,1)$ on $B_{\widetilde{\varepsilon}_0}(1)$ with a pole at $1$ of order $1$. We are then led to introduce for all $j_0,j\in\N\backslash\lc0\rc$ the functions $\Pcc^{c}(\cdot,j_0,j)$ and $\Pcc^{u}(\cdot,j_0,j)$ defined on $B_{\widetilde{\varepsilon}_0}(1)$ as
		\begin{equation}\label{defPr}
			\forall z\in B_{\widetilde{\varepsilon}_0}(1), \quad \begin{array}{cc} \Pcc^{c}(z,j_0,j)&:= \M(z)^{j-1}V^c(z,j_0), \\ \Pcc^{u}(z,j_0,j)&:= \M(z)^{j-1}V^u(z,j_0). \end{array}
		\end{equation}
		The two functions $\Pcc^{c}(\cdot,j_0,j)$ and $\Pcc^{u}(\cdot,j_0,j)$ are both holomorphic on $B_{\widetilde{\varepsilon}_0}(1)$. Since $\kappa(1)=1$, we notice that $V^c(1,j_0)$ does not depend on $j_0$. We then notate 
		\begin{equation}\label{defRcc}
			\forall j_0,j\in\N\backslash\lc0\rc,\quad \Rcc^c(j):= \Pcc^{c}(1,j_0,j).
		\end{equation}
		We recall that $(X)_p$ denotes the $p$th coefficient of $X\in\C^{p+r}$. We denote for $z\in B_{\widetilde{\varepsilon}_0}(1)$ and $j,j_0\in \N\backslash\lc0\rc$ 
		\begin{equation}\label{defPuPcRc}
			P^c(z,j_0,j) := (\Pcc^c(z,j_0,j))_p,\quad P^u(z,j_0,j) := (\Pcc^u(z,j_0,j))_p, \quad \Rc^c(j) := (\Rcc^c(j))_p.
		\end{equation}		
	
		Using the equalities \eqref{egSWV} as well as the definition \eqref{defPr} of the functions $\Pcc^{c}$ and $\Pcc^{u}$, we have that for all $j_0,j\in\N\backslash\lc0\rc$, we can extend the function 
		$$ z\in\Oc\cap\rho(\Tc)\mapsto \M(z)^{j-1}\pi^{ss}(z)W(z,j_0,1) $$
		meromorphically on $B_{\widetilde{\varepsilon}_0}(1)$ with a pole at $1$ of order $1$ with the expression
		\begin{equation}\label{lem:GS_près:eg_Vect}
			\forall j_0,j\in\N\backslash\lc0\rc,\forall z\in B_{\widetilde{\varepsilon}_0}(1)\backslash\lc1\rc,\quad \M(z)^{j-1}\pi^{ss}(z)W(z,j_0,1)= \frac{\Pcc^{c}(z,j_0,j)}{z-1}+\frac{\Pcc^{u}(z,j_0,j)}{z-1}.
		\end{equation}
		
		Using \eqref{eq:R}, \eqref{defPuPcRc} and \eqref{lem:GS_près:eg_Vect} directly imply \eqref{egDecW}. 
		
		We will now prove the different estimates presented in the statement of Lemma \ref{lem:GS_près}. We start by noticing that there exists a radius $\widetilde{\varepsilon}_1\in]0,\widetilde{\varepsilon}_0[$ such that the families of matrices $(\M(z)^j\pi^{ss}(z))_{j\in\N}$ and $(\M(z)^{-j}\pi^{su}(z))_{j\in\N}$ uniformly decay exponentially for $z\in B_{\widetilde{\varepsilon}_1}(1)$, i.e. there exist two positive constants $C,c$ such that
		\begin{equation}\label{inExpoDec}
			\forall z\in B_{\widetilde{\varepsilon}_1}(1),\forall j\in\N, \quad \begin{array}{cc}\left\|\M(z)^j\pi^{ss}(z)\right\| & \leq Ce^{-cj},\\ \left\|\M(z)^{-j}\pi^{su}(z)\right\| & \leq Ce^{-cj}.\end{array}
		\end{equation}		
		To study the central component $\Pcc^{c}$, we will also need the following lemma:
		\begin{lemma}[\cite{CF2}]\label{lemExpoDec}
			We consider a holomorphic function $M$ with values in $\Mc_N(\C)$ defined on a open ball $B_\delta(0)$ with $\delta>0$ and $N\in\N$ such that there exist two positive constants $C,c$ such that
			$$\forall z\in B_\delta(0),\forall j\in\N,\quad \left\|M(z)^j\right\|\leq C e^{-cj}.$$
			Up to considering a smaller radius $\delta$, there exist two new positive constants $C,c$
			$$\forall z_1,z_2\in B_\delta(0),\forall j\in\N,\quad \left\|M(z_1)^j-M(z_2)^j\right\|\leq C |z_1-z_2|e^{-cj}.$$
		\end{lemma}
		
		Using inequalities \eqref{inExpoDec}, Lemma \ref{lemExpoDec} implies that, up to diminishing $\widetilde{\varepsilon}_1$, there exist some positive constants $C,c$ such that
		\begin{equation}\label{inExpoDecZ}
			\forall z\in B_{\widetilde{\varepsilon}_1}(1),\forall j\in\N, \quad \left\|\M(z)^j\pi^{ss}(z)-\M(1)^j\pi^{ss}(1)\right\|  \leq C|z-1|e^{-cj}.
		\end{equation}
		
		Let us find estimates on $\Pcc^{u}$. First, we observe that inequality \eqref{inExpoDec} implies that there exist two positive constant $C,c$ such that
		\begin{equation}\label{inVu}
			\forall z\in B_{\widetilde{\varepsilon}_1}(1), \forall j_0\in\N\backslash\lc0\rc,\quad |V^u(z,j_0)|\leq Ce^{-cj_0}.
		\end{equation}
		Since for $z\in B_{\widetilde{\varepsilon}_1}(1)$ and $j_0\in\N\backslash\lc0\rc$ we have $V^u(z,j_0)\in E^{ss}(z)$, we notice using the definition \eqref{defPr} that
		$$\forall z\in B_{\widetilde{\varepsilon}_1}(1), \forall j,j_0\in\N\backslash\lc0\rc,\quad \Pcc^{u}(z,j_0,j) = \M(z)^{j-1}\pi^{ss}(z)V^u(z,j_0).$$
		Therefore, using the inequalities \eqref{inExpoDec} and \eqref{inVu}, we can prove that there exist two positive constants $C,c$ such that
		$$\forall z\in B_{\widetilde{\varepsilon}_1}(1), \forall j_0,j\in\N\backslash\lc0\rc,\quad	|\Pcc^{u}(z,j_0,j)|\leq Ce^{-cj-cj_0}.$$
		We easily deduce \eqref{inPruprès}.
		
		We now prove the estimates on $\Pcc^{c}$. We observe that the definition \eqref{defV} implies that there exists a positive constant $C$ such that
		\begin{equation}\label{inVc}
			\forall z\in B_{\widetilde{\varepsilon}_1}(1), \forall j_0\in\N\backslash\lc0\rc,\quad |V^c(z,j_0)|\leq C|\kappa(z)|^{-j_0}.
		\end{equation}
		Since for $z\in B_{\widetilde{\varepsilon}_1}(1)$ and $j_0\in\N\backslash\lc0\rc$ we have $V^c(z,j_0)\in E^{ss}(z)$, we notice using the definition \eqref{defPr} that
		$$\forall z\in B_{\widetilde{\varepsilon}_1}(1), \forall j,j_0\in\N\backslash\lc0\rc,\quad \Pcc^{c}(z,j_0,j) = \M(z)^{j-1}\pi^{ss}(z)V^c(z,j_0).$$
		Therefore, using the inequalities \eqref{inExpoDec}, we prove the existence of a constant $C>0$ such that
		\begin{equation}\label{inPrc1}
			\forall z\in B_{\widetilde{\varepsilon}_1}(1), \forall j_0,j\in\N\backslash\lc0\rc,\quad \begin{array}{c}
				|\Pcc^{c}(z,j_0,j)|\leq Ce^{-cj}|\kappa(z)|^{-j_0},\\
				|\Rcc^{c}(j)|\leq Ce^{-cj}.
			\end{array}
		\end{equation}
		Finally, we observe that for $j_0,j\in\N\backslash\lc0\rc$ and $z\in B_{\widetilde{\varepsilon}_1}(1)$ we have
		\begin{multline*}
			\Pcc^{c}(z,j_0,j)-\kappa(z)^{-j_0}\Rcc^c(j) \\= \kappa(z)^{-j_0}\left(\M(1)^{j-1}\pi^{ss}(1)(e_1(1)\ppp e_r(1))D(1)\Bc\pi^c(1)e-\M(z)^{j-1}\pi^{ss}(z)(e_1(z)\ppp e_r(z))D(z)\Bc\pi^c(z)e\right).
		\end{multline*}		
		Therefore, the estimates \eqref{inExpoDec} and \eqref{inExpoDecZ} as well as the mean value inequality directly imply that there exist two positive constants $C,c$ such that
		$$\forall z\in B_{\widetilde{\varepsilon}_1}(1), \forall j_0,j\in\N\backslash\lc0\rc,\quad |\Pcc^{c}(z,j_0,j)-\kappa(z)^{-j_0}\Rcc^c(j)|\leq C|z-1|e^{-cj}|\kappa(z)|^{-j_0}.$$
		
		To conclude the proof of Lemma \ref{lem:GS_près}, there remains to prove the condition \eqref{lem:GS_près:condRc}. First, we need to compute the value of the vector $\pi^c(1)e$. We recall that $1$ is a simple eigenvalue of the matrix $\M(1)$, that $\begin{pmatrix} 1 & \hdots &1 \end{pmatrix} ^T$ is an eigenvector of $\M(1)$ associated with $1$ and thus that
		$$E^c(1) = \mathrm{Span}\begin{pmatrix} 1 & \hdots &1 \end{pmatrix} ^T.$$
		We also know that there exists a unique eigenvector $L=(l_j)_{j\in\lc1,\ppp,p+r\rc}\in\C^{p+r}$ of $\M(1)^T$ associated with the eigenvalue $1$ such that
		\begin{equation}\label{normalisationL}
			L\cdot \begin{pmatrix} 1 & \hdots &1 \end{pmatrix} ^T=1
		\end{equation}
		where the symmetric bilinear form $\cdot$ on $\C^{p+r}$ is defined as\footnote{Observe that this symetric bilinear form is not the Hermitian product on $\C^{p+r}$.}
		$$\forall X,Y\in \C^{p+r},\quad X\cdot Y := \sum_{l=1}^{p+r} X_iY_i.$$
		Then, we have that 
		$$\forall Y\in \C^{p+r},\quad \pi^c(z)Y = (L\cdot Y)\begin{pmatrix} 1 & \hdots &1 \end{pmatrix} ^T.$$
		Thus, applying to the vector $e$ implies that
		$$\pi^c(1)e= l_1\begin{pmatrix} 1 & \hdots &1 \end{pmatrix} ^T.$$
		Since $L$ is an eigenvector of $\M(1)^T$ associated with the eigenvalue $1$, we have
		\begin{equation}\label{eg:lj}
			\forall j\in\lc1,\ppp,p+r\rc,\quad l_j = \left(\sum_{l=-r}^{p-j} a_l -\delta_{j\leq p}\right)\frac{l_1}{a_p}.
		\end{equation}
		Since $F^\prime(1)=-\alpha$, the normalization \eqref{normalisationL} and the equality \eqref{eg:lj} then imply that
		\begin{equation}\label{eg:Pi^c}
			l_1=\frac{a_p}{\alpha}\quad \text{and thus} \quad \pi^c(1)e= \frac{a_p}{\alpha}\begin{pmatrix} 1 & \hdots &1 \end{pmatrix} ^T.
		\end{equation}
		
		Using the definitions of \eqref{defRcc}, \eqref{defV} and \eqref{lem:GS_près:def:D} respectively of $\Rcc^c(j)$, the function $V^c$ and  the function $D$, we have that for $j\in\N\backslash\lc0\rc$ 
		\begin{align}\label{eg:Rcc}
		\begin{split}
			\Rcc^c(j) &=-\frac{1}{\alpha\Delta^\prime(1)} \M(1)^j (e_1(1) \hdots e_r(1)) \mathrm{com}\begin{pmatrix} \Bc e_1(1) & \hdots &\Bc e_r(1) \end{pmatrix} ^T\Bc \begin{pmatrix} 1 & \hdots &1 \end{pmatrix} ^T\\
			& =  -\frac{1}{\alpha\Delta^\prime(1)}  \M(1)^j \left(\sum_{k=1}^r \det\begin{pmatrix} \Bc e_1(1) & \hdots  &\Bc e_{k-1}(1) &\Bc\begin{pmatrix} 1 & \hdots &1 \end{pmatrix} ^T &\Bc e_{k+1}(1) & \hdots &\Bc e_r(1) \end{pmatrix} e_k(1)\right).
		\end{split}
		\end{align}
			
		$\bullet$ We recall that Hypothesis \ref{H:spec} implies that $1$ is a simple zero of the Lopatinskii determinant. As a consequence, the vector space $\Bc E^s(1)$ is of dimension $r-1$. If $\Bc\begin{pmatrix}
			1& \hdots  &1
		\end{pmatrix}^T\in \Bc E^s(1)$, then for all $k\in\lc1,\ppp,r\rc$
		$$\det\begin{pmatrix} \Bc e_1(1) & \hdots  &\Bc e_{k-1}(1) &\Bc\begin{pmatrix} 1 & \hdots &1 \end{pmatrix} ^T &\Bc e_{k+1}(1) & \hdots &\Bc e_r(1) \end{pmatrix}=0$$
		and thus $\Rcc^c(j)=0$ for all $j\in\N\backslash\lc0\rc$. Using \eqref{defPuPcRc}, we deduce the sequence $(\Rc^c(j))_{j\in\N\backslash\lc0\rc}$ is equal to $0$.
		
		$\bullet$ We now suppose that the sequence $(\Rc^c(j))_{j\in\N\backslash\lc0\rc}$ is equal to $0$. Since the matrix $\M(1)$ is a companion matrix  and $\Rc^c(j)$ is the $p$th coefficient of the vector $\Rcc^c(j)$ for all $j\in\N\backslash\lc0\rc$, the equality \eqref{eg:Rcc} implies that 
		$$\forall j\geq r+1,\quad  \Rcc^c(j) = \begin{pmatrix}
			\Rc^c(j+p-1) \\ \vdots \\ \Rc^c(j-r)
		\end{pmatrix}.$$
		If the sequence $(\Rc^c(j))_{j\in\N\backslash\lc0\rc}$ is equal to $0$, then the vector $\Rcc^c(j)$ are equal to $0$ for all $j\geq r+1$. Since the matrix $\M(1)$ is invertible and the family $(e_k(1))_{k\in\lc1,\ppp,r\rc}$ is linearly independent, the equality \eqref{eg:Rcc} implies that for all $k\in\lc1,\ppp,r\rc$
		$$\det\begin{pmatrix} \Bc e_1(1) & \hdots  &\Bc e_{k-1}(1) &\Bc\begin{pmatrix} 1 & \hdots &1 \end{pmatrix} ^T &\Bc e_{k+1}(1) & \hdots &\Bc e_r(1) \end{pmatrix}=0.$$
		Since $1$ is a simple zero of the Lopatinskii determinant $\Delta$, there exists an integer $k\in\lc1,\ppp,r\rc$ such that the family $(\Bc e_j(1))_{j\in\lc1,\ppp,r\rc\backslash\lc k\rc}$ is linearly independent and spans the whole vector space $\Bc E^s(1)$. Therefore, the vector $\Bc\begin{pmatrix} 1 & \hdots &1 \end{pmatrix} ^T$ belongs to $\Bc E^s(1)$.
	\end{proof}

	\section{Study of the temporal Green's function}\label{sec:GT}
	
	This section is dedicated to the proof of Theorem \ref{th:Green}. Our goal is to prove the estimate \eqref{in:Err} on the function $\mathrm{Err}(n,j_0,j)$ defined by \eqref{def:Err} for all $n,j_0,j\in\N\backslash\lc0\rc$. We expect three different behaviors depending on the ratio $j_0/n$ represented on Figure \ref{fig:zones}:
	
	\vspace{0.5cm}
	 \textbf{I-} The case where $j_0$ is large compared to $n$ (i.e. $j_0>np$) will be dealt in Section \ref{subsec:J0_Large}. For those small times $n<\frac{p}{j_0}$, the numerical boundary condition does not have any impact on the computation of the solution $(\Gc(n,j_0,\cdot))_{n\in\N}$ of the numerical scheme \eqref{def:numScheme} with the initial datum $u^0=\delta_{j_0}$. This will allow us to deduce \eqref{in:Err} in this case.
	
	\vspace{0.5cm}
	\textbf{II-} The case where $j_0$ is small compared to $n$ (i.e. $j_0<-\frac{n\alpha}{2}$) which will be tackled in Section \ref{subsec:J0_Small}. It corresponds to the case where the generalized Gaussian wave has already passed the boundary and the boundary layers are almost fully activated.
	
	\vspace{0.5cm}
	\textbf{III-} The case where $j_0$ is close to $-n\alpha$ (i.e. $j_0\in\left[-\frac{n\alpha}{2},np\right]$) which will be tackled in Section \ref{subsec:J0_Close}. The bulk of the proof happens in this section since the limiting estimates occur in this case

	\begin{figure}
		\begin{center}
			\begin{tikzpicture}[scale=4]
				
				\draw[color=gray!60] (0,0) grid[step=0.19] (1,1);
				
				\draw[->] (-0.1,0) -- (1,0) node[right] {$j_0$};
				\draw[->] (0,-0.1) -- (0,1) node[above] {$n$};
				
				\draw[dashed] (0,0) -- (1,1) node[above right] {$j_0=-n\alpha$};
				\draw[thick] (0,0) -- (1/2,1) node[above] {$j_0=-\frac{n\alpha}{2}$};
				\draw[thick] (0,0) -- (1,1/2) node[right] {$j_0=np$};
				
				\draw (3/4,1/5) node {\textbf{I}};
				\draw (1/5,3/4) node {\textbf{II}};
				\draw (2/3,4/5) node {\textbf{III}};
			\end{tikzpicture}
			\caption{The different sectors for the cases \textbf{I}, \textbf{II} and \textbf{III}.}
			\label{fig:zones}
		\end{center}
	\end{figure}
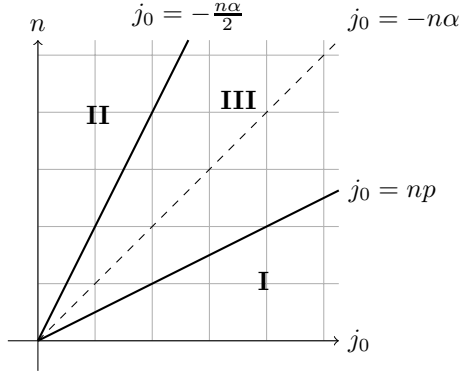

	\vspace{0.5cm}
	
	Before starting the proof of Theorem \ref{th:Green}, we will introduce some useful inequalities on the functions $H_{2\mu}^\beta$ and $E_{2\mu}^\beta$ defined by \eqref{def:H2mu_et_E2mu}.
	\begin{lemma}\label{lemHE}
		There exist two positive constants $C,c$ such that
		\begin{align}
			\forall x\in\R,\quad & |H_{2\mu}^\beta(x)|\leq C\exp(-c|x|^\frac{2\mu}{2\mu-1}),\label{inH}\\
			\forall x\in]0,+\infty[,\quad & |E_{2\mu}^\beta(x)|\leq C\exp(-c|x|^\frac{2\mu}{2\mu-1}),\label{inE+}\\
			\forall x\in]-\infty,0[, \quad& |1-E_{2\mu}^\beta(x)|\leq C\exp(-c|x|^\frac{2\mu}{2\mu-1}).\label{inE-}
		\end{align}
	\end{lemma}
	
	The interested reader can find a proof of \eqref{inH} in \cite[Lemma 9]{Coeuret} or in \cite[Proposition 5.3]{Rob} for a more general point of view. Inequalities \eqref{inE+} and \eqref{inE-} for the function $E_{2\mu}^\beta$ are directly deduced by integrating the function $H_{2\mu}^\beta$ and using \eqref{eq:E_en_-infty} and \eqref{inH}.
	
	\subsection{Case \textbf{I}: $j_0$ is large compared to $n$}\label{subsec:J0_Large}
	
	We consider $n,j_0\in \N\backslash\lc0\rc$ such that $j_0>np$ and we aim to prove the estimate \eqref{in:Err} on $\mathrm{Err}(n,j_0,j)$ in this case for all $j\in\N\backslash\lc0\rc$. First, we can prove that 
	\begin{equation}\label{eg:Gccj0large}
		\forall j\in \N\backslash\lc0\rc, \quad \Gc(n,j_0,j)=\widetilde{\Gc}(n,j-j_0).
	\end{equation}
	This equality translates the fact that for an initial condition $u^0=\delta_{j_0}$, the solution $\Gc(n,j_0,j)$ of the numerical scheme \eqref{def:numScheme} does not see the boundary condition for sufficiently small times $n$.
	
	Using the definition \eqref{def:Err} of $\mathrm{Err}(n,j_0,j)$ and the equality \eqref{eg:Gccj0large}, we have that
	$$\mathrm{Err}(n,j_0,j)=E_{2\mu}^\beta\left(\frac{j_0+n\alpha}{n^\frac{1}{2\mu}}\right)\Rc^c(j).$$
	Using \eqref{inPrcprès} to exponentially bound $\Rc^c(j)$, we observe that there just remains to prove generalized Gaussian estimates on $E_{2\mu}^\beta\left(\frac{j_0+n\alpha}{n^\frac{1}{2\mu}}\right)$ to conclude. Besides, since we have
	$$\frac{j_0+n\alpha}{n^\frac{1}{2\mu}}\geq (p+\alpha)n^\frac{2\mu-1}{2\mu}>0$$
	and the function $x\in[p+\alpha,+\infty[\mapsto x^\frac{1}{2\mu-1}\exp\left(-\frac{c}{2}x^\frac{2\mu}{2\mu-1}\right)$ is bounded where $c$ is the positive constant in \eqref{inE+}, we conclude using \eqref{inE+} that there exists a positive constant $\tilde{C}$ which verify for all $n,j_0\in \N\backslash\lc0\rc$ such that $j_0>np$, we have
	\begin{equation}\label{inAc}
		\left|E_{2\mu}^\beta\left(\frac{j_0+n\alpha}{n^\frac{1}{2\mu}}\right)\right|\leq \frac{\tilde{C}}{n^\frac{1}{2\mu}}\exp\left(-\frac{c}{2}\left(\frac{|j_0+n\alpha|}{n^\frac{1}{2\mu}}\right)^\frac{2\mu}{2\mu-1}\right).
	\end{equation}
	This concludes the proof of \eqref{in:Err} when $j_0>np$.
	
	\subsection{Inverse Laplace transform}\label{subsecInvLapl}
	
	To prove \eqref{in:Err} when $j_0\leq np$, we will use a representation of the temporal Green's functions $\Gc$ and $\widetilde{\Gc}$ using the spatial Green's functions we defined in Section \ref{sec:GS}. Considering a path that surrounds the spectra $\sigma(\Tc)$ and $\sigma(\Lcc)$, for instance $\widetilde{\Gamma}_{r_0}=\exp(r_0)\S^1$ with $r_0\in]0,+\infty[$, the inverse Laplace transform implies that 
	\begin{align}
		\forall n,j_0,j\in\N\backslash\lc0\rc, \quad &\Gc(n,j_0,j) = \frac{1}{2i\pi}\int_{\widetilde{\Gamma}_{r_0}}z^nG(z,j_0,j)dz,\label{egGTGSToep}\\
		\forall n\in\N\backslash\lc0\rc,\forall j\in\Z, \quad &\widetilde{\Gc}(n,j) = \frac{1}{2i\pi}\int_{\widetilde{\Gamma}_{r_0}}z^n\widetilde{G}(z,j)dz,\label{egGTGSLaur}
	\end{align}
	where the spatial Green's functions $G$ and $\widetilde{G}$ are defined by \eqref{def:GS_Toep} and \eqref{def:GS_Laur}. Using the definition of the function $R$ given by \eqref{def:R} and the equalities \eqref{egGTGSToep} and \eqref{egGTGSLaur}, we then obtain that
	\begin{equation}
		\forall n,j_0,j\in\N\backslash\lc0\rc, \quad \Gc(n,j_0,j) -\widetilde{\Gc}(n,j-j_0)= \frac{1}{2i\pi}\int_{\widetilde{\Gamma}_{r_0}}z^nR(z,j_0,j)dz.\label{egPrinc}
	\end{equation}
	Our goal will be to deform the path $\widetilde{\Gamma}_{r_0}$ to use optimally the estimates we determined in Section \ref{sec:GS} on the function $R(z,j_0,j)$ while being aware that this function has a pole of order $1$ at $z=1$. We use the change of variables $z=\exp(\tau)$. If we define the path $\Gamma_{r_0}:=\lc r_0+it, t\in[-\pi,\pi]\rc$ represented on Figure \ref{figGamma} and $\Rg(\tau,j_0,j):=e^\tau R(e^\tau,j_0,j)$, we then have
	\begin{equation}\label{egGTGSred}
		\forall n,j_0,j\in\N\backslash\lc0\rc, \quad \Gc(n,j_0,j) -\widetilde{\Gc}(n,j-j_0) = \frac{1}{2i\pi}\int_{\Gamma_{r_0}}e^{n\tau}\Rg(\tau,j_0,j)d\tau.
	\end{equation}
	
	We recall that in Lemma \ref{lem:GS_près}, we present a precise description of the function $R$ in a neighborhood $B_{\widetilde{\varepsilon}_1}(1)$ of $1$. We fix a radius $\varepsilon^\star_0\in]0,\pi[$ such that
	$$\forall \tau\in B_{\varepsilon^\star_0}(0),\quad e^\tau\in B_{\widetilde{\varepsilon}_1}(1)$$
	and such that there exists a holomorphic function $\varpi:B_{\varepsilon^\star_0}(0)\rightarrow \C$ which verifies $\varpi(0)=0$ and
	$$\forall \tau\in B_{\varepsilon^\star_0}(0),\quad \exp(\varpi(\tau))=\kappa(e^\tau).$$
	We recall that $\kappa(z)$ is the eigenvalue of the matrix $\M(z)$ such that $\kappa(1)=1$ and which depends holomorphically on $z$. We observe that Lemma \ref{lem:SpecSpl} implies that 
	\begin{equation}\label{egFvarpi}
		\forall \tau\in B_{\varepsilon^\star_0}(0),\quad F(e^{\varpi(\tau)}) = e^\tau.
	\end{equation}
	If we define the holomorphic function $\varphi$ such that
	\begin{equation}\label{defPhi}
		\forall \tau\in \C, \quad \varphi(\tau):= -\frac{\tau}{\alpha} +(-1)^{\mu+1}\frac{\beta}{\alpha^{2\mu+1}}\tau^{2\mu}
	\end{equation}
	where $\alpha$ and $\beta$ are defined in Hypothesis \ref{H:scheme}, then using the equality \eqref{egFvarpi} and the asymptotic expansion \eqref{eq:devAsympF} of the logarithm of $F$, we end up proving that there exists a holomorphic function $\xi:B_{\varepsilon^\star_0}(0)\rightarrow \C$ such that 
	\begin{equation}\label{decVarpi}
		\forall \tau\in B_{\varepsilon^\star_0}(0), \quad \varpi(\tau)= \varphi(\tau) +\xi(\tau)\tau^{2\mu+1}.
	\end{equation}

	For all $j_0,j\in\N\backslash\lc0\rc$, we then define the holomorphic functions
	\begin{equation}\label{defPg}
		\forall \tau\in B_{\varepsilon^\star_0}(0),\quad \Pg^{c}(\tau,j_0,j):= \frac{\tau}{e^\tau-1}e^\tau P^{c}(e^\tau,j_0,j) \quad \text{and} \quad \Pg^{u}(\tau,j_0,j):=\frac{\tau}{e^\tau-1}e^\tau P^{u}(e^\tau,j_0,j).
	\end{equation}
	Using Lemma \ref{lem:GS_près}, we can prove that for all $j_0,j\in\N\backslash\lc0\rc$, the function $\Rg(\cdot,j_0,j)$ can be meromorphically extended on $B_{\varepsilon^\star_0}(0)$ with a pole of order $1$ at $0$ and that it satisfies the equality
	\begin{equation}\label{egRgPg}
		\forall \tau\in B_{\varepsilon^\star_0}(0)\backslash\lc0\rc, \quad \Rg(\tau,j_0,j) = \frac{\Pg^{c}(\tau,j_0,j)}{\tau}+\frac{\Pg^{u}(\tau,j_0,j)}{\tau}.
	\end{equation}
	
	We will now prove a lemma to pass from estimates on the function $R(z,j_0,j)$ in Lemmas \ref{lem:GS_loin} and \ref{lem:GS_près} to estimates on the function $\Rg(\tau,j_0,j)$.
	
	\begin{lemma}\label{lem:Rg_Mero_and_esti}
		There exist two positive constants $C,c$ such that
		\begin{align}
			\forall \tau \in B_{\varepsilon^\star_0}(0),\forall j_0,j\in\N\backslash\lc0\rc, \quad &\left|\Pg^{u}(\tau,j_0,j)\right|\leq Ce^{-cj_0-cj},\label{inPgruprès}\\
			\forall \tau \in B_{\varepsilon^\star_0}(0),\forall j_0,j\in\N\backslash\lc0\rc, \quad &\left|\Pg^{c}(\tau,j_0,j)-\Rc^c(j)e^{-j_0\varpi(\tau)}\right|\leq C|\tau|e^{-cj}\exp\left(-j_0\Re(\varpi(\tau))\right).\label{inPgrcprès}
		\end{align}
		
		Furthermore, for all $\varepsilon\in]0,\varepsilon^\star_0[$, there exists a width $\eta_\varepsilon\in]0,\varepsilon[$ such that if we define
		$$\Omega_\varepsilon := \lc\tau\in\C,\Re(\tau)\in]-\eta_\varepsilon,\pi],\Im(\tau)\in[-\pi,\pi]\rc \backslash B_\varepsilon(0)$$
		then for all $j_0,j\in\N\backslash\lc0\rc$ the function $\tau \mapsto \Rg(\tau,j_0,j)$ is holomorphically defined on $\Omega_\varepsilon$ and there exist two positive constants $C_\varepsilon,c_\varepsilon>0$ such that
		\begin{equation}
			\forall \tau \in \Omega_\varepsilon,\forall j_0,j\in\N\backslash\lc0\rc, \quad \left|\Rg(\tau,j_0,j)\right|\leq C_\varepsilon e^{-c_\varepsilon j_0-c_\varepsilon j}.\label{inRgloin}
		\end{equation}		
	\end{lemma}
	
	\begin{proof}
		Inequality \eqref{inPgruprès} is a direct consequence from \eqref{inPruprès}. We also observe that the triangular inequality implies
		\begin{multline*}
			\forall \tau\in B_{\varepsilon_\star}(0) , \forall j_0,j\in\N\backslash\lc0\rc, \\ |\Pg^{c}(\tau,j_0,j)-\Rc^c(j)e^{-j_0\varpi(\tau)}|\leq \left|\frac{\tau}{e^\tau-1}e^\tau-1\right||P^{c}(e^\tau,j_0,j)|  + |P^{c}(e^\tau,j_0,j)-\Rc^c(j)\kappa(e^\tau)^{-j_0}|.
		\end{multline*}  
		Therefore, \eqref{inPgrcprès} is a direct consequence from \eqref{inPrcprès} and the mean value inequality.
		
		We now consider $\varepsilon\in]0,\varepsilon^\star_0[$. The set 
		$$U_\varepsilon:=  \lc\tau\in\C,\Re(\tau)\in[0,\pi],\Im(\tau)\in[-\pi,\pi]\rc \backslash B_\varepsilon(0)$$
		is compact. Furthermore, for all $\tau_0 \in U_\varepsilon$, since $e^{\tau_0}\in\overline{\Uc}\backslash\lc1\rc$, we have thanks to Lemma \ref{lem:GS_loin} the existence of a radius $\delta>0$ and two positive constants $C,c$ such that $\tau\mapsto \Rg(\tau,j_0,j)$ is holomorphically defined on $B_\delta(\tau_0)$ and 
		$$\forall \tau\in B_\delta(\tau_0), \forall j_0,j\in\N\backslash\lc0\rc, \quad \left|\Rg(\tau,j_0,j)\right| \leq Ce^{-c(j+j_0)}.$$
		Using a compactness argument, we find a width $\eta_\varepsilon\in]0,\varepsilon[$ such that for all $j_0,j\in\N\backslash\lc0\rc$ the function $\tau \mapsto \Rg(\tau,j_0,j)$ is holomorphically defined on $\Omega_\varepsilon$ and there exist two positive constants $C,c>0$ such that \eqref{inRgloin} is verified.
	\end{proof}
	
	The following lemma gives us bounds on the real part of the functions $\varpi$ and $\varphi$ that will be useful later on, for instance when using \eqref{inPgrcprès}.
	
	\begin{lemma}\label{lem:BorneVarphiVarpi}
		There exist a radius $\varepsilon^\star_1\in]0,\varepsilon^\star_0[$ and two positive constants $A_R,A_I$ such that 
		\begin{align}
			\forall \tau\in \C,\quad & \alpha \Re(\varphi(\tau)) \leq -\Re(\tau)+A_R\Re(\tau)^{2\mu}-A_I\Im(\tau)^{2\mu},\label{estVarphi}\\
			\forall \tau\in B_{\varepsilon^\star_1}(0),\quad &\alpha \Re(\varpi(\tau)) +|\alpha| |\xi(\tau)\tau^{2\mu+1}|\leq -\Re(\tau)+A_R\Re(\tau)^{2\mu}-A_I\Im(\tau)^{2\mu}.\label{estVarpiReste}
		\end{align}
	\end{lemma}
	
	\begin{proof}
		We start with the proof of \eqref{estVarphi}. Because of Young's inequality, for $l\in\lc1,\ppp,2\mu-1\rc$, we have that for all $\delta>0$, there exists a constant $C_\delta>0$ such that for all $\tau \in \C$ 
		$$|\Re(\tau)|^l|\Im(\tau)|^{2\mu-l}\leq \delta \Im(\tau)^{2\mu} + C_\delta \Re(\tau)^{2\mu}.$$
		Furthermore, we have that 
		$$\alpha\Re(\varphi(\tau))  =-\Re(\tau) + (-1)^{\mu+1}\left(\frac{\Re(\beta)}{\alpha^{2\mu}}\Re(\tau^{2\mu}) - \frac{\Im(\beta)}{\alpha^{2\mu}}\Im(\tau^{2\mu})\right).$$
		Then, for $\delta>0$, there exists $C_\delta>0$ such that 
		\begin{equation}\label{in:Interm_lem:BorneVarphiVarpi}
			\alpha\Re(\varphi(\tau))  \leq -\Re(\tau) +\Re(\tau)^{2\mu} \left(\frac{\Re(\beta)}{\alpha^{2\mu}}+C_\delta\right)+ \Im(\tau)^{2\mu} \left(-\frac{\Re(\beta)}{\alpha^{2\mu}}+\delta\right).
		\end{equation}	
		Therefore, by taking $\delta$ small enough, we can end the proof of inequality \eqref{estVarphi}. 
		
		We will now prove inequality \eqref{estVarpiReste}. Using \eqref{decVarpi}, we have for $\tau\in B_{\varepsilon_\star}(0)$
		\begin{equation}\label{in:Interm2_lem:BorneVarphiVarpi}
			\alpha\Re(\varpi(\tau)) +|\alpha||\xi(\tau)\tau^{2\mu+1}| \leq \alpha\Re(\varphi(\tau))+2|\alpha||\xi(\tau)\tau^{2\mu+1}| .
		\end{equation}
		If we fix a radius $\varepsilon\in]0,\varepsilon_0^\star[$, the function $\xi$ is bounded by some constant $\tilde{C}>0$ on $B_\varepsilon(0)$.	Furthermore, we know there exist two constants $c_1,c_2>0$ such that 
		$$\forall \tau \in \C,\quad |\tau|^{2\mu}\leq c_1\Re(\tau)^{2\mu} + c_2 \Im(\tau)^{2\mu}.$$
		Thus, using \eqref{in:Interm_lem:BorneVarphiVarpi} and \eqref{in:Interm2_lem:BorneVarphiVarpi}, for all radii $\varepsilon_1^\star\in]0,\varepsilon]$ and for all $\delta >0$, there exists a constant $C_\delta>0$ such that
		\begin{multline*}
			\forall \tau\in B_{\varepsilon_1^\star}(1),\quad \alpha\Re(\varpi(\tau)) +|\alpha||\xi(\tau)\tau^{2\mu+1}|
			\\ \leq -\Re(\tau) +\Re(\tau)^{2\mu} \left(\frac{\Re(\beta)}{\alpha^{2\mu}}+C_\delta+2\alpha c_1\tilde{C}\varepsilon_1^\star\right)+ \Im(\tau)^{2\mu} \left(-\frac{\Re(\beta)}{\alpha^{2\mu}}+\delta+2\alpha c_2\tilde{C}\varepsilon_1^\star\right).
		\end{multline*}
		Taking $\delta$ and $\varepsilon^\star_1$ small enough allows us to prove \eqref{estVarpiReste}. 
	\end{proof}
	
	\textbf{Choice of the radius $\varepsilon$ and of the width $\eta$}
	\vspace{0.5cm}
	
	We will now introduce a radius $\varepsilon>0$ and a width $\eta>0$ which will satisfy a list of conditions. Those conditions will be used throughout the proof and are centralized here in order to fix the notations.
	
	First, we fix a choice of radius $\varepsilon\in\left]0,\min\left(\varepsilon_1^\star,\left(\frac{1}{2\mu A_R}\right)^\frac{1}{2\mu-1}\right)\right[$ where the radius $\varepsilon_1^\star$ is defined in Lemma \ref{lem:BorneVarphiVarpi}. This choice for $\varepsilon$ will allow us to use the results of Lemmas \ref{lem:Rg_Mero_and_esti} and \ref{lem:BorneVarphiVarpi}. Furthermore, if we introduce the function 
	\begin{equation}\label{defPsi}
		\begin{array}{cccc}
			\Psi: & \R & \rightarrow & \R \\ & \tau_p & \mapsto & \tau_p-A_R {\tau_p}^{2\mu}
		\end{array}
	\end{equation}
	which we will use to define a family of parameterized curve in Section \ref{subsec:Int_path}, then the function $\Psi$ is continuous and strictly increasing on $\left]-\infty,\varepsilon\right]$.
	
	We now introduce the function
	\begin{equation}\label{defreps}
		\begin{array}{cccc}
			r_\varepsilon: & ]0,\varepsilon[& \rightarrow & \R \\
			& \eta & \mapsto & \sqrt{\varepsilon^2-\eta^2}
		\end{array}
	\end{equation}
	which serves to define the extremities of the curve $-\eta+i\R\cap B_\varepsilon(0)$. We recall that the width $\eta_\varepsilon$ is defined in Lemma \ref{lem:Rg_Mero_and_esti}. We claim that there exists a width $\eta\in]0,\eta_\varepsilon[$ that we fix for the rest of the paper such that:
	\begin{itemize}
		\item The following inequality is satisfied:
		\begin{equation}\label{condEta}
			\frac{\eta}{2}>A_R\eta^{2\mu}.
		\end{equation}
		\item There exists a radius $\varepsilon_\#\in]0,\varepsilon[$ such that if we define 
		\begin{equation}\label{defIextr}
			l_{extr}:= \left(\frac{\Psi(\varepsilon_\#)-\Psi(-\eta)}{A_I}\right)^\frac{1}{2\mu},
		\end{equation}
		then $-\eta+il_{extr}\in B_{\varepsilon}(0)$.
		\item For all $n,j_0\in\N^*$ which verify $-\frac{n\alpha}{2}\leq j_0\leq np$, we have
		\begin{equation}\label{condEta2}
			\left(-\frac{n\alpha}{j_0}-1\right)(-\eta)+A_R(-\eta)^{2\mu}\leq \frac{A_I}{2}r_\varepsilon(\eta)^{2\mu}.
		\end{equation}
	\end{itemize}

	We introduce the paths $\Gamma_{in}$, $\Gamma_{out}$, $\Gamma$ and $\Gamma_{\eta,in}$ that are represented on Figure \ref{figGamma} and are defined as
	\begin{align*}
		\Gamma_{out}& := [-\eta-i\pi,-\eta-ir_\varepsilon(\eta)]\cup [-\eta+ir_\varepsilon(\eta),-\eta+i\pi], \\
		\Gamma_{in}& := \left[-\eta-ir_\varepsilon(\eta),\frac{\varepsilon}{2}\right]\cup \left[\frac{\varepsilon}{2},-\eta+ir_\varepsilon(\eta)\right],\\
		\Gamma& := \Gamma_{in}\cup \Gamma_{out},\\
		\Gamma_{\eta,in}& := \left[-\eta-ir_\varepsilon(\eta),-\eta+ir_\varepsilon(\eta)\right].
	\end{align*}
	
	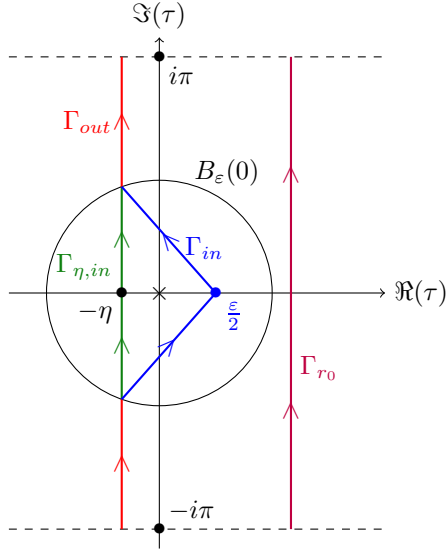
\begin{figure}
		\begin{center}
			\begin{tikzpicture}[scale=1]
				\draw[->] (-2,0) -- (3,0) node[right] {$\Re(\tau)$};
				\draw[->] (0,-3.4) -- (0,3.4) node[above] {$\Im(\tau)$};
				\draw[dashed] (-2,pi) -- (3,pi);
				\draw[dashed] (-2,-pi) -- (3,-pi);
				\draw[dartmouthgreen,thick] (-0.5,{-sqrt(1.5^2-0.5^2)}) -- (-0.5,{sqrt(1.5^2-0.5^2)}) node[near start, sloped] {$>$} node[near end, sloped] {$>$} ;
				\draw[dartmouthgreen] (-0.5,0.35) node[left] {$\Gamma_{\eta,in}$};
				\draw (0,pi) node {$\bullet$} node[below right] {$i\pi$};
				\draw (0,-pi) node {$\bullet$} node[above right] {$-i\pi$};
				\draw (-0.5,0) node {$\bullet$} node[below left] {$-\eta$};				
				\draw (0,0) node {$\times$} circle (1.5);
				\draw (0.9,1.6)  node {$B_\varepsilon(0)$};
				\draw[red,thick] (-0.5,-pi) -- (-0.5,{-sqrt(1.5^2-0.5^2)}) node[midway, sloped] {$>$};
				\draw[red,thick] (-0.5,{sqrt(1.5^2-0.5^2)}) -- (-0.5,pi) node[midway, left] {$\Gamma_{out}$} node[midway, sloped] {$>$};
				\draw[blue] (0.75,0) node {$\bullet$} node[below right] {$\frac{\varepsilon}{2}$};
				\draw[thick,blue] (-0.5,{-sqrt(1.5^2-0.5^2)}) -- (0.75,0) node[midway, sloped] {$>$};
				\draw[thick,blue] (0.75,0) -- (-0.5,{sqrt(1.5^2-0.5^2)})  node[midway, sloped] {$<$};
				\draw[blue] (0.6,0.6) node {$\Gamma_{in}$};
				\draw[thick, purple] (1.75,-pi)-- (1.75,pi) node[near start, sloped] {$>$} node[near end, sloped] {$>$} ;
				\draw[purple] (1.75,-1) node[right] {$\Gamma_{r_0}$} ;
			\end{tikzpicture}
			\caption{A representation of the path $\Gamma_{r_0}$ (in purple), $\Gamma_{in}$ (in blue), $\Gamma_{out}$ (in red), $\Gamma:=\Gamma_{out}\cup \Gamma_{in}$ and $\Gamma_{\eta,in}$ (in green)}
			\label{figGamma}
		\end{center}
	\end{figure}
	
	We observe that those paths lie in $\Omega_\varepsilon\cup B_{\varepsilon}(0)$. Noticing the "$2i\pi$-periodicity" of the function $\Rg(\cdot,j_0,j)$, Cauchy's formula implies that equality \eqref{egGTGSred} can be rewritten as 
	\begin{align}\label{egGTGSred2}
		\begin{split}\forall n\in\N, \forall j_0,j\in\N\backslash\lc0\rc, \quad \Gc(n,j_0,j) -\widetilde{\Gc}(n,j-j_0) & = \frac{1}{2i\pi}\int_{\Gamma}e^{n\tau} \Rg(\tau,j_0,j)d\tau\\
			&= \frac{1}{2i\pi}(T^{out}+T^u+T^c)\end{split}
	\end{align}
	where
	$$T^{out} :=\int_{\Gamma_{out}}e^{n\tau}\Rg(\tau,j_0,j)d\tau,\quad T^u :=\int_{\Gamma_{in}}e^{n\tau}\frac{\Pg^u(\tau,j_0,j)}{\tau}d\tau,\quad  T^c :=\int_{\Gamma_{in}}e^{n\tau}\frac{\Pg^c(\tau,j_0,j)}{\tau}d\tau.$$	
	
	Thus, to prove \eqref{in:Err} when $j_0\leq np$, we need to estimate the terms $T^{out}$, $T^u$ and $T^c$. We start by proving estimates for $T^{out}$ and $T^u$.
	
	\begin{prop}[Estimate on $T^{out}$ and $T^u$]\label{prop:est_Tout_Tu}
		There exist two constants $C,c>0$ such that
		$$\forall n\in\N, \forall j_0,j\in \N\backslash\lc0\rc, \quad |T^{out}|\leq C \exp(-n\eta-c(j+j_0)).$$
		and
		$$\forall n\in\N, \forall j_0,j\in \N\backslash\lc0\rc, \quad |T^{u}-2i\pi\Rc^u(j_0,j)|\leq C \exp(-n\eta-c(j+j_0)).$$
	\end{prop}
	
	\begin{proof}
		We consider $n\in \N$ and $j_0,j\in \N\backslash\lc0\rc$. 
		
		$\bullet$ Since $\Gamma_{out}$ lies within $\Omega_\varepsilon$, using \eqref{inRgloin}, there exists a positive constant $c$ such that
		$$|T^{out}|\lesssim \int_{\Gamma_{out}}\exp(-n\eta)\exp(-cj-cj_0)|d\tau|\lesssim \exp(-n\eta)\exp(-cj-cj_0).$$
		
		$\bullet$ Using the residue theorem and the definition \eqref{defRu} of $\Rc^u(j_0,j)$, we have that
		$$T^{u} = \int_{\Gamma_{in}}e^{n\tau}\frac{\Pg^u(\tau,j_0,j)}{\tau}d\tau =2i\pi\Rc^u(j_0,j)+\int_{\Gamma_{\eta,in}}e^{n\tau}\frac{\Pg^u(\tau,j_0,j)}{\tau}d\tau.$$
		Thus, using \eqref{inPgruprès}, there exists a constant $c>0$ such that 
		$$|T^u-2i\pi\Rc^u(j_0,j)|\leq \int_{\Gamma_{\eta,in}}e^{n\Re(\tau)}\frac{|\Pg^u(\tau,j_0,j)|}{|\tau|}|d\tau|\lesssim \exp(-n\eta-c(j+j_0)).$$
	\end{proof}
	
	Let us observe that the exponential estimates on the terms $T^{out}$ and $T^u$ we just proved can be altered to recover similar generalized Gaussian estimates as in \eqref{in:Err} since there exists a constant $c>0$ such that for all $n,j_0\in\N\backslash\lc0\rc$ which verify $j_0\leq np$, we have 
	$$-n\leq -c\left(\frac{|j_0+n\alpha|}{n^\frac{1}{2\mu}}\right)^\frac{2\mu}{2\mu-1}.$$
	The same kind of exponential bounds will be encountered regularly in the rest of the proof and the same reasoning will allow us to obtain generalized Gaussian estimates.
	
	Now that have found estimates for the two terms $T^{out}$ and $T^u$, there just remains to study the term $T^c$. Section \ref{subsec:J0_Small} will be dedicated to proving estimates on $T^c$ in the case \textbf{II} when $j_0$ small with regard to $-n\alpha$. Finally, Section \ref{subsec:J0_Close} will tackle the study of the term $T^c$ in the case \textbf{III} when $j_0$ is close to $-n\alpha$. 
	
	\subsection{Case \textbf{II}: Estimate for $T^c$ for $j_0$ small with regard to $-n\alpha$}\label{subsec:J0_Small}
	
	The main goal of this section is to prove the following proposition.
	\begin{prop}\label{prop:est_Tc_J0_Small}
		There exist two positive constants $C,c$ such that for all $n\in\N$ and $j,j_0\in\N\backslash\lc0\rc$ such that $j_0<-\frac{n\alpha}{2}$, we have
		$$\left|T^c-2i\pi E_{2\mu}^\beta\left(\frac{j_0+n\alpha}{n^\frac{1}{2\mu}}\right)\Rc^c(j)\right|\leq\frac{Ce^{-cj}}{n^\frac{1}{2\mu}}\exp\left(-c\left(\frac{|n\alpha+j_0|}{n^\frac{1}{2\mu}}\right)^\frac{2\mu}{2\mu-1}\right).$$
	\end{prop}
	
	Combining \eqref{egGTGSred2}, Propositions \ref{prop:est_Tout_Tu} and \ref{prop:est_Tc_J0_Small}, we then prove that there exist two positive constants $C,c$ such that for all $n\in\N$ and $j_0,j\in\N\backslash\lc0\rc$ which verify $j_0<-\frac{n\alpha}{2}$:
	\begin{equation}\label{inThSmall}
		|\mathrm{Err}(n,j_0,j)|\leq \frac{Ce^{-cj}}{n^\frac{1}{2\mu}}\exp\left(-c\left(\frac{|n\alpha+j_0|}{n^\frac{1}{2\mu}}\right)^\frac{2\mu}{2\mu-1}\right).
	\end{equation}
	Thus, \eqref{in:Err} is proved when $j_0$ is small compared to $n$ (Case \textbf{II}).
	
	\begin{proof}
		
		\textbf{\underline{Step 1:}} We decompose $T^c$ in two parts:
		\begin{equation}\label{eg:decompoTc}
			T^c=T^c_1+T^c_2
		\end{equation}
		where
		\begin{align*}
			T^c_1  &:=\int_{\Gamma_{in}}e^{n\tau}\frac{\Pg^c(\tau,j_0,j)-\Rc^c(j)\exp(-j_0\varpi(\tau))}{\tau}d\tau, \\
			T^c_2 &:= \Rc^c(j)\int_{\Gamma_{in}}\frac{\exp(n\tau-j_0\varpi(\tau))}{\tau}d\tau.
		\end{align*}
		We will estimate separately both terms in order to prove the existence of two positive constants $C,c$ such that for all $n\in\N$ and $j,j_0\in\N\backslash\lc0\rc$ which verify $j_0<-n\frac{\alpha}{2}$, we have
		\begin{equation}\label{lemTcSmallIn1}
			|T^c-2i\pi \Rc^c(j)|\leq C e^{-c(j+n)}.
		\end{equation}
		
		$\bullet$ Inequality \eqref{inPgrcprès} implies that the function 
		$$\tau \in B_{\varepsilon}(0)\backslash\lc0\rc\mapsto e^{n\tau}\frac{\Pg^c(\tau,j_0,j)-\Rc^c(j)\exp(-j_0\varpi(\tau))}{\tau}$$
		can be holomorphically extended on $B_{\varepsilon}(0)$. Using Cauchy's formula, we then have
		$$T^c_1 =\int_{\Gamma_{\eta,in}}e^{n\tau}\frac{\Pg^c(\tau,j_0,j)-\Rc^c(j)\exp(-j_0\varpi(\tau))}{\tau}d\tau .$$
		Using \eqref{inPgrcprès}, there exist two positive constants $C,c$ such that
		$$|T^c_1| \leq Ce^{-cj} \int_{\Gamma_{\eta,in}}\exp(n\Re(\tau)-j_0\Re(\varpi(\tau)))|d\tau| =Ce^{-cj-n\eta} \int_{\Gamma_{\eta,in}}\exp(-j_0\Re(\varpi(\tau)))|d\tau| .$$
		For $\tau\in\Gamma_{\eta,in}$, using \eqref{estVarpiReste} and the fact that $j_0<-\frac{n\alpha}{2}$, we have that
		\begin{equation}\label{lemTcSmallIn}
			-j_0\Re(\varpi(\tau))\leq \frac{-j_0}{\alpha}\left(-\Re(\tau)+A_R\Re(\tau)^{2\mu}-A_I\Im(\tau)^{2\mu}\right) \leq  \frac{n}{2}(\eta + A_R\eta^{2\mu}).
		\end{equation}
		Thus, there exists a new constant $C>0$ such that
		$$|T^c_1| \leq C\exp\left(-cj-n\left(\frac{\eta}{2}-A_R\eta^{2\mu}\right)\right) .$$
		Therefore, the condition \eqref{condEta} on $\eta$ implies that there exist two positive constants $C,c$ such that 
		\begin{equation}\label{in:propSmall1}
			\forall n\in\N, \forall j_0,j\in\N\backslash\lc0\rc, \quad j_0<-\frac{n\alpha}{2}\Rightarrow |T^c_1|\leq C\exp(-c(j+n)).
		\end{equation}
	
		$\bullet$ Using the residue theorem, we have
		$$T^c_2 -2i\pi\Rc^c(j)=\Rc^c(j)\int_{\Gamma_{\eta,in}}\frac{\exp(n\tau-j_0\varpi(\tau))}{\tau}d\tau.$$
		Therefore, using \eqref{inPrcprès} to exponentially bound $\Rc^c(j)$, there exist two positive constants $C,c$ such that
		$$|T^c_2 -2i\pi\Rc^c(j)| \leq Ce^{-cj} \int_{\Gamma_{\eta,in}}\frac{\exp(n\Re(\tau)-j_0\Re(\varpi(\tau)))}{|\tau|}|d\tau| \leq C\frac{\exp(-cj-n\eta)}{\eta} \int_{\Gamma_{\eta,in}}\exp(-j_0\Re(\varpi(\tau)))|d\tau| .$$
		Thus, using \eqref{lemTcSmallIn}, there exists a new constant $C>0$ independent from $n$, $j_0$ and $j$ such that
		$$|T^c_2 -2i\pi \Rc^c(j)| \leq C\exp\left(-cj-n\left(\frac{\eta}{2}-A_R\eta^{2\mu}\right)\right) .$$
		Therefore, the condition \eqref{condEta} on $\eta$ implies that there exist two positive constants $C,c$ such that 
		\begin{equation}\label{in:propSmall2}
			\forall n\in\N, \forall j_0,j\in\N\backslash\lc0\rc, \quad j_0<-\frac{n\alpha}{2}\Rightarrow |T^c_2-2i\pi\Rc^c(j)|\leq C\exp(-c(j+n)).
		\end{equation}
		Using \eqref{eg:decompoTc}, \eqref{in:propSmall1} and \eqref{in:propSmall2}, we conclude the proof of \eqref{lemTcSmallIn1}.
		
		\textbf{\underline{Step 2:}} Since we have
		$$\frac{j_0+n\alpha}{n^\frac{1}{2\mu}}\leq \frac{\alpha}{2}n^\frac{2\mu-1}{2\mu}<0,$$
		and the function $x\in]-\infty,\frac{\alpha}{2}]\mapsto |x|^\frac{1}{2\mu-1}\exp\left(-\frac{c}{2}|x|^\frac{2\mu}{2\mu-1}\right)$ is bounded where $c$ is the positive constant in \eqref{inE-}, we conclude using \eqref{inE-} that there exist a positive constants $\tilde{C}$ such that for all $n,j_0\in \N\backslash\lc0\rc$ such that $j_0<-\frac{n\alpha}{2}$, we have
		\begin{equation}\label{inAc-}
			\left|1-E_{2\mu}^\beta\left(\frac{j_0+n\alpha}{n^\frac{1}{2\mu}}\right)\right|\leq \frac{\tilde{C}}{n^\frac{1}{2\mu}}\exp\left(-\frac{c}{2}\left(\frac{|j_0+n\alpha|}{n^\frac{1}{2\mu}}\right)^\frac{2\mu}{2\mu-1}\right).
		\end{equation}
		Using \eqref{lemTcSmallIn1}, \eqref{inAc-} and the estimate \eqref{inPrcprès} to exponentially bound $\Rc^c(j)$, we conclude the proof of Proposition \ref{prop:est_Tc_J0_Small}.
	\end{proof}
	
	\subsection{Case \textbf{III}: Estimate for $T^{c}$ for $j_0$ close to $-n\alpha$}\label{subsec:J0_Close}
	
	The goal of this section will be to study what happens when $j_0$ is close to $-n\alpha$ and to prove the following proposition:
	\begin{prop}\label{prop:est_Tc_J0_Close}
		There exist two constants $C,c>0$ such that for all $n,j_0,j\in\N\backslash\lc0\rc$ such that $j_0\in\left[-\frac{n\alpha}{2},np\right]$, we have
		$$\left|T^c - 2i\pi E_{2\mu}^\beta\left(\frac{j_0+n\alpha}{n^\frac{1}{2\mu}}\right)\Rc^c(j)\right|\leq \frac{Ce^{-cj}}{n^\frac{1}{2\mu}}\exp\left(-c\left(\frac{|n\alpha+j_0|}{n^\frac{1}{2\mu}}\right)^\frac{2\mu}{2\mu-1}\right).$$
	\end{prop}
	
	Combining \eqref{egGTGSred2} and Propositions \ref{prop:est_Tout_Tu} and \ref{prop:est_Tc_J0_Close}, we prove that there exist two constants $C,c>0$ such that for $j_0\in\left[-\frac{n\alpha}{2},np\right]$, we have
	\begin{equation}\label{inThMedium}
		\left|\mathrm{Err}(n,j_0,j)\right|\leq \frac{Ce^{-cj}}{n^\frac{1}{2\mu}}\exp\left(-c\left(\frac{|n\alpha+j_0|}{n^\frac{1}{2\mu}}\right)^\frac{2\mu}{2\mu-1}\right).
	\end{equation}
	Consequently, using the result of Section \ref{subsec:J0_Large}, \eqref{inThSmall} and \eqref{inThMedium}, we can conclude the proof of Theorem \ref{th:Green}.
	
	Therefore, there just remains to prove Proposition \ref{prop:est_Tc_J0_Close}. This part of the article requires the finest attention since the limiting estimates of Theorem \ref{th:Green} occur here. To prove Proposition \ref{prop:est_Tc_J0_Close}, we will decompose $T^c$ in three parts:
	\begin{equation}\label{decompoTc}
		T^c=T^c_1+T^c_2+T^c_{princ}
	\end{equation}
	where
	\begin{align}
		T^c_1  &:=\int_{\Gamma_{in}}e^{n\tau}\frac{\Pg^c(\tau,j_0,j)-\Rc^c(j)\exp(-j_0\varpi(\tau))}{\tau}d\tau, \label{defTc1}\\
		T^c_2 & :=\Rc^c(j)\int_{\Gamma_{in}}e^{n\tau}\frac{e^{-j_0\varpi(\tau)}-e^{-j_0\varphi(\tau)}}{\tau}d\tau,\label{defTc2}\\
		T^c_{princ}  &:=\Rc^c(j)\int_{\Gamma_{in}}\frac{e^{n\tau}e^{-j_0\varphi(\tau)}}{\tau}d\tau.	\label{defTcprinc}	
	\end{align}	
	
	We now summarize the method of proof of Proposition \ref{prop:est_Tc_J0_Close}. In Section \ref{subsec:Int_path}, we introduce a family of integration paths $\Gamma_p$ that are fundamental to optimally use the estimates on the function $\Rg(\tau,j_0,j)$ we proved in Section \ref{sec:GS}. We then prove in Section \ref{subsec:est_Tc12} estimates on the terms $T^c_1$ and $T^c_2$ respectively in Propositions \ref{prop:est_Tc1} and \ref{lem:est_Tc2}. Finally, Section \ref{subsec:est_Tcprinc} is dedicated to the analysis of the term $T^c_{princ}$. We will change the integration path in the term
	$$\int_{\Gamma_{in}}\frac{e^{n\tau}e^{-j_0\varphi(\tau)}}{\tau}d\tau$$
	in order to compare it with $E_{2\mu}^\beta\left(\frac{j_0+n\alpha}{n^\frac{1}{2\mu}}\right)$ (see Proposition \ref{prop:est_Tcprinc}).

	\subsubsection{Choice of integration path}\label{subsec:Int_path}
	
	We will now follow a strategy developed in \cite{ZH}, which has also been used in \cite{Godillon,CF,CF2,Coeuret}, and introduce a family of parameterized curves.
	
	We recall that we introduced in \eqref{defPsi} the function $\Psi$ defined by 
	$$\forall \tau_p\in\R, \quad \Psi(\tau_p):= \tau_p-A_R{\tau_p}^{2\mu}.$$
	and that we chose $\varepsilon$ small enough so that the function $\Psi$ is continuous and strictly increasing on $]-\infty,\varepsilon]$. We can therefore introduce for $\tau_p\in[-\eta,\varepsilon]$ the curve $\Gamma_p$ defined by
	$$\Gamma_{p} := \lc \tau\in \C, -\eta\leq \Re(\tau)\leq \tau_p, \quad \Re(\tau) - A_R \Re(\tau)^{2\mu} +  A_I \Im(\tau)^{2\mu}= \Psi(\tau_p)\rc.$$
	It is a symmetric curve with respect to the axis $\R$ which intersects this axis on the point $\tau_p$. If we introduce $\ell_{p}= \left(\frac{\Psi(\tau_p)-\Psi(-\eta)}{A_I}\right)^\frac{1}{2\mu}$, then $-\eta +i\ell_{p}$ and $-\eta -i\ell_{p}$ are the end points of $\Gamma_{p}$. We can also introduce a parametrization of this curve by defining $\gamma_{p}:[-\ell_{p}, \ell_{p}]\rightarrow \C$ such that 
	\begin{equation}\label{param}
		\forall \tau_p\in\left[-\eta,\varepsilon\right], \forall t\in[-\ell_{p},\ell_{p}],\quad \Im(\gamma_{p}(t))=t, \quad \Re(\gamma_{p}(t))=h_{p}(t):=\Psi^{-1}\left(\Psi(\tau_p)-A_It^{2\mu}\right).
	\end{equation}
	
	The above parametrization immediately yields that there exists a constant $C>0$ such that 
	\begin{equation}\label{hp}
		\forall \tau_p \in[-\eta,\varepsilon], \forall t \in[-\ell_{p},\ell_{p}], \quad |h_{p}^\prime(t)|\leq C.
	\end{equation}
	Also, there exists a constant $c_\star>0$ such that 
	\begin{equation}
		\forall \tau_p\in[-\eta,\varepsilon], 	\forall \tau \in \Gamma_{p}, \quad \Re(\tau)-\tau_p\leq -c_\star \Im(\tau)^{2\mu}. \label{ine_Re}
	\end{equation}
	
	We introduce those integration paths $\Gamma_{p}$ because they allow us to use optimally the inequalities \eqref{estVarphi} and \eqref{estVarpiReste}. For example, if we seek to bound $e^{n\tau-j_0\varpi(\tau)}$ for $\tau\in \Gamma_{p}\cap B_{\varepsilon}(0)$, it follows from the equality $\mathrm{sgn}(-j_0)=\mathrm{sgn}(\alpha)$ and the inequalities \eqref{estVarpiReste} and \eqref{ine_Re} that
	\begin{align}
		\begin{split}
			n\Re(\tau)-j_0\Re(\varpi(\tau))& \leq n\Re(\tau)+\frac{j_0}{\alpha} \left(\Re(\tau)-A_R\Re(\tau)^{2\mu}+A_I\Im(\tau)^{2\mu}\right)\\
			& \leq -nc_\star \Im(\tau)^{2\mu}+ \left(\frac{j_0}{\alpha}+n\right)\tau_p -\frac{j_0}{\alpha}A_R\tau_p^{2\mu}.
		\end{split}\label{estClas}
	\end{align}
	Such calculations will happen regularly in the following proof. There remains to make an appropriate choice of $\tau_p$ depending on $n$ and $j_0$ that minimizes the right-hand side of the inequality \eqref{estClas} whilst the paths $\Gamma_{p}$ remain within the ball $B_{\varepsilon}(0)$. We recall that when we fixed our choice of width $\eta$, we defined a radius $\varepsilon_\#\in]0,\varepsilon[$ such that $-\eta+il_{extr}\in B_{\varepsilon}(0)$ where the real number $l_{extr}$ is defined by \eqref{defIextr}. This implies that the curve $\Gamma_{p}$ associated with $\tau_p=\varepsilon_\#$ intersects the axis $-\eta+i\R$ within $B_\varepsilon(0)$. We let 
	$$\zeta=\frac{-j_0-n\alpha}{2\mu n}, \quad \gamma=\frac{-j_0A_R}{n}, \quad \rho\left(\frac{\zeta}{\gamma}\right)=\mathrm{sgn}(\zeta)\left(\frac{|\zeta|}{\gamma}\right)^\frac{1}{2\mu-1}.$$
	Inequality \eqref{estClas} thus becomes
	\begin{equation}\label{estClas2}
		n\Re(\tau)-j_0\Re(\varpi(\tau))\leq -nc_\star \Im(\tau)^{2\mu}-\frac{n}{\alpha}(2\mu\zeta\tau_p-\gamma \tau_p^{2\mu}).
	\end{equation}
	Our limiting estimates will come from the case where $\zeta$ is close to $0$. We observe that the condition $j_0\in\left[-\frac{n\alpha}{2},np\right]$ implies 
	\begin{equation}
		-pA_R\leq \gamma \leq \frac{\alpha}{2}A_R.\label{ineg_gamma}
	\end{equation}
	
	Then, we take 
	$$\tau_p:=\lc\begin{array}{ccc}
		\rho\left(\frac{\zeta}{\gamma}\right), & \text{ if }\rho\left(\frac{\zeta}{\gamma}\right)\in[-\frac{\eta}{2},\varepsilon_\#],& \text{(Case \textbf{A})}\\
		\varepsilon_\#, & \text{ if }\rho\left(\frac{\zeta}{\gamma}\right)>\varepsilon_\#, &\text{(Case \textbf{B})}\\
		-\frac{\eta}{2}, & \text{ if }\rho\left(\frac{\zeta}{\gamma}\right)<-\frac{\eta}{2}.&\text{(Case \textbf{C})}\\
	\end{array}\right.$$
	The case \textbf{A} corresponds to the choice to minimize the right-hand side of \eqref{estClas2} since $\rho\left(\frac{\zeta}{\gamma}\right)$ is the unique real root of the polynomial
	$$\gamma x^{2\mu-1}=\zeta.$$
	The cases \textbf{B} and \textbf{C} allow the path $\Gamma_{p}$ to stay within $B_\varepsilon(0)$.
	
	We now define the paths represented on Figure \ref{chem}:
	\begin{align*}
		\Gamma_{p,res}:=&\lc-\eta +it,\quad t\in[-r_{\varepsilon}(\eta),-\ell_{p}]\cup[\ell_{p},r_{\varepsilon}(\eta)]\rc,\\
		\Gamma_{p,in}:=&\Gamma_{p}\cup\Gamma_{p,res},
	\end{align*}
	where the function $r_\varepsilon$ is defined by \eqref{defreps}.
	
	\begin{figure}
		\begin{center}
			\begin{tikzpicture}[scale=0.8]
				\draw[->] (-2,0) -- (3,0) node[right] {$\Re(\tau)$};
				\draw[->] (0,-3.4) -- (0,3.4) node[above] {$\Im(\tau)$};
				\draw[dashed] (-2,pi) -- (3,pi);
				\draw[dashed] (-2,-pi) -- (3,-pi);
				\draw (0,pi) node {$\bullet$} node[below right] {$i\pi$};
				\draw (0,-pi) node {$\bullet$} node[above right] {$-i\pi$};
				\draw (-0.5,0) node {$\bullet$} node[below left] {$-\eta$};
				\draw (0,0) node {$\times$} circle (1.5);
				\draw (1.5,0) node[above right] {$B_\varepsilon(0)$};
				\draw[blue] (1,0) node {$\bullet$} ;
				\draw[blue] (1.25,-0.25) node {$\tau_p$};
				\draw[dartmouthgreen,thick] (-0.5,{-sqrt(1.5^2-0.5^2)}) -- (-0.5,{-sqrt((3/4+0.5+1/16)/3)});
				\draw[dartmouthgreen,thick] (-0.5,{sqrt((3/4+0.5+1/16)/3)}) -- (-0.5,{sqrt(1.5^2-0.5^2)}) ;
				\draw[dartmouthgreen] (-1.7,{(sqrt(1.5^2-0.5^2)+sqrt((3/4+0.5+1/16)/3))/2}) node {$\Gamma_{p,res}$};
				\draw[dashed,thick] (-0.5,{-sqrt((3/4+0.5+1/16)/3)}) -- (-0.5,{sqrt((3/4+0.5+1/16)/3)});
				\draw[thick,blue] plot [samples = 100, domain={-sqrt((3/4+0.5+1/16)/3)}:{sqrt((3/4+0.5+1/16)/3)}] ({2*(1-sqrt(1/4+3*abs(\x)^2))},\x) ;
				\draw[blue] (0.6,0.6) node {$\Gamma_{p}$};
			\end{tikzpicture}
			\caption{A representation of the path $\Gamma_{p,in}$. It is composed of $\Gamma_{p,res}$ (in green) and $\Gamma_{p}$ (in blue).}
			\label{chem}
		\end{center}
	\end{figure}
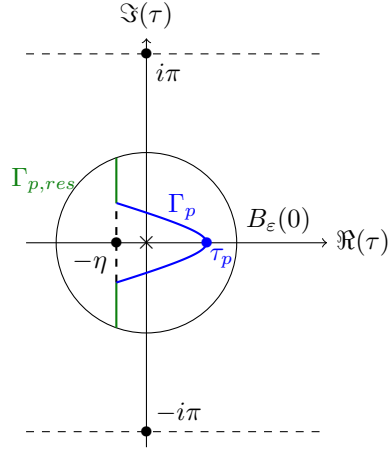
	
	Finally, before we start to determine the estimates on the terms $T^c_1$ and $T^c_2$ in Section \ref{subsec:est_Tc12}, we are going to introduce some inequalities to simplify the redaction.
	
	\begin{lemma}\label{lem:inInterm}
		\begin{itemize}
			\item There exists a constant $C>0$ such that for all $\tau \in B_{\varepsilon}(0)$ and $n,j_0\in\N\backslash\lc0\rc$ such that $j_0\in\left[-\frac{n\alpha}{2},np\right]$, we have 
			\begin{equation}\label{in:DevAsymp}
				\left|e^{n\tau} \left( e^{-j_0\varpi(\tau)} -e^{-j_0\varphi(\tau)}\right)\right| \leq C n|\tau|^{2\mu+1} \exp(n \Re(\tau)-j_0 (\Re(\varpi(\tau)) -|\xi(\tau)(\tau)^{2\mu+1}|)).
			\end{equation}
			\item For $n,j_0\in\N\backslash\lc0\rc$ and $\tau\in \Gamma_{p}$, we have
			
			\begin{itemize}
				\item Case \textbf{A}: $\rho\left(\frac{\zeta}{\gamma}\right) \in \left[-\frac{\eta}{2},\varepsilon_\#\right]$
				\begin{equation}
					n\Re(\tau)-j_0(\Re(\varpi(\tau)) -|\xi(\tau)\tau^{2\mu+1}|) \leq -nc_\star\Im(\tau)^{2\mu} - \frac{n}{\alpha} (2\mu-1)\gamma\left(\frac{|\zeta|}{\gamma}\right)^\frac{2\mu}{2\mu-1}.\label{ineVarpiReste_casA}
				\end{equation}	
				\item Case \textbf{B}: $\rho\left(\frac{\zeta}{\gamma}\right) >\varepsilon_\#$
				\begin{equation}
					n\Re(\tau)-j_0(\Re(\varpi(\tau)) -|\xi(\tau)\tau^{2\mu+1}|) \leq- \frac{n}{2} (2\mu-1)A_R\varepsilon_\#^{2\mu}.\label{ineVarpiReste_casB}
				\end{equation}	
				\item Case \textbf{C}: $\rho\left(\frac{\zeta}{\gamma}\right) <-\frac{\eta}{2}$
				\begin{equation}
					n\Re(\tau)-j_0(\Re(\varpi(\tau)) -|\xi(\tau)\tau^{2\mu+1}|) \leq- \frac{n}{2} (2\mu-1)A_R\left(\frac{\eta}{2}\right)^{2\mu}.\label{ineVarpiReste_casC}
				\end{equation}	
			\end{itemize} 
			
			\item For $n,j_0\in\N\backslash\lc0\rc$ and $\tau\in \Gamma_{p,res}$, we have in all cases (\textbf{A}, \textbf{B} and \textbf{C})
			\begin{equation}
				n\Re(\tau)-j_0(\Re(\varpi(\tau))-|\xi(\tau)\tau^{2\mu+1}|) \leq -n\frac{\eta}{2}.\label{ineVarpiReste_res}
			\end{equation}
		\end{itemize}
	\end{lemma}
	
	The proof of inequalities \eqref{ineVarpiReste_casA}-\eqref{ineVarpiReste_res} mainly rely on the inequalities \eqref{estVarphi} and \eqref{estVarpiReste} and calculations similar as those done to obtain \eqref{estClas2}. For a complete proof of Lemma \ref{lem:inInterm}, we advice the interested reader to look at the proof of \cite[Lemmas 17, 18 and 19]{Coeuret} which prove similar inequalities in the context of the study of the temporal Green's function for the Laurent operator. The notation have intentionally been kept quite similar. The only difference is that the proof in \cite{Coeuret} are usually done with a positive velocity $\alpha$.
	
	\subsubsection{Estimates for $T^c_1$ and $T^c_2$}\label{subsec:est_Tc12}
	
	In this section, we prove generalized Gaussian estimates for the terms $T^c_1$ and $T^c_2$ when $j_0$ is close to $-n\alpha$.
	
	\begin{prop}\label{prop:est_Tc1}
		There exist two positive constants $C,c$ such that
		$$\forall n,j,j_0\in\N^*, \quad j_0\in\left[-\frac{n\alpha}{2},np\right]\Rightarrow |T^c_1|\leq \frac{C}{n^\frac{1}{2\mu}}\exp\left(-cj-c\left(\frac{|j_0+n\alpha|}{n^\frac{1}{2\mu}}\right)^\frac{2\mu}{2\mu-1}\right).$$
	\end{prop}
	
	\begin{proof}
		$\bullet$ Inequality \eqref{inPgrcprès} implies that the function
		$$\tau\in B_{\varepsilon}(0)\backslash\lc0\rc\mapsto e^{n\tau}\frac{\Pg^c(\tau,j_0,j)-\Rc^c(j)\exp(-j_0\varpi(\tau))}{\tau}$$
		can be holomorphically extended on $B_{\varepsilon}(0)$. Therefore, Cauchy's formula implies that
		$$T^c_1 = \int_{\Gamma_{in,p}}e^{n\tau}\frac{\Pg^c(\tau,j_0,j)-\Rc^c(j)\exp(-j_0\varpi(\tau))}{\tau}d\tau.$$
		Using \eqref{inPgrcprès}, there exist two positive constants $C,c$ such that
		\begin{align*}
			|T^c_1|&\leq Ce^{-cj} \int_{\Gamma_{in,p}}\exp\left(n\Re(\tau)-j_0\Re(\varpi(\tau))\right)|d\tau|\\
			&\leq Ce^{-cj} \left(\int_{\Gamma_{p}}\exp\left(n\Re(\tau)-j_0\Re(\varpi(\tau))\right)|d\tau|+ \int_{\Gamma_{res,p}}\exp\left(n\Re(\tau)-j_0\Re(\varpi(\tau))\right)|d\tau|\right).
		\end{align*}
		
		$\bullet$ Using \eqref{ineVarpiReste_res}, we have
		\begin{equation}\label{lemTc1In1}
			\int_{\Gamma_{res,p}}\exp\left(n\Re(\tau)-j_0\Re(\varpi(\tau))\right)|d\tau|\leq 2\pi\exp\left(-n\frac{\eta}{2}\right).
		\end{equation}
		
		$\bullet$ In cases \textbf{B} and \textbf{C}, using \eqref{ineVarpiReste_casB} or \eqref{ineVarpiReste_casC} depending on the case, there exist two constants $C,c>0$ such that
		\begin{equation}\label{lemTc1In2}
			\int_{\Gamma_{p}}\exp\left(n\Re(\tau)-j_0\Re(\varpi(\tau))\right)|d\tau|\leq C\exp\left(-cn\right).
		\end{equation}
		
		$\bullet$ In case \textbf{A}, using \eqref{ineVarpiReste_casA}, we have
		$$\int_{\Gamma_{p}}\exp\left(n\Re(\tau)-j_0\Re(\varpi(\tau))\right)|d\tau|\leq \int_{\Gamma_{p}}\exp\left(-nc_\star\Im(\tau)^{2\mu}\right)|d\tau| \exp\left(- \frac{n}{\alpha} (2\mu-1)\gamma\left(\frac{|\zeta|}{\gamma}\right)^\frac{2\mu}{2\mu-1}\right). $$
		Using the parametrization \eqref{param}, the inequality \eqref{hp} and the change of variables $u=n^\frac{1}{2\mu}t$, we have
		$$\int_{\Gamma_{p}}  e^{-nc_\star\Im(\tau)^{2\mu}}|d\tau| \lesssim \int_{-\ell_{p}}^{\ell_{p}}e^{-nc_\star t^{2\mu}}dt\lesssim \frac{1}{n^\frac{1}{{2\mu}}}.$$
		Thus,
		$$\int_{\Gamma_{p}}\exp\left(n\Re(\tau)-j_0\Re(\varpi(\tau))\right)|d\tau| \lesssim \frac{1}{n^\frac{1}{2 \mu}} \exp\left(- \frac{n}{\alpha} (2\mu-1)\gamma\left(\frac{|\zeta|}{\gamma}\right)^\frac{2\mu}{2\mu-1}\right). $$
		Lastly, the inequality \eqref{ineg_gamma} implies that we have a constant $c>0$ independent from $j_0$ and $n$ such that 
		$$- \frac{n}{\alpha} (2\mu-1)\gamma\left(\frac{|\zeta|}{\gamma}\right)^\frac{2\mu}{2\mu-1} \leq -c\left(\frac{|n\alpha+j_0|}{n^\frac{1}{2\mu}}\right)^\frac{2\mu}{2\mu-1}$$
		so,
		\begin{equation}\label{lemTc1In3}
			\int_{\Gamma_{p}}\exp\left(n\Re(\tau)-j_0\Re(\varpi(\tau))\right)|d\tau|\lesssim \frac{1}{n^\frac{1}{2\mu}} \exp\left(-c \left(\frac{|n\alpha+j_0|}{n^\frac{1}{2\mu}}\right)^\frac{2\mu}{2\mu-1}\right).
		\end{equation}
	
		Combining \eqref{lemTc1In1}-\eqref{lemTc1In3}, we conclude the proof of Proposition \ref{prop:est_Tc1}.
	\end{proof}
	
	\begin{prop}\label{lem:est_Tc2}
		There exist two positive constants $C,c$ such that
		$$\forall n,j,j_0\in\N^*, \quad j_0\in\left[-\frac{n\alpha}{2},np\right]\Rightarrow |T^c_2|\leq \frac{C}{n^\frac{1}{2\mu}}\exp\left(-cj-c\left(\frac{|j_0+n\alpha|}{n^\frac{1}{2\mu}}\right)^\frac{2\mu}{2\mu-1}\right).$$
	\end{prop}
	
	\begin{proof}
		The function 
		$$\tau\in B_{\varepsilon}(0)\backslash\lc0\rc\mapsto e^{n\tau}\frac{e^{-j_0\varpi(\tau)}-e^{-j_0\varphi(\tau)}}{\tau}$$
		can be holomorphically extended on $B_{\varepsilon}(0)$. Therefore, Cauchy's formula implies that
		$$T^c_2 = \Rc^c(j)\int_{\Gamma_{in,p}}e^{n\tau}\frac{e^{-j_0\varpi(\tau)}-e^{-j_0\varphi(\tau)}}{\tau}d\tau.$$
		
		Using \eqref{in:DevAsymp} and \eqref{inPrcprès}, there exist two positive constant $C,c$ such that
		\begin{equation}\label{lemTc2In1}
			|T^c_2| \leq Ce^{-cj}n \int_{\Gamma_{in,p}}|\tau|^{2\mu} \exp(n \Re(\tau)-j_0 (\Re(\varpi(\tau)) -|\xi(\tau)\tau^{2\mu+1}|))|d\tau|.
		\end{equation}
		
		$\bullet$ Using \eqref{ineVarpiReste_res}, there exist a constant $C>0$ independent from $j_0$ and $n$ such that
		\begin{equation}\label{lemTc2In2}
			n\int_{\Gamma_{res,p}}|\tau|^{2\mu} \exp(n \Re(\tau)-j_0 (\Re(\varpi(\tau)) -|\xi(\tau)\tau^{2\mu+1}|))|d\tau|\leq Cn\exp\left(-n\frac{\eta}{2}\right).
		\end{equation}
		
		$\bullet$ In cases \textbf{B} and \textbf{C}, using \eqref{ineVarpiReste_casB} or \eqref{ineVarpiReste_casC} depending on the case, there exist two constants $C,c>0$ such that
		\begin{equation}\label{lemTc2In3}
			n\int_{\Gamma_{p}}|\tau|^{2\mu} \exp(n \Re(\tau)-j_0 (\Re(\varpi(\tau)) -|\xi(\tau)\tau^{2\mu+1}|))|d\tau|\leq Cn\exp\left(-cn\right).
		\end{equation}
		
		$\bullet$ In case \textbf{A}, using \eqref{ineVarpiReste_casA}, we have
		\begin{multline*}
			n\int_{\Gamma_{p}}|\tau|^{2\mu} \exp(n \Re(\tau)-j_0 (\Re(\varpi(\tau)) -|\xi(\tau)\tau^{2\mu+1}|))|d\tau| \\ \leq n\int_{\Gamma_{p}}|\tau|^{2\mu}\exp\left(-nc_\star\Im(\tau)^{2\mu}\right)|d\tau| \exp\left(- \frac{n}{\alpha} (2\mu-1)\gamma\left(\frac{|\zeta|}{\gamma}\right)^\frac{2\mu}{2\mu-1}\right).
		\end{multline*}
		But, the inequality \eqref{ineg_gamma} and the fact that $\rho\left(\frac{\zeta}{\gamma}\right)=\tau_p$ imply
		$$-\frac{n}{\alpha} (2\mu-1)\gamma\left(\frac{|\zeta|}{\gamma}\right)^\frac{2\mu}{2\mu-1}\leq -\frac{2\mu-1}{2}A_R n|\tau_p|^{2\mu}.$$
		If we introduce $c>0$ small enough, then
		$$n\int_{\Gamma_{p}}|\tau|^{2\mu} \exp(n \Re(\tau)-j_0 (\Re(\varpi(\tau)) -|\xi(\tau)\tau^{2\mu+1}|))|d\tau| \leq n\int_{\Gamma_{p}}|\tau|^{2\mu}\exp\left(-nc_\star\Im(\tau)^{2\mu}\right)|d\tau| \exp\left(- cn|\tau_p|^{2\mu}\right).$$
		Using the parametrization \eqref{param} and the inequality \eqref{hp}, we have
		$$\int_{\Gamma_{p}} |\tau|^{2\mu} e^{-nc_\star\Im(\tau)^{2\mu}}|d\tau| \lesssim \int_{-\ell_{p}}^{\ell_{p}}(|\tau_p|^{2\mu}+|t|^{2\mu})e^{-nc_\star t^{2\mu}}dt.$$
		The change of variables $u=n^\frac{1}{2\mu}t$ and the fact that the function $\displaystyle x\geq0\mapsto x^{2\mu}\exp\left(-\frac{c}{2}x^{2\mu}\right)$ is bounded imply
		$$\lc \begin{array}{c}
			\displaystyle \int_{-\ell_{p}}^{\ell_{p}}|t|^{2\mu}e^{-nc_\star t^{2\mu}}dt\lesssim \frac{1}{n^{1+\frac{1}{2\mu}}},\\
			\displaystyle \int_{-\ell_{p}}^{\ell_{p}}|\tau_p|^{2\mu}e^{-nc_\star t^{2\mu}}dt\lesssim \frac{1}{n^{1+\frac{1}{2\mu}}}\exp\left(\frac{c}{2}n|\tau_p|^{2\mu}\right).
		\end{array}\right.$$
		Thus,
		$$n\int_{\Gamma_{p}}|\tau|^{2\mu} \exp(n \Re(\tau)-j_0 (\Re(\varpi(\tau)) -|\xi(\tau)\tau^{2\mu+1}|))|d\tau| \lesssim \frac{1}{n^\frac{1}{2 \mu}} \exp\left(- \frac{c}{2}n|\tau_p|^{2\mu}\right). $$
		Lastly, the inequality \eqref{ineg_gamma} implies that we have a constant $c>0$ independent from $j_0$ and $n$ such that 
		$$\frac{c}{2}n|\tau_p|^{2\mu} \geq \tilde{c}\left(\frac{|n\alpha+j_0|}{n^\frac{1}{2\mu}}\right)^\frac{2\mu}{2\mu-1}$$
		so,
		\begin{equation}\label{lemTc2In4}
			n\int_{\Gamma_{p}}|\tau|^{2\mu} \exp(n \Re(\tau)-j_0 (\Re(\varpi(\tau)) -|\xi(\tau)\tau^{2\mu+1}|))|d\tau|\lesssim \frac{1}{n^\frac{1}{2\mu}} \exp\left(-\tilde{c} \left(\frac{|n\alpha+j_0|}{n^\frac{1}{2\mu}}\right)^\frac{2\mu}{2\mu-1}\right).
		\end{equation}
		
		Combining \eqref{lemTc2In1}-\eqref{lemTc2In4}, we conclude the proof of Proposition \ref{lem:est_Tc2}.
	\end{proof}
	
	\subsubsection{Calculations around $T^c_{princ}$}\label{subsec:est_Tcprinc}
	
	There just remains to study the term $T^c_{princ}$ defined by \eqref{defTcprinc}. The end goal of this section is to prove the following proposition.
	\begin{prop}\label{prop:est_Tcprinc}
		There exist two positive constants $C,c$ such that for all $n,j_0\in\N^*$ which verify $-\frac{n\alpha}{2}\leq j_0\leq np$, we have 
		$$\left|\int_{\Gamma_{in}} \frac{\exp(n\tau-j_0\varphi(\tau))}{\tau} d\tau -2i\pi E_{2\mu}^\beta\left(\frac{j_0+n\alpha}{n^\frac{1}{2\mu}}\right)\right|\leq \frac{C}{n^\frac{1}{2\mu}}\exp\left(-c\left(\frac{|j_0+n\alpha|}{n^\frac{1}{2\mu}}\right)^\frac{2\mu}{2\mu-1}\right).$$
	\end{prop}
	
	By utilizing \eqref{decompoTc} along with Propositions \ref{prop:est_Tc1}, \ref{lem:est_Tc2}, and \ref{prop:est_Tcprinc}, we complete the demonstration of Proposition \ref{prop:est_Tc_J0_Close}. Subsequently, this concludes the proof of Theorem \ref{th:Green}.
	
	Thus, there just remains to prove Proposition \ref{prop:est_Tcprinc}. The main idea of the proof is to change the integration path on the term
	$$\int_{\Gamma_{in}}\frac{e^{n\tau}e^{-j_0\varphi(\tau)}}{\tau}d\tau.$$
	
	\begin{proof}
	We fix a constant $s>0$ such that for all $n,j_0\in\N^*$ which verify $-\frac{n\alpha}{2}\leq j_0\leq np$, we have
	\begin{equation}\label{condS}
		\left(-\frac{n\alpha}{j_0}-1\right)s+A_Rs^{2\mu}\leq \frac{A_I}{2}r_\varepsilon(\eta)^{2\mu}.
	\end{equation}
	We introduce the paths $\Gamma_s$, $\Gamma_{comp}^+$, $\Gamma_{comp}^-$, $\Gamma_{s,\infty}^+$, $\Gamma_{s,\infty}^-$ and $\Gamma_{s,\infty}$ represented on Figure \ref{chem2} and that are defined as
	\begin{align*}
		\Gamma_{s} &:= \lc s + i t, \quad t \in[-r_\varepsilon(\eta),r_\varepsilon(\eta)]\rc,&\Gamma_{s,\infty}^+ &:= \lc s + i t, \quad t \in[r_\varepsilon(\eta),+\infty[\rc,\\
		\Gamma^+_{comp} &:= \lc t + i r_\varepsilon(\eta), \quad t \in[-\eta,s]\rc,&\Gamma_{s,\infty}^- &:= \lc s +i t, \quad t \in]-\infty,-r_\varepsilon(\eta)]\rc,\\
		\Gamma^-_{comp} &:= \lc t -i r_\varepsilon(\eta), \quad t \in[-\eta,s]\rc,&\Gamma_{s,\infty} &:=\Gamma_{s,\infty}^-\cup\Gamma_s\cup\Gamma_{s,\infty}^+.
	\end{align*}
	
	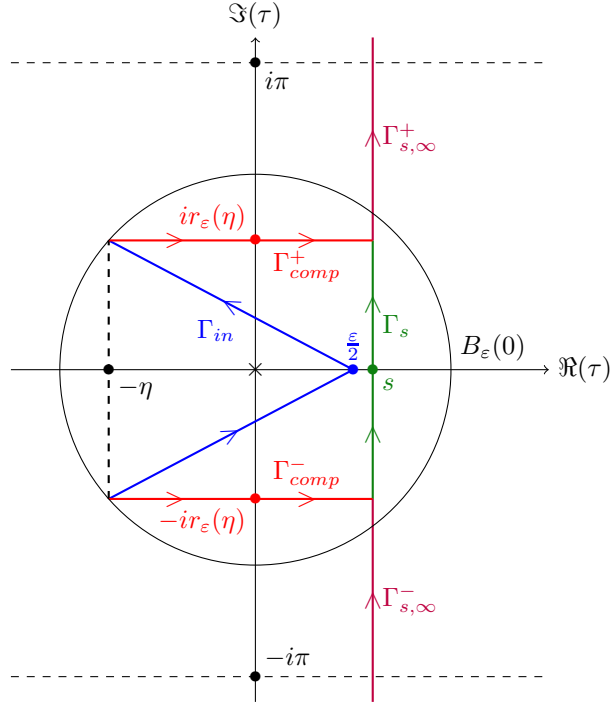
\begin{figure}
		\begin{center}
			\begin{tikzpicture}[scale=1.3]
				\draw[->] (-2.5,0) -- (3,0) node[right] {$\Re(\tau)$};
				\draw[->] (0,-3.4) -- (0,3.4) node[above] {$\Im(\tau)$};
				\draw[dashed] (-2.5,pi) -- (3,pi);
				\draw[dashed] (-2.5,-pi) -- (3,-pi);
				\draw (0,pi) node {$\bullet$} node[below right] {$i\pi$};
				\draw (0,-pi) node {$\bullet$} node[above right] {$-i\pi$};
				\draw (-1.5,0) node {$\bullet$} node[below right] {$-\eta$};
				\draw (2,0) node[above right] {$B_\varepsilon(0)$};
				
				\draw[blue,thick] (-1.5,{-sqrt(4-1.5^2)}) -- (1,0) node[midway,sloped] {$>$};
				\draw[blue,thick] (1,0) -- (-1.5,{sqrt(4-1.5^2)}) node[midway,sloped] {$<$};
				\draw[dashed,thick] (-1.5,{-sqrt(4-1.5^2)}) -- (-1.5,{sqrt(4-1.5^2)});
				\draw[blue] (1,0) node {$\bullet$} node[above] {$\frac{\varepsilon}{2}$}; 
				\draw[blue] (-0.4,0.4) node {$\Gamma_{in}$};
				
				\draw[dartmouthgreen,thick] (1.2,{-sqrt(4-1.5^2)}) -- (1.2,{sqrt(4-1.5^2)}) node[near end,sloped] {$>$} node[near start,sloped] {$>$} node[near end,below right] {$\Gamma_{s}$};
				\draw[dartmouthgreen] (1.2,0) node {$\bullet$} node[below right] {$s$};
				
				\draw[red,thick] (-1.5,{-sqrt(4-1.5^2)}) -- (1.2,{-sqrt(4-1.5^2)}) node[near start,sloped] {$>$} node[near end,sloped] {$>$} node[near end,above] {$\Gamma^-_{comp}$};
				\draw[red,thick] (-1.5,{sqrt(4-1.5^2)}) -- (1.2,{sqrt(4-1.5^2)}) node[near start,sloped] {$>$} node[near end,sloped] {$>$} node[near end,below] {$\Gamma^+_{comp}$};
				
				\draw[red] (0,{sqrt(4-1.5^2)}) node {$\bullet$} node[above left] {$ir_\varepsilon(\eta)$};
				\draw[red] (0,{-sqrt(4-1.5^2)}) node {$\bullet$} node[below left] {$-ir_\varepsilon(\eta)$};
				
				\draw[purple,thick] (1.2,{sqrt(4-1.5^2)}) -- (1.2,3.4) node[midway,sloped] {$>$} node[midway, right] {$\Gamma_{s,\infty}^+$};
				\draw[purple,thick] (1.2,-3.4) -- (1.2,{-sqrt(4-1.5^2)}) node[midway,sloped] {$>$} node[midway, right] {$\Gamma_{s,\infty}^-$};
								
				\draw (0,0) node {$\times$} circle (2);
			\end{tikzpicture}
			\caption{Representation of the paths $\Gamma_s$ (in green), $\Gamma_{comp}^+$, $\Gamma_{comp}^-$ (both in red), $\Gamma_{s,\infty}^+$, $\Gamma_{s,\infty}^-$ (both in purple) and $\Gamma_{s,\infty}:=\Gamma_{s,\infty}^-\cup\Gamma_s\cup\Gamma_{s,\infty}^+$.}
			\label{chem2}
		\end{center}
	\end{figure}

		The proof of Proposition \ref{prop:est_Tcprinc} is separated in different steps where we will use the different paths we introduced.
	
		\textbf{$\bullet$ Step 1:} In this step, we start by proving that there exist two positive constants $C,c$ such that for all $n,j_0\in\N^*$ which verify $-\frac{n\alpha}{2}\leq j_0\leq np$, we have 
		\begin{equation}\label{inDec}
			\left|\int_{\Gamma_{in}} \frac{\exp(n\tau-j_0\varphi(\tau))}{\tau} d\tau - \int_{\Gamma_{s}} \frac{\exp(n\tau-j_0\varphi(\tau))}{\tau} d\tau \right|\leq Ce^{-cn}.
		\end{equation}
		
		Cauchy's formula implies that
		\begin{equation}\label{inDec1}
			\left|\int_{\Gamma_{in}}\frac{e^{n\tau-j_0\varphi(\tau)}}{\tau}d\tau -\int_{\Gamma_{s}}\frac{e^{n\tau-j_0\varphi(\tau)}}{\tau}d\tau\right|\leq \left|\int_{\Gamma^+_{comp}}\frac{e^{n\tau-j_0\varphi(\tau)}}{\tau}d\tau\right|+\left|\int_{\Gamma^-_{comp}}\frac{e^{n\tau-j_0\varphi(\tau)}}{\tau}d\tau\right|.
		\end{equation}
		We need to find estimates for the two terms on the right-hand side. Both terms will be bounded similarly so we will focus on the first one. First, we observe that
		$$\left|\int_{\Gamma^+_{comp}}\frac{e^{n\tau-j_0\varphi(\tau)}}{\tau}d\tau\right|\leq \frac{1}{r_\varepsilon(\eta)}\int_{-\eta}^s\exp\left(nt -j_0 \Re(\varphi(t + i r_\varepsilon(\eta)))\right)dt.$$
		Using \eqref{estVarphi}, we have for $t\in[-\eta,s]$
		$$nt -j_0 \Re(\varphi(t + i r_\varepsilon(\eta)))\leq \left(n+\frac{j_0}{\alpha}\right)t -\frac{j_0}{\alpha}A_R t^{2\mu} + \frac{j_0}{\alpha} A_I r_\varepsilon(\eta)^{2\mu} = -\frac{j_0}{\alpha}\left(\left(-\frac{n\alpha}{j_0}-1\right)t+A_Rt^{2\mu}-A_Ir_\varepsilon(\eta)^{2\mu}\right).$$
		Since the function
		$$t\in[-\eta,s]\mapsto\left(-\frac{n\alpha}{j_0}-1\right)t+A_Rt^{2\mu}-A_Ir_\varepsilon(\eta)^{2\mu}$$
		is convex, it attains its maximum for $t\in\lc-\eta,s\rc$. Thus, the conditions \eqref{condEta2} and \eqref{condS} on $\eta$ and $s$ imply that for $t\in[-\eta,s]$
		$$nt -j_0 \Re(\varphi(t + i r_\varepsilon(\eta)))\leq\frac{j_0}{2\alpha}A_Ir_\varepsilon(\eta)^{2\mu}.$$
		Thus, recalling that $\alpha$ is negative and that $j_0\in\left[-\frac{n\alpha}{2},np\right]$, we have that
		$$\left|\int_{\Gamma^+_{comp}}\frac{e^{n\tau-j_0\varphi(\tau)}}{\tau}d\tau\right|\leq \frac{s+\eta}{r_\varepsilon(\eta)}\exp\left(-\frac{n}{4}A_Ir_\varepsilon(\eta)^{2\mu}\right).$$
		Using a similar proof to bound the second term in the right-hand side of \eqref{inDec1}, we can conclude the proof of \eqref{inDec}.
		
		\textbf{$\bullet$ Step 2:} In this second step, we now prove that there exist two positive constants $C,c$ such that for all $n,j_0\in\N^*$ which verify $-\frac{n\alpha}{2}\leq j_0\leq np$, we have 
		\begin{equation}\label{inProl}
			\left|\int_{\Gamma_{s}} \frac{\exp(n\tau-j_0\varphi(\tau))}{\tau} d\tau - \int_{\Gamma_{s,\infty}} \frac{\exp(n\tau-j_0\varphi(\tau))}{\tau} d\tau \right|\leq Ce^{-cn}.
		\end{equation}
		
		We have that
		\begin{equation}\label{inProl1}
			\left|\int_{\Gamma_{s}} \frac{\exp(n\tau-j_0\varphi(\tau))}{\tau} d\tau - \int_{\Gamma_{s,\infty}} \frac{\exp(n\tau-j_0\varphi(\tau))}{\tau} d\tau\right|\leq \left|\int_{\Gamma^+_{s,\infty}}\frac{e^{n\tau-j_0\varphi(\tau)}}{\tau}d\tau\right|+\left|\int_{\Gamma^-_{s,\infty}}\frac{e^{n\tau-j_0\varphi(\tau)}}{\tau}d\tau\right|.
		\end{equation}
		We need to find estimates for the two terms on the right-hand side. Both terms will be bounded similarly so we will focus on the first one. First, we observe that
		$$\left|\int_{\Gamma^+_{s,\infty}}\frac{e^{n\tau-j_0\varphi(\tau)}}{\tau}d\tau\right|\leq \frac{1}{r_\varepsilon(\eta)}\int_{r_\varepsilon(\eta)}^{+\infty}\exp\left(ns -j_0 \Re(\varphi(s + i t))\right)dt.$$
		Using \eqref{estVarphi} and \eqref{condS}, we have for $t\in[r_\varepsilon(\eta),+\infty[$ and $j_0\in\left[-\frac{n\alpha}{2},np\right]$
		$$ns -j_0 \Re(\varphi(s + it))\leq \left(n+\frac{j_0}{\alpha}\right)s -\frac{j_0}{\alpha}A_R s^{2\mu} + \frac{j_0}{\alpha} A_I t^{2\mu} \leq \frac{j_0}{\alpha} A_I \left(t^{2\mu} - \frac{r_\varepsilon(\eta)^{2\mu}}{2}\right)\leq-\frac{n}{2} A_I \left(t^{2\mu} - \frac{r_\varepsilon(\eta)^{2\mu}}{2}\right) .$$
		Thus,
		$$ns -j_0 \Re(\varphi(s + it))\leq -\frac{n}{4} A_Ir_\varepsilon(\eta)^{2\mu} -\frac{n}{2} A_I \left(t^{2\mu} - r_\varepsilon(\eta)^{2\mu}\right)\leq -\frac{n}{4} A_Ir_\varepsilon(\eta)^{2\mu} - \frac{1}{2} A_I \left(t^{2\mu} - r_\varepsilon(\eta)^{2\mu}\right).$$
		We can then conclude that 
		$$\left|\int_{\Gamma^+_{s,\infty}}\frac{e^{n\tau-j_0\varphi(\tau)}}{\tau}d\tau\right|\leq \frac{1}{r_\varepsilon(\eta)}\exp\left(-\frac{n}{4} A_Ir_\varepsilon(\eta)^{2\mu}\right)\underset{<+\infty}{\underbrace{\int_{r_\varepsilon(\eta)}^{+\infty}\exp\left(- \frac{1}{2} A_I \left(t^{2\mu} - r_\varepsilon(\eta)^{2\mu}\right)\right)dt}}.$$
		Using a similar proof to bound the second term in the right-hand side of \eqref{inProl1}, we can conclude the proof of \eqref{inProl}.
		
		\textbf{$\bullet$ Step 3:} We introduce the functions
		\begin{align}
			\forall u\in\R, \forall x\in\R, \forall s\in\R,\quad &g(u,x,s):=\exp\left(i(u+is)x - \beta (u+is)^{2\mu}\right),\label{defg}\\
			\forall x\in\R, \forall s\in]0,+\infty[,\quad & \Fc(x,s) :=\displaystyle\int_{-\infty}^{+\infty} \frac{g(u,x,s)}{i(u+is)}du.\label{defF}
		\end{align}
		We can prove that the function $\Fc$ verifies
		\begin{equation}\label{egF}
			\forall s\in]0,+\infty[,\forall x\in\R, \quad -\Fc(x,s) = 2\pi E_{2\mu}^{\beta}(x).
		\end{equation}
		
		For the sake of completeness, we give a proof of \eqref{egF} in the Appendix (Section \ref{secAppendix}).
		
		We observe that for $n,j_0\in\N^*$, if we define $\tilde{s}_{j_0}:=\frac{s}{-\alpha}\left(-\frac{j_0}{\alpha}\right)^\frac{1}{2\mu}$,  we have using the change of variables $\left(-\frac{j_0}{\alpha}\right)^\frac{1}{2\mu}t=\alpha u$ that
		\begin{align}\label{egInfty}
			\begin{split}
				\int_{\Gamma_{s,\infty}} \frac{\exp(n\tau-j_0\varphi(\tau))}{\tau} d\tau 
				&= i\int_{-\infty}^{+\infty} \frac{\exp\left(\left(n+\frac{j_0}{\alpha}\right)(s+it)+(-1)^{\mu+1}\frac{\beta}{\alpha^{2\mu}}\left(-\frac{j_0}{\alpha}\right)(s+it)^{2\mu}\right)}{s+it}dt\\
				& =  -i\int_{-\infty}^{+\infty} \frac{\exp\left(i(u+i\tilde{s}_{j_0})\frac{(n\alpha+j_0)}{\left(-\frac{j_0}{\alpha}\right)^\frac{1}{2\mu}} - \beta (u+i\tilde{s}_{j_0})^{2\mu}\right)}{i(u+i\tilde{s}_{j_0})}du\\
				& =  -i \Fc\left(\frac{(n\alpha+j_0)}{\left(-\frac{j_0}{\alpha}\right)^\frac{1}{2\mu}} ,\tilde{s}_{j_0}\right).
			\end{split}	
		\end{align}
		The equalities \eqref{egF} and \eqref{egInfty} imply that
		\begin{equation}\label{egInfty2}
			\int_{\Gamma_{s,\infty}} \frac{\exp(n\tau-j_0\varphi(\tau))}{\tau} d\tau  = 2i\pi E_{2\pi}^\beta\left(\frac{j_0+n\alpha}{\left(-\frac{j_0}{\alpha}\right)^\frac{1}{2\mu}}\right).
		\end{equation}
		
		To end the proof of Proposition \ref{prop:est_Tcprinc}, we will prove the existence of two positive constants $C,c$ such that for all $n,j_0\in\N\backslash\lc0\rc$ which verify $j_0\in\left[-\frac{n\alpha}{2},np\right]$, we have
		\begin{equation}\label{inE}
			\left|E_{2\mu}^\beta\left(\frac{j_0+n\alpha}{\left(-\frac{j_0}{\alpha}\right)^\frac{1}{2\mu}}\right)-E_{2\mu}^\beta\left(\frac{j_0+n\alpha}{n^\frac{1}{2\mu}}\right)\right|\leq \frac{C}{n^\frac{1}{2\mu}}\exp\left(-c\left(\frac{|j_0+n\alpha|}{n^\frac{1}{2\mu}}\right)^\frac{2\mu}{2\mu-1}\right).
		\end{equation}
		
		We recall that 
		$$\forall x\in\R, \quad {E_{2\mu}^\beta}^\prime(x)= - H_{2\mu}^\beta(x).$$
		Because of the mean value inequality, the fact that $j_0\in\left[-\frac{n\alpha}{2},np\right]$ and \eqref{inH}, there exists a positive constant $c>0$ such that
		\begin{align*}
			\left|E_{2\mu}^\beta\left(\frac{j_0+n\alpha}{\left(-\frac{j_0}{\alpha}\right)^\frac{1}{2\mu}}\right)-E_{2\mu}^\beta\left(\frac{j_0+n\alpha}{n^\frac{1}{2\mu}}\right)\right|& \leq |j_0+n\alpha|\left|\frac{1}{\left(-\frac{j_0}{\alpha}\right)^\frac{1}{2\mu}}-\frac{1}{n^\frac{1}{2\mu}}\right|\sup_{t\in\left[\frac{j_0+n\alpha}{\left(-\frac{j_0}{\alpha}\right)^\frac{1}{2\mu}},\frac{j_0+n\alpha}{n^\frac{1}{2\mu}}\right]}\left|{H_{2\mu}^{\beta}}(t)\right|\\
			& \lesssim |j_0+n\alpha|\left|\frac{1}{\left(-\frac{j_0}{\alpha}\right)^\frac{1}{2\mu}}-\frac{1}{n^\frac{1}{2\mu}}\right|\exp\left(-c\left(\frac{|j_0+n\alpha|}{n^\frac{1}{2\mu}}\right)^\frac{2\mu}{2\mu-1}\right).
		\end{align*}
		Furthermore, using once again the mean value inequality and the fact that $j_0\in\left[-\frac{n\alpha}{2},np\right]$, we have
		$$\left|\frac{1}{\left(-\frac{j_0}{\alpha}\right)^\frac{1}{2\mu}}-\frac{1}{n^\frac{1}{2\mu}}\right|\lesssim |j_0+n\alpha| \sup_{t\in[j_0,-n\alpha]}\frac{1}{t^{1+\frac{1}{2\mu}}}\lesssim \frac{|j_0+n\alpha|}{n^{1+\frac{1}{2\mu}}}.$$
		Therefore,
		$$\left|E_{2\mu}^\beta\left(\frac{j_0+n\alpha}{\left(-\frac{j_0}{\alpha}\right)^\frac{1}{2\mu}}\right)-E_{2\mu}^\beta\left(\frac{j_0+n\alpha}{n^\frac{1}{2\mu}}\right)\right|\lesssim \frac{1}{n^{1-\frac{1}{2\mu}}} \left(\frac{|j_0+n\alpha|}{n^\frac{1}{2\mu}}\right)^2 \exp\left(-c\left(\frac{|j_0+n\alpha|}{n^\frac{1}{2\mu}}\right)^\frac{2\mu}{2\mu-1}\right).$$
		Since the function $x\mapsto x^2\exp\left(-\frac{c}{2}x\right)$ is bounded, we conclude that
		$$\left|E_{2\mu}^\beta\left(\frac{j_0+n\alpha}{\left(-\frac{j_0}{\alpha}\right)^\frac{1}{2\mu}}\right)-E_{2\mu}^\beta\left(\frac{j_0+n\alpha}{n^\frac{1}{2\mu}}\right)\right|\lesssim \frac{1}{n^{1-\frac{1}{2\mu}}} \exp\left(-\frac{c}{2}\left(\frac{|j_0+n\alpha|}{n^\frac{1}{2\mu}}\right)^\frac{2\mu}{2\mu-1}\right).$$
		Since $1-\frac{1}{2\mu}\geq \frac{1}{2\mu}$, we easily conclude the proof of \eqref{inE}.
		
		Combining \eqref{inDec}, \eqref{inProl}, \eqref{egInfty2} and \eqref{inE}, we can end the proof of Proposition \ref{prop:est_Tcprinc}.
	\end{proof}
	
	\textbf{Acknowledgement:} The author is largely indebted to Jean-François Coulombel and Grégory Faye for their advice that led to this result, as well as their proofreading and corrections.
	
	\section{Appendix}\label{secAppendix}
	
	\subsection{Proof of equality \eqref{egF}}
	
	We recall that \eqref{egF} states that
	\begin{equation*}
		\forall s\in]0,+\infty[,\forall x\in\R, \quad -\Fc(x,s) = 2\pi E_{2\mu}^{\beta}(x).
	\end{equation*}
	
	\begin{proof}
		The starting point of the proof will be to prove sharp estimates on the function $g$ defined by \eqref{defg}. We observe that
		$$\forall u,x,s\in\R, \quad |g(u,x,s)|\leq \exp(-sx)\exp\left(-\Re(\beta(u+is)^{2\mu})\right).$$
		Using Young's inequality, we prove that there exists a constant $c>0$ such that
		$$\forall u,s\in\R, \quad \Re(\beta(u+is)^{2\mu})\geq \frac{\Re(\beta)}{2}u^{2\mu} - c s^{2\mu}.$$
		Thus, we have
		\begin{equation}\label{ing}
			\forall u,x,s\in\R,\quad |g(u,x,s)|\leq \exp(-sx+cs^{2\mu})\exp\left(-\frac{\Re(\beta)}{2}u^{2\mu}\right).
		\end{equation}
		
		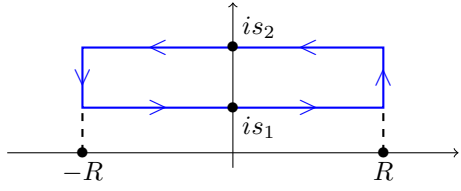
\begin{figure}
			\centering
			\begin{tikzpicture}[scale=2]
				\draw[->] (-1.5,0) -- (1.5,0);
				\draw[->] (0,-0.1) -- (0,1);
				\draw[thick,color=blue] (-1,0.3) -- (1,0.3) node[near start, sloped] {$>$} node[near end, sloped] {$>$} -- (1,0.7)node[midway, sloped] {$>$} -- (-1,0.7) node[near start, sloped] {$<$} node[near end, sloped] {$<$} -- cycle node[midway, sloped] {$>$};
				\draw[thick,dashed] (-1,0) -- (-1,0.3);
				\draw[thick,dashed] (1,0) -- (1,0.3);
				\draw (-1,0) node {$\bullet$};
				\draw (-1,0) node[below] {$-R$};
				\draw (1,0) node {$\bullet$};
				\draw (1,0) node[below] {$R$};
				\draw (0,0.7) node {$\bullet$} node[above right] {$is_2$};
				\draw (0,0.3) node {$\bullet$} node[below right] {$is_1$};
			\end{tikzpicture}
			\caption{Integrating path to prove \eqref{egFs}.}
			\label{int_path}
		\end{figure}
		
		We observe that for all $s\in]0,+\infty[$, the function $\Fc(\cdot,s)$ defined by \eqref{defF} is in the class $\Cc^1$ and 
		\begin{equation}\label{egderF}
			\forall x\in\R,\forall s\in]0,+\infty[,\quad \frac{\partial \Fc}{\partial x} (x,s) = 2\pi H_{2\mu}^\beta(x).
		\end{equation}
		Integrating the function $z\mapsto \frac{\exp(izx-\beta z^{2\mu})}{iz}$ on the rectangle depicted in the right-side of Figure \ref{int_path}, using the Cauchy formula as well as \eqref{ing} and passing to the limit $R\rightarrow +\infty$, we prove that
		\begin{equation}\label{egFs}
			\forall s_1,s_2\in]0,+\infty[, \quad \Fc(\cdot,s_1)=\Fc(\cdot,s_2).
		\end{equation}
		Finally, using \eqref{ing}, there exists $C>0$ independent from $x$ and $s$ such that 
		\begin{equation}\label{inF}
			\forall x\in\R,\forall s\in]0,+\infty[, \quad \left|\Fc(x,s)\right| \leq C \frac{e^{-s x+cs^{2\mu}}}{s}.
		\end{equation}
		For $x>0$, optimizing $e^{-s x+cs^{2\mu}}$ with respect to $s$ drives us to choose $s=\left(\frac{x}{2\mu c}\right)^\frac{1}{2\mu-1}$ in \eqref{inF}. Using \eqref{egFs}, we can prove that there exist two constants $C,c>0$ such that
		$$\forall x\in]0,+\infty[,\forall s\in]0,+\infty[, \quad \left|\Fc(x,s)\right| \leq \frac{C}{x^\frac{1}{2\mu-1}} \exp(-c|x|^\frac{2\mu}{2\mu-1}).$$
		Thus, 
		\begin{equation}\label{limF}
			\forall s\in]0,+\infty[,\quad \lim_{x\rightarrow+\infty}\Fc(x,s)=0.
		\end{equation}
		Using \eqref{egderF} and \eqref{limF}, we easily conclude the proof of \eqref{egF}.
	\end{proof}
	
	\bibliographystyle{alpha}
	\bibliography{references_IBVP_V9}{}
\end{document}